\documentclass{amsart}

\usepackage{amssymb}
\usepackage{amsrefs}
\usepackage{fullpage}
\usepackage{tikz}
\usetikzlibrary{matrix,arrows,decorations.pathmorphing, cd}

\newcommand{\Bilin}     {\operatorname{Bilin}}
\newcommand{\End}       {\operatorname{End}}
\newcommand{\Ext}       {\operatorname{Ext}}
\newcommand{\Hom}       {\operatorname{Hom}}
\newcommand{\Map}       {\operatorname{Map}}
\newcommand{\Tor}       {\operatorname{Tor}}

\newcommand{\cok}       {\operatorname{cok}}
\newcommand{\image}     {\operatorname{image}}
\newcommand{\img}       {\operatorname{img}}
\newcommand{\supp}      {\operatorname{supp}}
\newcommand{\tors}      {\operatorname{tors}}

\newcommand{\Z}         {\mathbb{Z}}
\newcommand{\Zp}        {\mathbb{Z}_p}
\newcommand{\Zpl}       {\mathbb{Z}_{(p)}}
\newcommand{\Zpi}       {\mathbb{Z}/p^\infty}
\newcommand{\R}         {\mathbb{R}}
\newcommand{\N}         {\mathbb{N}}
\newcommand{\Q}         {{\mathbb{Q}}}
\newcommand{\QZ}        {{\mathbb{Q}/\mathbb{Z}}}

\newcommand{\CA}        {{\mathcal{A}}}
\newcommand{\CC}        {{\mathcal{C}}}
\newcommand{\CD}        {{\mathcal{D}}}
\newcommand{\CE}        {{\mathcal{E}}}
\newcommand{\CJ}        {{\mathcal{J}}}
\newcommand{\CS}        {{\mathcal{S}}}

\newcommand{\ov}[1]     {\overline{#1}}
\newcommand{\st}        {\;|\;}
\newcommand{\tm}        {\times}
\newcommand{\ot}        {\otimes}
\newcommand{\sm}        {\setminus}
\newcommand{\sse}       {\subseteq}
\newcommand{\pinv}      {{\textstyle\frac{1}{p}}}
\newcommand{\ip}[1]     {\langle #1\rangle}
\newcommand{\ib}        {\overline{\imath}}
\newcommand{\jb}        {\overline{\jmath}}
\newcommand{\kb}        {\overline{k}}
\newcommand{\psb}[1]    {[\![#1]\!]}
\newcommand{\nl}        {\ \par}

\newcommand{\xla}       {\xleftarrow}
\newcommand{\xra}       {\xrightarrow}
\newcommand{\era}       {\twoheadrightarrow}
\newcommand{\mra}       {\rightarrowtail}
\newcommand{\mar}       {\ar@{ >->}}
\newcommand{\ear}       {\ar@{->>}}

\newcommand{\al}        {\alpha}
\newcommand{\bt}        {\beta} 
\newcommand{\gm}        {\gamma}
\newcommand{\dl}        {\delta}
\newcommand{\ep}        {\epsilon}
\newcommand{\zt}        {\zeta}
\newcommand{\lm}        {\lambda}
\newcommand{\sg}        {\sigma}
\newcommand{\om}        {\omega}

\renewcommand{\:}       {\colon}

\newcommand{\bcf}[2]{\left(\begin{array}{c}{#1}\\{#2}\end{array}\right)}
\newcommand{\invlim} {\operatornamewithlimits{\underset{\longleftarrow}{lim}}}
\newcommand{\colim}  {\operatornamewithlimits{\underset{\longrightarrow}{lim}}}

\definecolor{olivegreen}{cmyk}{0.64,0,0.95,0.40}
\definecolor{rawsienna}{cmyk}{0,0.72,1,0.45}

\newtheorem{theorem}{Theorem}
\numberwithin{theorem}{section}

\newtheorem{lemma}[theorem]{Lemma}
\newtheorem{proposition}[theorem]{Proposition}
\newtheorem{corollary}[theorem]{Corollary}

\theoremstyle{definition}

\newtheorem{definition}[theorem]{Definition}
\newtheorem{example}[theorem]{Example}

\newtheorem{remark}[theorem]{Remark}

\begin{document}
\title{Algebraic theory of abelian groups}
\author{N.~P.~Strickland}
\date{\today}
\bibliographystyle{abbrv}

\maketitle 
\tableofcontents

\section{Introduction}
\label{sec-intro}

This document aims to give a self-contained account of the parts of
abelian group theory that are most relevant for algebraic topology.
It is almost purely expository, although there are some slightly
unusual features in the treatment of tensor products, torsion products
and $\Ext$ groups.  The book~\cite{boka:hlc} is a good reference for
Sections~\ref{sec-towers} and~\ref{sec-completion}.  Earlier sections
are more standard and can be found in very many sources.

\section{Exactness and splittings}
\label{sec-exact}

\begin{definition}\label{defn-exact}
 Consider a sequence $A_0\xra{f_0}A_1\xra{}\dotsb\xra{f_{r-1}}A_r$ of
 abelian groups and homomorphisms.  We say that the sequence is
 \emph{exact} at $A_i$ if $\image(f_{i-1})=\ker(f_i)\leq A_i$ (which
 implies that $f_i\circ f_{i-1}=0$).  We say that the whole sequence
 is exact if it is exact at $A_i$ for $0<i<r$.

 Next, we say that a sequence $A\xra{f}B\xra{g}C$ is
 \emph{short exact} if it is exact, and also $f$ is injective and $g$
 is surjective.
\end{definition}

\begin{remark}\label{rem-exact-omni}
 One can easily check the following facts.
 \begin{itemize}
  \item[(a)] A sequence $A\xra{f}B\xra{0}C$ is exact iff $f$ is
   surjective.  In particular, a sequence $A\xra{f}B\xra{}0$ is exact
   iff $f$ is surjective.
  \item[(b)] A sequence $A\xra{0}B\xra{g}C$ is exact iff $g$ is
   injective.  In particular, a sequence $0\xra{}B\xra{g}C$ is exact
   iff $g$ is injective.
  \item[(c)] A sequence $A\xra{0}B\xra{g}C\xra{0}D$ is exact iff $g$
   is an isomorphism.
  \item[(d)] A sequence $0\xra{}A\xra{f}B\xra{g}C\xra{}0$ is exact iff
   $A\xra{f}B\xra{g}C$ is short exact.
  \item[(e)] Suppose we have an exact sequence
   \[ A \xra{f} B \xra{g} C \xra{h} D \xra{k} E. \]
   Then $g$ induces a map from $\cok(f)=B/f(A)$ to $C$, and $h$ can be
   regarded as a map from $C$ to $\ker(k)$, and the resulting sequence 
   \[ \cok(f) \xra{g} C \xra{h} \ker(k) \]
   is short exact.
  \item[(f)] If $A\xra{f}B\xra{g}C$ is short exact, then $f$ induces
   an isomorphism $A\to f(A)$ and $g$ induces an isomorphism
   $B/f(A)\to C$.  Thus, if $A$, $B$ and $C$ are finite we have
   $|B|=|f(A)|.|B/f(A)|=|A||C|$.  Similarly, if $A$ and $C$ are free
   abelian groups of ranks $n$ and $m$, then $B$ is a free abelian
   group of rank $n+m$.
 \end{itemize}
\end{remark}

\begin{proposition}[The five lemma]\label{lem-five}
 Suppose we have a commutative diagram as follows, in which the rows
 are exact, and $p_0$, $p_1$, $p_3$ and $p_4$ are isomorphisms:
 \begin{center}
  \begin{tikzcd}
   A_0 \arrow[r,"f_0"] \arrow[d,"p_0"',"\simeq"] & 
   A_1 \arrow[r,"f_1"] \arrow[d,"p_1"',"\simeq"] & 
   A_2 \arrow[r,"f_2"] \arrow[d,"p_2"']        & 
   A_3 \arrow[r,"f_3"] \arrow[d,"p_3"',"\simeq"] & 
   A_4                 \arrow[d,"p_4"',"\simeq"] \\
   B_0 \arrow[r,"g_0"'] &
   B_1 \arrow[r,"g_1"'] &
   B_2 \arrow[r,"g_2"'] &
   B_3 \arrow[r,"g_3"'] &
   B_4.
  \end{tikzcd}
 \end{center}
 Then $p_2$ is also an isomorphism.
\end{proposition}
\begin{proof}
 First suppose that $a_2\in A_2$ and $p_2(a_2)=0$.  It follows that
 $p_3f_2(a_2)=g_2p_2(a_2)=g_2(0)=0$, but $p_3$ is an isomorphism,
 so $f_2(a_2)=0$, so $a_2\in\ker(f_2)$.  The top row is exact, so
 $\ker(f_2)=\image(f_1)$, so we can choose $a_1\in A_1$ with
 $f_1(a_1)=a_2$.  Put $b_1=p_1(a_1)\in B_1$.  We then have
 $g_1(b_1)=g_1p_1(a_1)=p_2f_1(a_1)=p_2(a_2)=0$, so $b_1\in\ker(g_1)$.
 The bottom row is exact, so $\ker(g_1)=\image(g_0)$, so we can choose
 $b_0\in B_0$ with $g_0(b_0)=b_1$.  As $p_0$ is an isomorphism, we can
 now put $a_0=p_0^{-1}(b_0)\in A_0$.  We then have
 $p_1f_0(a_0)=g_0p_0(a_0)=g_0(b_0)=b_1=p_1(a_1)$.  Here $p_1$ is an
 isomorphism, so it follows that $f_0(a_0)=a_1$.  We now have
 $a_2=f_1(a_1)=f_1f_0(a_0)$.  However, as the top row is exact we have
 $f_1f_0=0$, so $a_2=0$.  We conclude that $p_2$ is injective.

 Now suppose instead that we start with an element $b_2\in B_2$.  Put
 $b_3=g_2(b_2)\in B_3$ and $a_3=p_3^{-1}(b_3)\in A_3$.  We then have
 $p_4f_3(a_3)=g_3p_3(a_3)=g_3(b_3)=g_3g_2(b_2)=0$ (because
 $g_3g_2=0$).  As $p_4$ is an isomorphism, this means that
 $f_3(a_3)=0$, so $a_3\in\ker(f_3)$.  As the top row is exact we have
 $\ker(f_3)=\image(f_2)$, so we can choose $a_2\in A_2$ with
 $f_3(a_2)=a_3$.  Put $b'_2=b_2-p_2(a_2)\in B_2$.  We have
 $g_2(b'_2)=g_2(b_2)-g_2p_2(a_2)=b_3-p_3f_2(a_2)=b_3-p_3(a_3)=0$,
 so $b'_2\in\ker(g_2)=\image(g_1)$.  We can thus choose $b'_1\in B_1$
 with $g_1(b'_1)=b'_2$.  Now put $a'_1=p_1^{-1}(b'_1)\in A_1$ and
 $a'_2=f_1(a'_1)\in A_2$.  We find that
 $p_2(a'_2)=p_2f_1(a'_1)=g_1p_1(a'_1)=g_1(b'_1)=b'_2=b_2-p_2(a_2)$, so
 $p_2(a_2+a'_2)=b_2$.  This shows that $p_2$ is also surjective, and
 so is an isomorphism as claimed.
\end{proof}

\begin{proposition}\label{prop-snake-lemma}
 Suppose we have a commutative diagram as follows, in which the rows
 are short exact sequences:
 \begin{center}
  \begin{tikzcd}
   A \arrow[rightarrowtail,r,"j"]    \arrow[d,"f"'] & 
   B \arrow[twoheadrightarrow,r,"q"] \arrow[d,"g"'] &
   C                                 \arrow[d,"h"'] \\
   A' \arrow[rightarrowtail,r,"j'"'] &
   B' \arrow[twoheadrightarrow,r,"q'"']  &
   C'
  \end{tikzcd}
 \end{center}  
 Then there is a unique homomorphism $\dl\:\ker(h)\to\cok(f)$ such
 that $\dl(q(b))=a'+f(A)$ whenever $g(b)=j'(a')$.  Moreover, this fits
 into an exact sequence 
 \[ 0 \to \ker(f) \xra{j} \ker(g) \xra{q} \ker(h) \xra{\dl} 
     \cok(f) \xra{j} \cok(g) \xra{q} \cok(h) \to 0.
 \]
\end{proposition}

\begin{proof}
 A \emph{snake} for the above diagram is a list $(c,b,a',\ov{a})$ such
 that
 \begin{itemize}
  \item[(1)] $c\in\ker(h)\leq C$
  \item[(2)] $b\in B$ with $qb=c$
  \item[(3)] $a'\in A'$ with $j'a'=gb\in B'$
  \item[(4)] $\ov{a}$ is the image of $a'$ in $\cok(f)$.
 \end{itemize}
 It is easy to see that the snakes form a subgroup of
 $\ker(h)\tm B\tm A'\tm\cok(f)$.  We claim that for all $c\in\ker(h)$,
 there exists a snake starting with $c$.  Indeed, as $q$ is
 surjective, we can choose $b\in B$ satisfying~(2).  Then
 $q'g(b)=hq(b)=h(c)=0$, so $g(b)\in\ker(q')=\img(j')$, so we can
 choose $a'\in A'$ satisfying~(3).  Finally, we can define $\ov{a}$ to
 be the image of $a'$ in $\cok(f)$, so that~(4) is satisfied: this
 gives a snake as required.  Next, we claim that any two snakes
 starting with $c$ have the same endpoint.  By subtraction we reduce
 to the following claim: if $(0,b,a',\ov{a})$ is a snake, then
 $\ov{a}=0$, or equivalently $a'\in\img(f)$.  Indeed, condition~(2)
 says that $b\in\ker(q)=\img(j)$, so we can find $a\in A$ with
 $b=ja$.  Now $j'(fa-a')=j'fa-gb=gja-gb=gb-gb=0$, and $j'$ is
 injective, so $fa=a'$ as required.  This allows us to construct a map 
 $\dl\:\ker(h)\to\cok(f)$ as follows: we define $\dl(c)$ to be the
 endpoint of any snake starting with $c$.

 We now need to check exactness of the resulting sequence.
 \begin{itemize}
  \item[(1)] As $j\:A\to B$ is injective, it is clear that the
   restricted map $j\:\ker(f)\to\ker(g)$ is also injective.
  \item[(2)] As the composite $A\xra{j}B\xra{q}C$ is zero, the same is
   true of the restricted composite
   $\ker(f)\xra{j}\ker(g)\xra{q}\ker(h)$.  Moreover, suppose we
   have $b\in\ker(g)$ with $qb=0$.  By the original exactness
   assumption we can find $a\in A$ with $ja=b$.  Now $j'fa=gja=gb=0$
   but $j'$ is injective so $fa=0$ so $a\in\ker(f)$.  Thus, $b$ is in
   the image of the map $j\:\ker(f)\to\ker(g)$.
  \item[(3)] Suppose we have $b\in\ker(g)$.  Then $(qb,b,0,0)$ is a
   snake starting with $qb$, showing that $\dl qb=0$.  Conversely,
   suppose that $c\in\ker(h)$ with $\dl c=0$, so there exists a snake
   $(c,b,a',0)$.  By the last snake condition, we must have
   $a'\in\img(f)$, say $a'=fa$ for some $a\in A$.  Put $b'=b-ja\in B$.
   Snake condition~(2) gives $qb=c$ but also $qj=0$ so $qb'=c$.  On
   the other hand, snake condition~(3) gives $gb=j'a'=j'fa=gja$ so
   $gb'=0$.  This means that $c$ is in the image of the map
   $q\:\ker(g)\to\ker(h)$. 
  \item[(4)] Suppose we have $c\in\ker(h)$ with $\dl(c)=\ov{a}$.  This
   means that there is a snake $(c,b,a',\ov{a})$.  We claim that the
   induced map $j'\:\cok(f)\to\cok(g)$ sends $\ov{a}$ to $0$, or
   equivalently that $j'a'\in\img(g)$.  This is clear because
   $j'a'=gb$ by the snake axioms.  Conversely, suppose that
   $\ov{a}\in\cok(f)$ and that $\ov{a}$ maps to $0$ in $\cok(g)$.
   This means that we can find $a'\in A$ representing $\ov{a}$ and
   that $ja'$ lies in the image of $g$, say $ja'=gb$ for some
   $b\in B$.  If we put $c=qb\in C$ we find that
   $hc=hqb=q'gb=q'j'a'=0$, so $c\in\ker(h)$.  By construction we see
   that $(c,b,a',\ov{a})$ is a snake so $\ov{a}\in\img(\dl)$.
  \item[(5)] As the composite $A'\xra{j'}B'\xra{q'}C'$ is zero, the
   same is clearly true for the induced maps
   $\cok(f)\to\cok(g)\to\cok(h)$.  Conversely, suppose we have an
   element $\ov{b}\in\cok(g)$ that maps to zero in $\cok(h)$.  We can
   choose $b'\in B'$ representing $\ov{b}$, and then $q'b'$ must lie
   in $\img(h)$, say $q'b'=hc$.  As $q$ is surjective we can choose
   $b\in B$ with $qb=c$.  This gives $q'b'=hqb=q'gb$, so the element
   $b'-gb$ lies in $\ker(q')$, which is the same as $\img(j')$.  We
   can therefore choose $a'\in A'$ with $b'=gb+j'a'$.  If we let
   $\ov{a}$ denote the image of $a'$ in $\cok(f)$, we find that
   $\ov{b}=j'\ov{a}$ in $\cok(g)$.
  \item[(6)] Finally, suppose we have $\ov{c}\in\cok(h)$.  We can then
   choose a representing element $c'\in C'$.  As $q'$ is surjective we
   can choose $b'\in B'$ with $q'b'=b$, then we can put
   $\ov{b}=[b']\in\cok(g)$.  We find that $q'\ov{b}=\ov{c}$.  This
   shows that $q'\:\cok(g)\to\cok(h)$ is surjective.
 \end{itemize}
\end{proof}

\begin{definition}
 A \emph{split short exact sequence} is a diagram
 \begin{center}
  \begin{tikzcd}
   A \arrow[r,bend left=20,"i"] &
   B \arrow[r,bend left=20,"p"] \arrow[l,bend left=20,"r"] &
   C                            \arrow[l,bend left=20,"s"]
  \end{tikzcd}
 \end{center}  
 where 
 \[ pi=0 \hspace{3em} 
    rs=0 \hspace{3em} 
    ri=1_A \hspace{3em} 
    ps=1_C \hspace{3em}
    ir+sp=1_B.
 \]
 This can also be displayed as 
 \begin{center}
  \begin{tikzcd}
   A \arrow[rr,"0"] \arrow[twoheadrightarrow,dr,"i"] \arrow[dd,"1"'] &&
   C \\ &
   B \arrow[twoheadrightarrow,ur,"p"] \arrow[twoheadrightarrow,dl,"r"] \\
   A &&
   C \arrow[ll,"0"] \arrow[rightarrowtail,ul,"s"] \arrow[uu,"1"']
  \end{tikzcd}
 \end{center}  
\end{definition}

\begin{example}
 Given abelian groups $A$ and $C$, there is a split short exact
 sequence 
 \begin{center}
  \begin{tikzcd}[sep=1.3cm]
   A         \arrow[r,bend left=20,"i'"] &
   A\oplus C \arrow[r,bend left=20,"p'"] \arrow[l,bend left=20,"r'"] &
   C                                     \arrow[l,bend left=20,"s'"]
  \end{tikzcd}
 \end{center}  
 given by 
 \begin{align*}
  i'(a) &= (a,0) & p'(a,c) &= c \\
  s'(c) &= (0,c) & r'(a,c) &= a.
 \end{align*}
\end{example}

The above example is essentially the only example, as we see from the
following result:
\begin{proposition}\label{prop-split-sum}
 Suppose we have a split short exact sequence
 \begin{center}
  \begin{tikzcd}
   A \arrow[r,bend left=20,"i"] &
   B \arrow[r,bend left=20,"p"] \arrow[l,bend left=20,"r"] &
   C                            \arrow[l,bend left=20,"s"]
  \end{tikzcd}
 \end{center}  
 Then there is an isomorphism $f\:B\to A\oplus C$ given by
 $f(b)=(r(b),p(b))$ with inverse $f^{-1}(a,c)=i(a)+s(c)$.  Moreover,
 the diagram 
 \begin{center}
  \begin{tikzcd}[sep=1.3cm]
   A \arrow[r,bend left=20,"i"]                         \arrow[equal,d] &
   B \arrow[r,bend left=20,"p"] \arrow[l,bend left=20,"r"] \arrow[d,"f","\simeq"'] &
   C                            \arrow[l,bend left=20,"s"] \arrow[equal,d] \\
   A         \arrow[r,bend left=20,"i'"] &
   A\oplus C \arrow[r,bend left=20,"p'"] \arrow[l,bend left=20,"r'"] &
   C                                     \arrow[l,bend left=20,"s'"]
  \end{tikzcd}
 \end{center}  
 commutes in the sense that
 \[ fi=i' \hspace{4em}
    fs=s' \hspace{4em}
    pf=p' \hspace{4em}
    rf=f'.
 \]
\end{proposition}
\begin{proof}
 We can certainly define homomorphisms $B\xra{f}A\oplus C\xra{g}B$ by
 $f(b)=(r(b),p(b))$ and $g(a,c)=i(a)+s(c)$.  We then have
 $gf(b)=(ir+sp)(b)=b$ and $fg(a,c)=(ri(a)+rs(c),pi(a)+ps(c))=(a,c)$ so
 $f$ and $g$ are mutually inverse isomorphisms.  We also have
 $fi(a)=(ri(a),pi(a))=(a,0)=i'(a)$, and the equations $fs=s'$, $pf=p'$
 and $rf=f'$ can be verified equally easily.
\end{proof}

Our terminology is justified by the following observation:
\begin{lemma}\label{lem-split-exact}
 If 
 \begin{center}
  \begin{tikzcd}
   A \arrow[r,bend left=20,"i"] &
   B \arrow[r,bend left=20,"p"] \arrow[l,bend left=20,"r"] &
   C                            \arrow[l,bend left=20,"s"]
  \end{tikzcd}
 \end{center}  
 is a split short exact sequence, then 
 \[ A \xra{i} B \xra{p} C \]
 is a short exact sequence.
\end{lemma}
\begin{proof}
 Suppose that $i(a)=0$.  As $ri=1_A$ we have $a=r(i(a))=r(0)=0$.  This
 shows that $\ker(i)=0$, so $i$ is injective.  Next, we have $ps=1_C$,
 so for all $c\in C$ we have $c=p(s(c))\in\image(p)$; so $p$ is
 surjective.  We also have $pi=0$, so $\image(i)\leq\ker(p)$.
 Finally, we have $ir+sp=1_B$, so for $b\in B$ we have
 $b=i(r(b))+s(p(b))$.  If $b\in\ker(p)$ this reduces to
 $b=i(r(b))\in\image(i)$, so $\ker(p)\leq\image(i)$ as required. 
\end{proof}

\begin{proposition}\label{prop-half-split}
 Let $A\xra{i}B\xra{p}C$ be a short exact sequence.
 \begin{itemize}
  \item[(a)] For any map $r\:B\to A$ with $ri=1_A$, there is a unique
   map $s\:C\to B$ such that $(i,p,r,s)$ gives a split short exact
   sequence. 
  \item[(b)] For any map $s\:C\to B$ with $ps=1_C$, there is a unique
   map $r\:B\to A$ such that $(i,p,r,s)$ gives a split short exact
   sequence. 
 \end{itemize}
\end{proposition}
\begin{proof}
 We will prove~(a) and leave the similar proof of~(b) to the reader.
 Define $f=1-ir\:B\to B$.  As $ri=1$ we have $fi=i-i(ri)=0$, so $f$
 vanishes on $\image(i)$, which is the same as $\ker(p)$.  We
 therefore have a well-defined map $s\:C\to B$ given by $s(c)=f(b)$
 for any $b$ with $p(b)=c$.  This means that $sp=f=1-ir$, or in other
 words $1_B=ir+sp$.  We also have $pi=0$ so $psp=p(1-ir)=p$, so
 $(ps-1)p=0$.  As $p$ is surjective this implies that $ps-1=0$ or
 $ps=1_C$.  Finally, we have $ri=1$ so $rsp=r(1-ir)=r-(ri)r=0$ but $p$
 is surjective so $rs=0$.  Thus, all the conditions for a split short
 exact sequence are verified.  If $s'\:C\to B$ is another map giving a
 split short exact sequence then we can subtract the equations
 $ir+sp=1$ and $ir+s'p=1$ to get $(s-s')p=0$ but $p$ is surjective so
 $s=s'$; this shows that $s$ is unique.
\end{proof}

\begin{proposition}\label{prop-internal-splitting}
 Let $B$ be an abelian group, and let $A$ and $C$ be subgroups such
 that $B=A+C$ and $A\cap C=0$.  Let $i\:A\to B$ and $s\:C\to B$ be the
 inclusion maps.  The there is a unique pair of homomorphisms
 $A\xla{r}B\xra{p}C$ giving a split short exact sequence.
\end{proposition}
\begin{proof}
 Consider an element $b\in B$.  As $B=A+C$ we can find
 $(a,c)\in A\oplus C$ such that $b=a+c$.  Suppose we have another pair
 $(a',c')\in A\oplus C$ with $b=a'+c'$.  We put $x=a-a'$, and by
 rearranging the equation $a+c=a'+c'$ we see that $x=c'-c$.  The first
 of these expressions shows that $x\in A$, and the second that
 $x\in C$.  As $A\cap C=0$ this means that $x=0$, so $a=a'$ and
 $c=c'$.  Thus, the pair $(a,c)$ is unique, so we can define maps
 $A\xla{r}B\xra{p}C$ by $r(b)=a$ and $p(b)=c$.  It is straightforward
 to check that these give a split short exact sequence.
\end{proof}

\begin{proposition}\label{prop-idempotent-splitting}
 Let $B$ be an abelian group, and let $e\:B\to B$ be a homomorphism
 with $e^2=e$.  Then $\image(e)=\ker(1-e)$ and $\ker(e)=\image(1-e)$
 and $B=\image(e)\oplus\image(1-e)$.
\end{proposition}
\begin{proof}
 First, if $b\in\image(e)$ then $b=e(a)$ for some $a$, so
 $(1-e)(b)=e(a)-e^2(a)=0$, so $b\in\ker(1-e)$.  Conversely, if
 $b\in\ker(1-e)$ then $b-e(b)=0$ so $b=e(b)\in\image(e)$.  This shows
 that $\image(e)=\ker(1-e)$ as claimed.  Now put $f=1-e$.  We then have
 $f^2=1-2e+e^2=1-2e+e=f$, so $f$ is another idempotent endomorphism of
 $B$.  We can thus apply the same logic to see that
 $\image(f)=\ker(1-f)$, or in other words $\image(1-e)=\ker(e)$.

 Now consider an arbitary element $b\in B$.  We can write $b$ as
 $e(b)+(1-e)(b)$, so $b\in\image(e)+\image(1-e)$; this shows that
 $B=\image(e)+\image(1-e)$.  Now suppose that
 $b\in\image(e)\cap\image(1-e)=\ker(1-e)\cap\ker(e)$.  This means that
 $(1-e)(b)=e(b)=0$, so $b=e(b)=0$.  This means that
 $\image(e)\cap\image(1-e)=0$, so the sum is direct.
\end{proof}

We now give a useful application of Proposition~\ref{prop-split-sum}
to the theory of additive functors.  We recall the definition:
\begin{definition}\label{defn-functor}
 A \emph{covariant functor} from abelian groups to abelian groups is a
 construction that gives an abelian group $F(A)$ for each abelian
 group $A$, and a homomorphism $f_*\:F(A)\to F(B)$ for each
 homomorphism $f\:A\to B$, in such a way that:
 \begin{itemize}
  \item[(a)] For identity maps we have $(1_A)_*=1_{F(A)}$ for all
   $A$.
  \item[(b)] For homomorphisms $A\xra{f}B\xra{g}C$ we have
   $(gf)_*=g_*f_*\:F(A)\to F(C)$.
 \end{itemize}
 We say that $F$ is a \emph{additive} if
 $(f_0+f_1)_*=(f_0)_*+(f_1)_*$ for all $f_0,f_1\:A\to B$.
\end{definition}
\begin{example}
 Fix an integer $n>0$.  We can then define an additive functor $F$ by
 $F(A)=A[n]=\{a\in A\st na=0\}$, and another additive functor $G$ by
 $G(A)=A/nA$.  In both cases the homomorphisms $f_*$ are just the
 obvious ones induced by $f$.
\end{example}

\begin{proposition}\label{prop-additive-functor}
 Let $F$ be an additive covariant functor as above.  Then for any
 abelian groups $A$ and $C$ we have an natural isomorphism
 $f\:F(A\oplus C)\to F(A)\oplus F(C)$ given by
 $f(b)=(r'_*(b),p'_*(b))$ with inverse $f^{-1}(a,c)=i'_*(a)+s'_*(c)$.
\end{proposition}
\begin{proof}
 Put $B=A\oplus C$, and recall that $1_B=i'r'+s'p'$.  As $F$ is an
 additive functor we have 
 \[ 1_{F(B)}=(1_B)_*=(i'r')_*+(s'p')_*=i'_*r'_*+s'_*p'_*. \]
 In the same way the equations $r'i'=1$, $p's'=1$, $p'i'=0$ and
 $r's'=0$ give $r'_*i'_*=1$, $p'_*s'_*=1$, $p'_*i'_*=0$ and
 $r'_*s'_*=0$, so we have a split short exact sequence 
 \begin{center}
  \begin{tikzcd}[sep=1.3cm]
   F(A) \arrow[r,bend left=20,"i'_*"] &
   F(B) \arrow[r,bend left=20,"p'_*"] \arrow[l,bend left=20,"r'_*"] &
   F(C)                               \arrow[l,bend left=20,"s'_*"]
  \end{tikzcd}
 \end{center}  
 Thus, Proposition~\ref{prop-split-sum} gives us an isomorphism
 $F(A\oplus C)=F(B)\to F(A)\oplus F(C)$, and by unwinding the
 definitions we see that this is given by the stated formulae.
\end{proof}

There is a similar statement for contravariant functors as follows.
\begin{definition}\label{defn-cofunctor}
 A \emph{contravariant functor} from abelian groups to abelian groups
 is a construction that gives an abelian group $F(A)$ for each abelian
 group $A$, and a homomorphism $f^*\:F(B)\to F(A)$ for each
 homomorphism $f\:A\to B$, in such a way that:
 \begin{itemize}
  \item[(a)] For identity maps we have $(1_A)_*=1_{F(A)}$ for all
   $A$.
  \item[(b)] For homomorphisms $A\xra{f}B\xra{g}C$ we have
   $(gf)^*=f^*g^*\:F(C)\to F(A)$.
 \end{itemize}
 We say that $F$ is \emph{additive} if
 $(f_0+f_1)^*=f_0^*+f_1^*$ for all $f_0,f_1\:A\to B$.
\end{definition}
\begin{proposition}\label{prop-additive-cofunctor}
 Let $F$ be an additive contravariant functor as above.  Then for any
 abelian groups $A$ and $C$ we have an natural isomorphism
 $f\:F(A\oplus C)\to F(A)\oplus F(C)$ given by
 $f(b)=((i')^*(b),(s')^*(b))$ with inverse
 $f^{-1}(a,c)=(r')^*(a)+(p')^*(c)$. 
\end{proposition}
\begin{proof}
 Essentially the same as Proposition~\ref{prop-additive-functor}.
\end{proof}

\section{Products and coproducts}
\label{sec-biprod}

If we have a finite list of abelian groups $A_1,\dotsc,A_n$, we can
form the product group $\prod_{i=1}^nA_i=A_1\tm\dotsb\tm A_n$, which
is also denoted by $\bigoplus_{i=1}^nA_i=A_1\oplus\dotsb\oplus A_n$.
This should be familiar.  These constructions can be generalised to
cover families of abelian groups $A_i$ indexed by a set $I$ that may
be infinite, and need not be ordered in any natural way.  This is a
little more subtle, and in particular $\bigoplus_iA_i$ is not the same
as $\prod_iA_i$ in this context.  In this section we will briefly
outline the relevant definitions and properties.

\begin{definition}\label{defn-cartesian-product}
 Let $I$ be a set, and let $(A_i)_{i\in I}$ be a family of abelian
 groups indexed by $I$.  The \emph{product group} $\prod_{i\in I}A_i$
 is the set of all systems $a=(a_i)_{i\in I}$ consisting of an element
 $a_i\in A_i$ for each $i\in I$.  We make this into an abelian group
 by the obvious rule
 \[ (a_i)_{i\in I} \pm (b_i)_{i\in I} = (a_i\pm b_i)_{i\in I}. \]
 For each $k\in I$ we define $\pi_k\:\prod_{i\in I}A_i\to A_k$ by
 $\pi_k((a_i)_{i\in I})=a_k$.  This is clearly a homomorphism.  We
 also define $\iota_k\:A_k\to\prod_{i\in I}A_i$ by 
 \[ \iota_k(a)_i = \begin{cases}
                    a \in A_k & \text{ if } i=k \\
                    0 \in A_i & \text{ if } i\neq k.
                   \end{cases}
 \]
\end{definition}

\begin{example}\label{eg-finite-product}
 If $I=\{1,2,\dotsc,n\}$, then $\prod_{i\in I}A_i$ is just the set of
 $n$-tuples $(a_1,\dotsc,a_n)$ with $a_i\in A_i$, as before.  
\end{example}
\begin{example}\label{eg-diagonal-product}
 Suppose we have a fixed group $U$, and we take $A_i=U$ for all $i$.
 Then $\prod_{i\in I}A_i$ is just the set $\Map(I,U)$ of all functions
 from $I$ to $U$, considered as a group under pointwise addition.
\end{example}

\begin{remark}\label{rem-categorical-product}
 It is easy to see that a homomorphism $f\:U\to\prod_{i\in I}A_i$ is
 essentially the same thing as a family of homomorphisms
 $f_i\:U\to A_i$, one for each $i\in I$.  Indeed, given such a family
 we define $f\:U\to\prod_{i\in I}A_i$ by $f(u)=(f_i(u))_{i\in I}$, and
 we can then recover the original homomorphisms $f_i$ as the
 composites $\pi_i\circ f$.  This means that $\prod_{i\in I}A_i$ is a
 product for the groups $A_i$ in the general sense considered in
 category theory.
\end{remark}

\begin{definition}\label{defn-coproduct}
 Given an element $a=(a_i)_{i\in I}\in\prod_{i\in I}A_i$, the
 \emph{support} of $a$ is the set 
 \[ \supp(a) = \{i\in I\st a_i\neq 0\} \sse I. \]
 We put 
 \[ \bigoplus_{i\in I} A_i =
     \{a\in\prod_{i\in I}A_i\st\supp(a) \text{ is a finite set } \}.
 \]
 It is easy to see that $\supp(a\pm b)\sse\supp(a)\cup\supp(b)$, and
 thus that $\bigoplus_{i\in I}A_i$ is a subgroup of
 $\prod_{i\in I}A_i$.  We call it the \emph{coproduct} of the family
 $(A_i)_{i\in I}$.  We also note that $\supp(\iota_k(a))\sse\{k\}$, so
 $\iota_k$ can be regarded as a homomorphism
 $A_k\to\bigoplus_{i\in I}A_i$.
\end{definition}

\begin{remark}\label{rem-finite-biproduct}
 If the index set $I$ is finite then all supports are automatically
 finite and so the coproduct is the same as the product.  In fact, we
 only need the set $I'=\{i\st A_i\neq 0\}$ to be finite for this to
 hold. 
\end{remark}

Definition~\ref{defn-coproduct} is again compatible with the more
general definition coming from category theory, as we see from the
following result:
\begin{proposition}\label{prop-categorical-coproduct}
 Suppose we have an abelian group $V$, and a system of homomorphisms
 $g_i\:A_i\to V$ for all $i\in I$.  Then there is a unique
 homomorphism $g\:\bigoplus_{i\in I}A_i\to V$ such that
 $g\circ\iota_k=g_k$ for all $k\in I$.
\end{proposition}
\begin{proof}
 Given a point $a=(a_i)_{i\in I}\in\bigoplus_{i\in I}A_i$, we define 
 \[ g(a) = \sum_{i\in\supp(a)} g_i(a_i) \in V. \]
 The terms in the sum are meaningful because $a_i\in A_i$ and
 $g_i\:A_i\to V$, and $\supp(a)$ is finite so there only finitely many
 terms so it is not a problem to add them up.  If we replace
 $\supp(a)$ by some larger finite set $J\sse I$ then this gives us
 some extra terms but they are all zero so the sum is unchanged.
 After taking $J=\supp(a)\cup\supp(b)$ it becomes easy to see that
 $g(a+b)=g(a)+g(b)$, so $g$ is a homomorphism.  Using
 $\supp(\iota_k(a))\sse\{k\}$ we see that $g\circ\iota_k=g_k$, as
 required.  Let $g'\:\bigoplus_{i\in I}A_i\to V$ be another
 homomorphism with $g'\circ\iota_k=g_k$ for all $k$.  If we have an
 element $a$ as before, we observe that
 $a=\sum_{i\in\supp(a)}\iota_i(a_i)$, and by applying $g'$ to this we
 get 
 \[ g'(a) = \sum_{i\in\supp(a)} g'(\iota_i(a_i)) =
     \sum_{i\in\supp(a)} g_i(a_i) = g(a),
 \]
 so $g$ is unique as claimed.
\end{proof}
\begin{remark}\label{rem-infinite-sum}
 It would at worst be a tiny abuse of notation to say that
 $g(a)=\sum_{i\in I}g_i(a_i)$.  This is a sum with infinitely many
 terms, which would not normally be meaningful, but only finitely many
 of the terms are nonzero, so the rest can be ignored.
\end{remark}

\begin{remark}\label{rem-internal-coproduct}
 Suppose we have an abelian group $A$, and a family of subgroups
 $(A_i)_{i\in I}$.  There is then a unique homomorphism
 $\sg\:\bigoplus_{i\in I}A_i\to A$ such that
 $\sg\circ\iota_k\:A_k\to A$ is just the inclusion for all $k$.  More
 explicitly, we just have $\sg(a)=\sum_{i\in\supp(a)}a_i$.  If this
 map $\sg$ is an isomorphism, we will say (with another slight abuse
 of notation) that $A=\bigoplus_{i\in I}A_i$.
\end{remark}

\section{Torsion groups}
\label{sec-torsion}

\begin{definition}\label{defn-torsion}
 Let $A$ be an abelian group.
 \begin{itemize}
  \item[(a)] We say that an element $a\in A$ is a \emph{torsion
    element} if $na=0$ for some integer $n>0$.
  \item[(b)] We write $\tors(A)$ for the set of torsion elements of
   $A$. This is easily seen to be a subgroup, because if $na=0$ and
   $mb=0$ then $nm(a\pm b)=0$.
  \item[(c)] We say that $A$ is a \emph{torsion group} if every
   element is torsion, or equivalently $\tors(A)=A$.  At the other
   extreme, we say that $A$ is \emph{torsion-free} if $\tors(A)=0$.
  \item[(d)] Now fix a prime $p$.  We say that $a$ is a
   \emph{$p$-torsion element} if $p^ka=0$ for some $k\geq 0$.  We
   write $\tors_p(A)$ for the set of $p$-torsion elements, which is
   again a subgroup.
  \item[(e)] We say that $A$ is a \emph{$p$-torsion group} if every
   element is $p$-torsion, or equivalently $\tors_p(A)=A$.  At the
   other extreme, we say that $A$ is \emph{$p$-torsion free} if
   $\tors_p(A)=0$.
 \end{itemize}
\end{definition}

\begin{remark}\label{rem-multiple-identity}
 We write $n.1_A$ for the endomorphism of $A$ given by $a\mapsto na$.
 Then $\tors(A)=\bigcup_{n>0}\ker(n.1_A)$, and $A$ is torsion-free if
 and only if the maps $n.1_A$ (for $n>0$) are all injective.
\end{remark}

\begin{example}\label{eg-finite-torsion}
 If $A$ is a finite abelian group with $|A|=n$ then Lagrange's Theorem
 tells us that $na=0$ for all $a\in A$, so $A$ is a torsion group.
 For another instructive proof of the same fact, consider the element
 $z=\sum_{x\in A}x$.  As $x$ runs over $A$, the elements $a+x$ also
 run over $A$, so $z=\sum_{x\in A}(a+x)=na+z$, so $na=0$.  It is a
 curious fact, which we leave to the reader, that $z$ itself is
 actually zero in all cases except when $|A|=2$.
\end{example}

\begin{example}\label{eg-torsion-free}
 It is clear that any free abelian group is torsion-free.  The groups
 $\Q$ and $\R$ are torsion-free but not free.
\end{example}
\begin{example}\label{eg-QZ}
 Consider the quotient group $A=\QZ$.  The subset
 \[ A_n =
   \{\Z=\tfrac{0}{n}+\Z,\tfrac{1}{n}+\Z,\dotsc,\tfrac{n-1}{n}+\Z\}
 \]
 is a cyclic subgroup of order $n$.  Any element $a\in\QZ$ can be
 written as $a=m/n+\Z$ for some $m,n\in\Z$ with $n>0$.  We can also
 write $m$ as $qn+r$ for some $q,r\in\Z$ with $0\leq r<n$ and observe
 that $a=m/n+\Z=r/n+q+\Z=r/n+\Z\in A_n$.  This proves that $A$ is the
 union of the subgroups $A_n$.  As $na=0$ for all $a\in A_n$, we see
 that $A$ is a torsion group.  One can also check that $A_n\leq A_m$
 if and only if $n$ divides $m$.  In particular, for each prime $p$ we
 have a chain of subgroups
 \[ A_p \leq A_{p^2} \leq A_{p^3} \leq \dotsb \leq
     \bigcup_k A_{p^k} = \tors_p(A).
 \]
\end{example}
\begin{example}\label{eg-RZ}
 Now consider instead the group $\R/\Z$.  Suppose we have a torsion
 element $a=t+\Z$.  This means that for some integer $n>0$ we have
 $nt\in\Z$, which implies that $t$ is rational.  It follows that
 $\tors(\R/\Z)=\QZ$.  A similar argument shows that $\R/\Q$ is
 torsion-free. 
\end{example}

The following result is known as the Chinese Remainder Theorem.
\begin{proposition}\label{prop-chinese}
 Suppose we have positive integers $n_1,\dotsc,n_r$ any two of which
 are coprime, and we put $n=\prod_in_i$.  Define 
 \[ \phi\: \Z/n \to (\Z/n_1) \tm \dotsb\tm (\Z/n_r) \]
 by 
 \[ \phi(k+n\Z) = (k+n_1\Z,\dotsc,k+n_r\Z). \]
 Then 
 \begin{itemize}
  \item[(a)] There exist integers $e_1,\dotsc,e_r$ such that
   $\sum_ie_i=1$ and $e_i=1\pmod{n_i}$ and $e_i=0\pmod{n/n_i}$.
  \item[(b)] The map $\phi$ is an isomorphism.
 \end{itemize}
\end{proposition}
\begin{proof}
 For $i\neq j$ we know that $n_i$ and $n_j$ are coprime, so we can
 choose integers $a_{ij}$ and $b_{ij}$ with $a_{ij}n_i+b_{ij}n_j=1$.
 We then put $f_{ij}=b_{ij}n_j=1-a_{ij}n_i$, so $f_{ij}=1\pmod{n_i}$
 and $f_{ij}=0\pmod{n_j}$.  Now fix $i$, and let $g_i$ be the product
 of the numbers $f_{ij}$ as $j$ runs over the remaining indices.  We
 find that $g_i=1\pmod{n_i}$, but $g_i$ is divisible by the product of
 all the $n_j$, or equivalently by $n/n_i$.  Thus, the numbers $g_i$
 almost have property~(a), but we will need a slight adjustment to
 make the sum equal to one.  However, we are now ready to prove~(b).  
 Given any integers $m_1,\dotsc,m_r$, we have 
 \[ \phi(\sum_im_ig_i+n\Z) = (m_1+n_1\Z,\dotsc,m_r+n_r\Z). \]
 This proves that $\phi$ is surjective, and the domain and codomain of
 $\phi$ both have order $n$, so $\phi$ must actually be an
 isomorphism.  By construction we have
 $\phi(\sum_ig_i+n\Z)=\phi(1+n\Z)$, and $\phi$ is injective, so
 $\sum_ig_i=1+nk$ for some $k$.  We define $e_i=g_i$ for $i<r$, and
 $e_r=1-\sum_{i<r}g_i=g_r-nk$; these clearly satisfy~(a).
\end{proof}

The following special case is often useful:
\begin{corollary}\label{cor-chinese}
 Suppose that the prime factorisation of $n$ is
 $n=p_1^{v_1}\dotsb p_r^{v_r}$, where the primes $p_i$ are all
 distinct.  Then there are integers $e_1,\dotsc,e_r$ such that
 $\sum_ie_i=1$ and $e_i=1\pmod{p_i^{v_i}}$ and
 $e_i=0\pmod{n/p_i^{v_i}}$.  Moreover, the natural map 
 \[ \Z/n \to (\Z/p_1^{v_1}) \tm \dotsb \tm (\Z/p_r^{v_r}) \]
 is an isomorphism. \qed
\end{corollary}

\begin{proposition}\label{prop-tors-split}
 For any abelian group $A$ we have $\tors(A)=\bigoplus_p\tors_p(A)$.
\end{proposition}
\begin{proof}
 Suppose we have a torsion element $a\in A$, so $na=0$ for some
 $n>0$.  We can factor this as $\prod_{i=1}^rp_i^{v_i}$ and then
 choose integers $e_i$ as in Corollary~\ref{cor-chinese}.  Now
 $e_i=0\pmod{n/p_i^{v_i}}$ so $p_i^{v_i}e_i$ is divisible by $n$, so
 $p_i^{v_i}e_ia=0$, so $e_ia\in\tors_{p_i}(A)$.  We also have
 $\sum_ie_i=1$, so $a=\sum_ie_ia\in\sum_i\tors_{p_i}(A)$.  This shows
 that $\tors(A)=\sum_p\tors_p(A)$.

 To show that the sum is direct, suppose we have a finite list of
 distinct primes $p_1,\dotsc,p_r$, and elements $a_i\in\tors_{p_i}(A)$
 with $\sum_ia_i=0$; we must show that $a_i=0$ for all $i$.  As
 $a_i\in\tors_{p_i}(A)$ we have $p_i^{v_i}a_i=0$ for some
 $v_i\geq 0$.  We again choose numbers $e_i$ as in
 Corollary~\ref{cor-chinese}.  As $p_i^{v_i}a_i=0$ and
 $e_i=1\pmod{p_i^{v_i}}$ we have $e_ia_i=a_i$.  On the other hand, for
 $j\neq i$ we have $e_i=0\pmod{p_j^{v_j}}$ and so $e_ia_j=0$.  We can
 thus multiply the relation $\sum_ja_j=0$ by $e_i$ to get $a_i=0$ as
 required. 
\end{proof}

\begin{lemma}\label{lem-flat-quotient}
 The quotient group $A/\tors(A)$ is always torsion-free.
\end{lemma}
\begin{proof}
 Suppose we have a torsion element $a=x+\tors(A)$ in $A/\tors(A)$.
 This means that for some $n>0$ we have $na=0$ or equivalentlt
 $nx\in\tors(A)$.  This in turn means that for some $m>0$ we have
 $mnx=0$, which shows that $x$ itself is a torsion element in $A$.
 This means that the coset $a=x+\tors(A)$ is zero, as required.
\end{proof}

\section{Finitely generated abelian groups}
\label{sec-fin-gen}

Let $A$ be an abelian group, and let $a_1,\dotsc,a_r$ be elements of
$A$.  We then have a homomorphism $f\:\Z^r\to A$ given by
\[ f(n_1,\dotsc,n_r) = n_1a_1 + \dotsb + n_ra_r. \]
In particular, if $e_i$ is the $i$'th standard basis vector in $\Z^r$
then $f(e_i)=a_i$.  Conversely, if we start with a homomorphism
$f\:\Z^r\to A$ we can put $a_i=f(e_i)\in A$ and we find that 
\[ f(n_1,\dotsc,n_r) = f(\sum_in_ie_i) = 
    \sum_in_ia_i,
\]
so everything fits together as before.  The image of $f$ is the
smallest subgroup of $A$ containing all the elements $a_i$, or in
other words the subgroup generated by $\{a_1,\dotsc,a_r\}$.  This
justifies the following definition:

\begin{definition}\label{defn-fin-gen}
 We say that an abelian group $A$ is \emph{finitely generated} if
 there exists a surjective homomorphism $f\:\Z^r\to A$ for some $r$. 
\end{definition}

\begin{example}\label{eg-finite-fg}
 Suppose that $A$ is actually finite, so we can choose a list
 $a_1,\dotsc,a_r$ that contains all the elements of $A$.  The
 corresponding map $\Z^r\to A$ is certainly surjective, so $A$ is
 finitely generated.
\end{example}

Our main aim in this section is to prove the following classification
theorem: 
\begin{theorem}\label{thm-fin-gen}
 Let $A$ be a finitely generated abelian group.  Then $A$ can be
 decomposed the direct sum of a finite list of subgroups, each of
 which is isomorphic either to $\Z$, or to $\Z/p^v$ for some prime $p$
 and some $v>0$.  The number of subgroups of each type in the
 decomposition is uniquely determined, although the precise list of
 subgroups is not.
\end{theorem}

\begin{remark}\label{rem-cyclic-splitting}
 Note that Proposition~\ref{prop-chinese} gives a decomposition of the
 stated type for the cyclic group $\Z/n$.
\end{remark}

The groups $\Z^r$ themselves are of course finitely generated.  It is
convenient to observe that no two of them are isomorphic:
\begin{lemma}\label{lem-rank}
 If $\Z^r$ is isomorphic to $\Z^s$, then $r=s$.
\end{lemma}
\begin{proof}
 Any isomorphism $f\:A\to B$ induces an isomorphism $A/2A\to B/2B$, so
 in particular $|A/2A|=|B/2B|$.  We have $\Z^r/2\Z^r=(\Z/2)^r$, which
 has order $2^r$, and the claim follows easily.
\end{proof}

This means that the term 'rank' in the following definition is
well-defined: 
\begin{definition}\label{defn-rank}
 We say that an abelian group $A$  is \emph{free of rank $r$} if it is
 isomorphic to $\Z^r$.
\end{definition}

\begin{lemma}\label{lem-pid}
 Let $A$ be a subgroup of $\Z$.  Then either $A=0\simeq\Z^0$ or
 $A=d\Z\simeq\Z$ for some (unique) $d>0$. 
\end{lemma}
\begin{proof}
 The case where $A=0$ is trivial, so suppose that $A\neq 0$.  As
 $A=-A$ we see that $A$ must contain at least one strictly positive
 integer.  Let $d$ be the smallest strictly positive integer in $A$.
 It is than clear that $d\Z\sse A$.  Conversely, suppose that
 $n\in A$.  As $d>0$ we see that $n$ must lie between $id$ and
 $(i+1)d$ for some $i\in\Z$, say $n=id+j$ with $0\leq j<d$.  Now
 $j=n-id\in A$ and $0\leq j<d$, which contradicts the defining
 property of $d$ unless $j=0$.  We thus have $n=di$, showing that
 $A=d\Z$ as claimed.
\end{proof}

\begin{proposition}\label{prop-hereditary-finite}
 Let $A$ be a subgroup of $\Z^r$; then $A$ is free of rank at most
 $r$. 
\end{proposition}

\begin{proof}
 For $i\leq r$ we put 
 \[ F_i = \{x\in\Z^r \st x_{i+1} = \dotsb = x_r = 0\}\simeq\Z^i. \]
 We let $\pi_s\:\Z^r\to\Z$ be the projection map $x\mapsto x_s$, and
 put 
 \[ J = \{j\st\pi_j(A\cap F_j)\neq 0\}
      = \{j\st A\cap F_j > A\cap F_{j-1}\}.
 \]
 We can list the elements of this set as $j_1<\dotsb<j_s$ for some
 $s\leq r$ (possibly $s=0$).  We see from Lemma~\ref{lem-pid} that
 $\pi_{j_p}(A\cap F_{j_p})$ must have the form $d_p\Z$ for some
 $d_p>0$ say.  We can thus choose $a_p\in A\cap F_{j_p}$ such that
 $\pi_{j_p}(a_p)=d_p$ for all $p$.  We claim that the list
 $a_1,\dotsc,a_s$ is a basis for $A$ over $\Z$, so that $A$ is a free
 abelian group as claimed.  More precisely, we claim that
 $a_1,\dotsc,a_p$ is always a basis for $A\cap F_{j_p}$.  In the case
 $p=0$ we have the empty list and the zero group so the claim is
 clear.  When $p>0$ we can inductively assume the statement for
 $p-1$.  Consider an arbitrary element $u\in A\cap F_{j_p}$.  By the
 definition of $d_p$, we have $\pi_{j_p}(u)=m_pd_p$ for some
 $m_p\in\Z$.  The element $u'=u-m_pa_p$ then lies in $A\cap F_{j_p}$ 
 and satisfies $\pi_{j_p}(u')=0$ so in fact
 $u'\in A\cap F_{j_{p-1}}$.  By the induction hypothesis there are
 unique integers $m_1,\dotsc,m_{p-1}$ with
 $u'=m_1a_1+\dotsb+m_{p-1}a_{p-1}$, and it follows that 
 $u=u'+m_pa_p=m_1a_1+\dotsb+m_pa_p$.  This shows that $u$ can be
 expressed as an integer combination of $a_1,\dotsc,a_p$, and a
 similar argument shows that the expression is unique.  This completes
 the induction step, and after $s$ steps we see that $A$ itself has a
 basis as claimed.
\end{proof}
\begin{remark}\label{rem-canonical-basis}
 In the above proof, we can alter our choice of $a_p$ by subtracting
 off suitable multiples of $a_{p-1}$, $a_{p-2}$ and so on in turn to
 ensure that $0\leq \pi_{j_q}(a_p)<d_q$ for $1\leq q<p$.  One can then
 check that the resulting basis satisfying this auxiliary condition is
 in fact unique.
\end{remark}

\begin{corollary}\label{cor-subgroup-fg}
 If $A$ is finitely generated and $B$ is a subgroup of $A$ then $B$
 and $A/B$ are also finitely generated.
\end{corollary}
\begin{proof}
 Choose a surjective homomorphism $f\:\Z^r\to A$.  The composite
 $\Z^r\xra{f}A\xra{\pi}A/B$ is again surjective, so $A/B$ is finitely
 generated.  Now put $F=\{x\in\Z^r\st f(x)\in B\}$, and let
 $g\:F\to B$ be the restriction of $f$.  For $b\in B\leq A$ we can
 choose $x\in\Z^r$ with $f(x)=b$ (because $f$ is surjective).  Then
 $x\in F$ be the definition of $f$, and $g(x)=f(x)=b$; this proves
 that $g$ is surjective.  Moreover, $F$ is a subgroup of $\Z^r$, so it
 is isomorphic to $\Z^s$ for some $s\leq r$ by the proposition.  It
 now follows that $B$ is finitely generated.
\end{proof}

\begin{proposition}\label{prop-flat-free}
 Suppose that $F$ is finitely generated and torsion free; then $F$ is
 free. 
\end{proposition}
\begin{proof}
 Choose a surjective homomorphism $f\:\Z^r\to F$ with $r$ as small as
 possible.  Put $A=\ker(f)$; it will suffice to show that $A=0$.  If
 not, choose some nonzero element $a=(a_1,\dotsc,a_r)\in\ker(f)$.  Let
 $d$ be the greatest common divisor of $a_1,\dotsc,a_r$ (or
 equivalently, the number $d>0$ such that $\sum_i a_i\Z=d\Z$, which
 exists by Lemma~\ref{lem-pid}).  We find that $a/d\in\Z^r$ and
 $d\,f(a/d)=f(a)=0$ in $F$, so $f(a/d)$ is a torsion element, but $F$
 is assumed torsion free, so $f(a/d)=0$.  We may thus replace $a$ by
 $a/d$ and assume that $d=1$, so $\sum_i a_i\Z=\Z$.  We can thus
 choose integers $b_1,\dotsc,b_r$ with $\sum_ia_ib_i=1$.  Now define
 $\bt\:\Z^r\to\Z$ by $\bt(x)=\sum_ix_ib_i$ and put $U=\ker(\bt)$.  As
 $\bt(a)=1$ we find that $x-\bt(x)a\in U$ for all $x$, and it follows
 that $\Z^r=U\oplus\Z a$.  Proposition~\ref{prop-hereditary-finite} tells
 us that $U$ is free, and using the splitting $\Z^r=U\oplus\Z a$ we see
 that $U$ has rank $r-1$.  As $f(a)=0$ we also see that
 $f(U)=f(U\oplus\Z a)=f(\Z^r)=F$, so $f$ restricts to give a surjective
 homomorphism $U\to F$.  This contradicts the assumed minimality of
 $r$, so we must have $A=0$ after all, so $f\:\Z^r\to F$ is an
 isomorphism. 
\end{proof}

\begin{corollary}\label{cor-free-summand}
 Let $A$ be a finitely generated abelian group.  Then $\tors(A)$ is
 finite, and there exists a finitely generated free subgroup $F\leq A$
 such that $A=\tors(A)\oplus F\simeq\tors(A)\oplus\Z^s$ for some $s$.
\end{corollary}
\begin{proof}
 Firstly, Corollary~\ref{cor-subgroup-fg} tells us that $\tors(A)$ is
 finitely generated, so we can choose a finite list of generators, say
 $a_1,\dotsc,a_r$.  These must be torsion elements, so we can choose
 $n_i>0$ with $n_ia_i=0$.  This means that the corresponding
 surjection $f\:\Z^r\to\tors(A)$ factors through the finite quotient
 group $\prod_{i=1}^r(\Z/n_i)$, so $\tors(A)$ is finite as claimed.
 Next, the quotient group $A/\tors(A)$ is finitely generated (by
 Corollary~\ref{cor-subgroup-fg}) and torsion-free (by
 Lemma~\ref{lem-flat-quotient}) so it is free of finite rank by
 Proposition~\ref{prop-flat-free}.  We can thus choose an isomorphism
 $\ov{g}\:\Z^s\to A/\tors(A)$.  Now choose an element $a_i\in A$
 representing the coset $g(e_i)$ (for $i=1,\dotsc,s$) and define
 $g\:\Z^s\to A$ by $g(x)=\sum_ix_ia_i$, and put $F=g(\Z^s)$.  If we
 let $q$ denote the quotient map $A\to A/\tors(A)$ we have
 $qg=\ov{g}$, which is an isomorphism.  It follows that
 $g\:\Z^s\to F$ is an isomorphism, so $F$ is free as claimed.  Next,
 let $h\:A\to F$ be the composite
 \[ A \xra{q} A/\tors(A) \xra{\ov{g}^{-1}} \Z^s \xra{g} F. \]
 We find that $qh=q$, so $q(a-h(a))=0$, so $a-h(a)\in\tors(A)$ for all
 $a$.  This implies that $a=(a-h(a))+h(a)\in\tors(A)+F$, so
 $A=\tors(A)+F$.  Moreover, the intersection $\tors(A)\cap F$ is both
 torsion and torsion-free, so it must be zero, so the sum is direct.
\end{proof}

This corollary allows us to focus on $\tors(A)$, which is a finite
group, of order $n$ say.  Proposition~\ref{prop-tors-split} gives a
splitting $\tors(A)=\bigoplus_p\tors_p(A)$, and it is clear that
$\tors_p(A)$ can only be nonzero if $p$ divides $n$.  In that case,
$\tors_p(A)$ will be a finite abelian group whose order is a power of
$p$.  

\begin{lemma}
 Let $A$ be an abelian group of order $p^v$.  Suppose we have an
 element $c$ of order $p^w$, and that every other element has order
 dividing $p^w$, and that the subgroup $C=\Z c$ has nontrivial
 intersection with every nontrivial subgroup.  Then $A=C$.
\end{lemma}
\begin{proof}
 Consider a nontrivial element $a\in A$.  The order of $a+C$ in $A/C$
 will then be $p^i$ for some $i\leq w$.  We then have $p^ia=mc$ for
 some $m\in\Z$, and the assumption $(\Z a)\cap C\neq 0$ means that
 $mc\neq 0$.  We can thus write $mc=up^jc$ for some $j<w$ and some $u$
 such that $u\neq 0\pmod{p}$.  It follows that the order of $mc$ in
 $A$ is $p^{w-j}$, and thus that the order of $a$ in $A$ is
 $p^{w-j+i}$.  By assumption, this is at most $p^w$, so $i\leq j$.  We
 can thus put $b=a-up^{j-i}c$, and observe that $p^ib=0$.  Now $b$ is
 congruent to $a$ mod $C$, so it again has order $p^i$ in $A/C$, but
 $p^ib$ is already zero in $A$, so $(\Z b)\cap C=0$.  As $C$ meets
 every nontrivial subgroup, we must have $\Z b=0$, so
 $a=up^{j-i}c\in C$.  This means that $A=C$ as claimed.
\end{proof}

\begin{corollary}\label{cor-split-once}
 Let $A$ be an abelian group of order $p^v$, and suppose that the
 largest order of any element of $A$ is $p^w$.  Then
 $A\simeq B\oplus(\Z/p^w)$ for some subgroup $B$ of order $p^{v-w}$.  
\end{corollary}
\begin{proof}
 Choose an element $c$ of order $p^w$, and let $C$ be the subgroup
 that it generates, so $C\simeq\Z/p^w$.  Among the subgroups $B$ with
 $B\cap C=0$, choose one of maximal order.  Then put $\ov{A}=A/B$, and
 let $\ov{C}$ be the image of $C$ in $\ov{A}$, which is isomorphic to
 $C$ because $B\cap C=0$.  It will suffice to prove that $A=B+C$, or
 equivalently that $\ov{A}=\ov{C}$.  By the lemma, we need only check
 that $\ov{C}$ has nontrivial intersection with every nontrivial
 subgroup of $\ov{A}$.  Consider a nonzero element $\ov{a}\in\ov{A}$,
 and choose a representing element $a\in A\sm B$.  Then $\Z a+B$ is
 strictly larger than $B$ and must meet $C$ nontrivially, so there
 exists $k\in\Z$ and $b\in B$ with $ka+b\in C\sm\{0\}$.  If $ka+b$
 were in $B$ it would give a nontrivial element of $B\cap C$, contrary
 to assumption.  It follows that $k\ov{a}$ is nontrivial in $\ov{A}$
 and lies in $C$, as required.
\end{proof}

\begin{corollary}\label{cor-local-split}
 Let $A$ be an abelian group of order $p^v$.  Then $A$ is isomorphic
 to $\bigoplus_{i=1}^r\Z/p^{w_i}$ for some list $w_1,\dotsc,w_r$ of
 positive integers with $\sum_iw_i=v$.
\end{corollary}
\begin{proof}
 This follows by an evident induction from
 Corollary~\ref{cor-split-once}. 
\end{proof}

\begin{definition}\label{defn-f-p-k}
 Let $A$ be a finite abelian group.  For any prime $p$ and positive
 integer $k$, we put 
 \[ F_p^k(A) = \{a\in p^{k-1}A\st pa=0\}. \]
 This is a finite abelian group of exponent $p$, so it has order $p^v$
 for some $v$.  We define $f_p^k(A)$ to be this $v$, and we also put
 $g_p^k(A)=f_p^k(A)-f_p^{k+1}(A)$.  
\end{definition}

\begin{proposition}\label{prop-f-p-k}
 Let $A$ be a finite abelian group.
 \begin{itemize}
  \item[(a)] If $A\simeq A'$, then $F_p^k(A)\simeq F_p^k(A')$ for all
   $p$ and $k$, so $f_p^k(A)=f_p^k(A')$ and $g_p^k(A)=g_p^k(A')$.
  \item[(b)] If $A=B\oplus C$ then $F_p^k(A)=F_p^k(B)\oplus F_p^k(C)$,
   so $f_p^k(A)=f_p^k(B)+f_p^k(C)$ and $g_p^k(A)=g_p^k(B)+g_p^k(C)$.
  \item[(c)] If $A$ has order not divisible by $p$, then $F_p^k(A)=0$
   and so $f_p^k(A)=g_p^k(A)=0$.
  \item[(d)] Suppose that $A$ has a decomposition as a sum of
   subgroups $\Z/p_i^{v_i}$ (with $v_i>0$).  Then $g_p^k(A)$ is the
   number of times that $\Z/p^k$ occurs in the decomposition.
 \end{itemize}
\end{proposition}
\begin{proof}
 Parts~(a) to~(c) are straightforward and are left to the reader.
 Given these, part~(d) reduces to the claim that $g_p^k(\Z/p^j)$ is
 one when $j=k$, and zero otherwise.  One can see from the definitions
 that $f_p^k(\Z/p^j)$ is one when $k\leq j$, and zero when $k>j$; the
 claim follows easily from this.
\end{proof}

\begin{proof}[Proof of Theorem~\ref{thm-fin-gen}]
 Let $A$ be a finitely generated abelian group.
 Corollary~\ref{cor-free-summand} and subsequent remarks show that
 $A\simeq\Z^s\oplus\bigoplus_{i=1}^m\tors_{p_i}(A)$ for some finite
 list of primes $p_i$.  After applying Corollary~\ref{cor-local-split}
 to each of the groups $\tors_{p_i}(A)$, we get the claimed splitting
 of $A$ as a sum of copies of $\Z$ and $\Z/p_j^{w_j}$.  The number $s$
 is the rank of the quotient group $A/\tors(A)$, which is well-defined
 by Lemma~\ref{lem-rank}.  We can also apply the last part of
 Proposition~\ref{prop-f-p-k} to $\tors(A)$ to see that the number of
 summands of each type is independent of the choice of splitting.
\end{proof}

\section{Free abelian groups and their subgroups}

In various places we have already used the free abelian group $\Z[I]$
generated by a set $I$.  We start with a more careful formulation of
this construction.  One approach is to define 
$\Z[I]=\bigoplus_{i\in I}\Z$ as in Definition~\ref{defn-coproduct}.
That is essentially what we will do, but we will spell out some
details. 

\begin{definition}\label{defn-MapIZ}
 Let $I$ be any set.  We write $\Map(I,\Z)$ for the set of all maps
 $u\:I\to\Z$.  These form an abelian group under pointwise addition.
 For any map $u\:I\to\Z$, the \emph{support} is the set
 \[ \supp(u) = \{i\in I \st u(i)\neq 0\} \sse I. \]
 We put 
 \[ \Map_0(I,\Z) = \{u\:I\to\Z\st \supp(u) \text{ is finite } \}. \]
 It is easy to see that $\supp(u\pm v)\sse\supp(u)\cup\supp(v)$, and
 thus that $\Map_0(I,\Z)$ is a subgroup of $\Map(I,\Z)$.

 Next, for any $i\in I$ we define $\dl_i\:I\to\Z$ by 
 \[ \dl_i(j) = \begin{cases} 1 & \text{ if } j = i \\
                             0 & \text{ if } j \neq i. \end{cases}
 \]
 Note that $\supp(\dl_i)=\{i\}$ so $\dl_i\in\Map_0(I,\Z)$.  
\end{definition}
\begin{remark}
 If $I$ is a finite set with $n$ elements then we see that
 $\Map_0(I,\Z)=\Map(I,\Z)\simeq\Z^n$.  The situation is a little more
 subtle when $I$ is infinite.
\end{remark}

The following lemma shows that $\Map_0(I,\Z)$ is generated freely, in
a certain sense, by the elements $\dl_i$. 
\begin{lemma}
 Let $A$ be an abelian group.  Then for any function $f\:I\to A$ there
 is a unique homomorphism $\ov{f}\:\Map_0(I,\Z)\to A$ such that
 $\ov{f}(\dl_i)=f(i)$ for all $i\in I$.
\end{lemma}
\begin{proof}
 We put 
 \[ \ov{f}(u) = \sum_{i\in\supp(u)} u(i)\, f(i). \]
 The terms are meaningful, because each $u(i)$ is in $\Z$ and each
 $f(i)$ is in $A$ so we can multiply to get an element of $A$.  The
 sum is meaningful because $u\in\Map_0(I,\Z)$, so $\supp(u)$ is
 finite, so there are only finitely many terms to add.  More
 explicitly, if $\supp(u)=\{i_1,\dotsc,i_r\}$ and $u(i_t)=n_t\in\Z$
 for all $t$ then 
 \[ \ov{f}(u) = n_1 f(i_1) + \dotsb + n_r f(i_r) \in A. \]

 Note that it would be harmless to replace $\supp(u)$ by any finite
 set $J$ with $\supp(u)\sse J\sse I$; this would introduce some extra
 terms, but they would all be zero.  After taking
 $J=\supp(u)\cup\supp(v)$ we can check that
 $\ov{f}(u+v)=\ov{f}(u)+\ov{f}(v)$, so $\ov{f}$ is a homomorphism.
 From the definitions it is clear that $\ov{f}(\dl_i)=f(i)$.  Now let
 $\al\:\Map_0(I,\Z)\to A$ be another homomorphism with
 $\al(\dl_i)=f(i)$.  Put $\bt=\al-\ov{f}$, so $\bt(\dl_i)=0$ for all
 $i$.  It is not hard to see that a general element $u$ as above can
 be expressed in the form 
 \[ u = n_1 \dl_{i_1} + \dotsb + n_r \dl_{i_r}. \]
 It follows that 
 \[ \bt(u) = n_1 \bt(\dl_{i_1}) + \dotsb + n_r \bt(\dl_{i_r}) = 
     n_1 . 0 + \dotsb + n_r . 0 = 0.
 \]
 This shows that $\bt=0$, so $\al=\ov{f}$ as required.
\end{proof}

The notation used so far is convenient for giving the definition, and
the proof of the above freeness property, but not for the
applications.  We thus introduce the following alternative:
\begin{definition}\label{defn-ZI}
 We write $\Z[I]$ for $\Map_0(I,\Z)$, and $[i]$ for $\dl_i$.  We say
 that an abelian group $A$ is \emph{free} if it is isomorphic to
 $\Z[I]$ for some $I$.
\end{definition}

The freeness property now takes the following form:
\begin{lemma}\label{lem-free-property}
 Let $A$ be an abelian group.  Then for any function $f\:I\to A$ there
 is a unique homomorphism $\ov{f}\:\Z[I]\to A$ such that
 $\ov{f}([i])=f(i)$ for all $i\in I$. 
\end{lemma}

One key result is as follows:
\begin{theorem}\label{thm-hereditary}
 If $A$ is a free abelian group, then every subgroup of $A$ is also
 free. 
\end{theorem}
This is a generalisation of Proposition~\ref{prop-hereditary-finite},
which covered the case where $A$ is finitely generated.  Below we will
state and prove some more refined statements.  To prove
Theorem~\ref{thm-hereditary} itself, we can use
Remark~\ref{rem-adapted-choice} and the case $k=\top$ of 
Lemma~\ref{lem-adapted-basis} (in the notation of
Definition~\ref{defn-hereditary-aux}).

To extend the proof of Proposition~\ref{prop-hereditary-finite} to
cover infinitely generated groups, we need two ingredients.  Firstly,
we need to modify the inductive argument so that it works for a
suitable class of infinite ordered sets.  Next, we need to show that
any set can be ordered in the required way.  The precise structure
that we need is as follows:

\begin{definition}\label{defn-well-order}
 A \emph{well-ordering} on a set $I$ is a relation on $I$ (denoted by
 $i\leq j$) such that 
 \begin{itemize}
  \item[(a)] For all $i\in I$ we have $i\leq i$.
  \item[(b)] For all $i,j\in I$ we have either $i\leq j$ or $j\leq i$,
   and if both hold then $i=j$.
  \item[(c)] For all $i,j,k\in I$, if $i\leq j$ and $j\leq k$ then
   $i\leq k$.
  \item[(d)] For any nonempty subset $J\sse I$ there is an element
   $j_0\in J$ such that $j_0\leq j$ for all $j\in J$.  (In other
   words, $j_0$ is smallest in $J$.)
 \end{itemize}
\end{definition}

\begin{remark}\label{rem-well-order}
 We have stated the axioms in a form that is conceptually natural but
 inefficient.  Axiom~(a) follows from~(d) by taking $J=\{i\}$, the
 first half of~(b) follows from~(d) by taking $J=\{i,j\}$, and with a
 little more argument one can deduce~(c) by taking $J=\{i,j,k\}$ and
 appealing to the second half of~(b).  Thus, we really only need~(d)
 together with the second half of~(b).
\end{remark}

\begin{example}\label{eg-well-order}
 The obvious ordering of $\N$ is a well-ordering, as is the obvious
 ordering on $\N\cup\{\infty\}$.  The obvious ordering on $\Z$ is not
 a well-ordering, because the whole set does not have a smallest
 element.  We can choose a bijection $f\:\N\to\Z$ (for example, by
 setting $f(2n)=n$ and $f(2n+1)=-n-1$) and use this to transfer the
 standard ordering of $\N$ to a nonstandard ordering of $\Z$ that is a
 well-ordering.  Alternatively, we can specify a well-ordering on $\Z$
 by the rules
 \[ 0 < 1 < 2 < 3 < 4 < \dotsb < -1 < -2 < -3 < \dotsb. \]
 By constructions such as these, one can give explicit well-orderings
 of most naturally occurring countable sets.
\end{example}

\begin{remark}\label{rem-successor}
 Let $I$ be a well-ordered set.  If $I$ is nonempty then it must have
 a smallest element, which we denote by $\bot_I$ or just $\bot$.  Now
 suppose that $i\in I$ and $i$ is not maximal, so the set
 $I_{>i}=\{j\in I\st j>i\}$ is nonempty.  Then $I_{>i}$ must have a
 smallest element.  We denote this by $s(i)$, and call it the
 \emph{successor} of $i$.  We say that an element $j\in I$ \emph{is a
  successor} if $j=s(i)$ for some $i$.  For example, in
 $\N\cup\{\infty\}$ the elements $0$ and $\infty$ are not successors,
 but all other elements are successors.
\end{remark}

\begin{theorem}\label{thm-well-order}
 Every set admits a well-ordering.
\end{theorem}

The proof will be given after some preliminaries.

There is no known well-ordering of $\R$, and indeed it is probably not
possible to specify a well-ordering concretely, although the author
does not know of any precise theorems to that effect.  Similarly,
there is no known well-ordering of the set of subsets of $\N$, or of
most other naturally occurring uncountable sets.  The problem is that
one needs to make an infinite number of arbitrary choices, which
cannot be done explicitly.  However, we shall assume the Axiom of
Choice, a standard principle of Set Theory, which says that such
choices are nonetheless possible.  More precisely, we shall assume
that every set has a choice function, in the following sense:

\begin{definition}\label{defn-choice-function}
 Let $I$ be a set, and let $P'(I)$ denote the set of nonempty subsets
 of $I$.  A \emph{choice function} for $I$ is a function
 $c\:P'(I)\to I$ such that $c(J)\in J$ for all $J\in P'(I)$.  (In
 other words, $c(J)$ is a ``chosen'' element of $J$.)
\end{definition}

If $I$ is well-ordered, we can define a choice function by taking
$c(J)$ to be the smallest element of $J$.  Conversely, if we are given
a choice function then we can use it to construct a well-ordering, as
we now explain.  In the literature it is more common to do this by
proving Zorn's Lemma as an intermediate step, but here we have chosen
to bypass that.

\begin{proof}[Proof of Theorem~\ref{thm-well-order}]
 Let $I$ be a set, and let $c$ be a choice function for $I$.  Let
 $P(I)$ be the set of all subsets of $I$, and put
 $P^*(I)=P(I)\sm\{I\}$.  Define $d\:P^*(I)\to I$ by $d(J)=c(I\sm J)$,
 and then define $e\:P(I)\to P(I)$ by 
 \[ e(J) = \begin{cases} 
            J\cup\{d(J)\} & \text{ if } J \neq I \\
            I & \text{ if } J=I.
           \end{cases}
 \]
 We call this the \emph{expander} function.  Clearly we have
 $J\sse e(J)$ for all $J$, with equality iff $J=I$.  Now say that a
 subset $\CA\sse P(I)$ is \emph{saturated} if 
 \begin{itemize}
  \item[(a)] Whenever $J\in\CA$, we also have $e(J)\in\CA$.
  \item[(b)] For any family of sets in $\CA$, the union of that family
   is also in $\CA$.
 \end{itemize}
 We say that a set $J$ is \emph{compulsory} if it lies in every
 saturated family.  For example, by applying~(b) to the empty family
 we see that the set $J_0=\emptyset$ is compulsory.  It follows
 using~(a) that the sets $J_n=e^n(\emptyset)$ are compulsory for all
 $n$.  Axiom~(b) then tells us that the set
 $J_\om=\bigcup_{n=0}^\infty J_n$ is compulsory, as is the set
 $J_{\om+1}=e(J_\om)$.  If we had developed the theory of infinite
 ordinals, we could use transfinite recursion to define compulsory
 sets $J_\al$ for all ordinals $\al$.  As we have not discussed that
 theory, we will instead use an approach that avoids it.  We let $\CC$
 denote the family of all compulsory sets.  This is clearly itself a
 saturated family.  We say that a set $J\in\CC$ is \emph{comparable}
 if for all other $K\in\CC$ we have either $J\sse K$ or $K\sse J$.
 Let $\CD\sse\CC$ be the set of all comparable sets; we will show that
 this is saturated, and thus equal to $\CC$.  Consider a family of
 comparable sets $J_\al$, with union $J$ say, and another set
 $K\in\CC$.  For each $\al$ we have $J_\al\sse K$ or $K\sse J_\al$,
 because $J_\al$ is comparable.  If $J_\al\sse K$ for all $K$ then
 clearly $J\sse K$.  Otherwise we must have $K\sse J_\al$ for some
 $\al$ but $J_\al\sse J$ so $K\sse J$.  This shows that $J$ is
 comparable, so $\CD$ is closed under unions.  Now consider a
 comparable set $J$; we claim that $e(J)$ is also comparable.  To see
 this, put
 \[ \CE_J=\{K\st e(J)\sse K \text{ or }  K\sse J\}. \]
 By a similar argument to the previous paragraph, this is closed under
 unions.  Suppose that $K\in\CE_J$.
 \begin{itemize}
  \item[(a)] If $e(J)\sse K$ then clearly $e(J)\subset e(K)$, so
   $e(K)\in\CE_J$.
  \item[(b)] If $K=J$ then $e(J)=e(K)$, so $e(K)\in\CE_J$.
  \item[(c)] Suppose instead that $K\subset J$.  As $K\in\CC$ we also
   have $e(K)\in\CC$, and $J$ is comparable so either $e(K)\sse J$ or
   $J\subset e(K)$.  In the latter case we have
   $K\subset J\subset e(K)$, so $|e(K)\sm K|\geq 2$, but
   $|e(K)\sm K|\leq 1$ by construction, so this is impossible.  We
   therefore have $e(K)\sse J$, so again $e(K)\in\CJ$.
 \end{itemize}
 We now see that $\CE_J$ is a saturated subset of $\CC$, so it must be
 all of $\CC$.  It follows easily from this that $e(J)$ is
 comparable.  This means that $\CD$ is a saturated subset of $\CC$, so
 it must be all of $\CC$, so all compulsory sets are comparable, or in
 other words $\CC$ is totally ordered by inclusion.

 We now claim that $\CC$ is in fact well-ordered by inclusion.  To see
 this, consider a nonempty family of compulsory sets $K_\al$.  Let $J$
 be the union of all compulsory sets that are contained in
 $\bigcap_\al K_\al$.  As $\CC$ is closed under unions, we see that
 $J$ is actually the largest compulsory set that is contained in
 $\bigcap_\al K_\al$.  In particular, the larger set $e(J)$ cannot be
 contained in $\bigcap_\al J_\al$, so for some $\al$ we have
 $e(J)\not\sse K_\al$.  Now $K_\al\in\CC$ and we have seen that
 $\CC=\CE_J$ and using this we see that $K_\al\sse J$.  From this it
 follows easily that $K_\al$ is the smallest set in the family, as
 required. 

 Next, put $\CC^*=\CC\sm\{I\}$, which is again well-ordered by
 inclusion.  For $i\in I$ we let $p(i)$ be the union of all compulsory
 sets that do not contain $i$.  As $\CC$ is closed under unions this
 defines a map $p\:I\to\CC^*$.  We can also restrict $d$ to get a map
 $d\:\CC^*\to I$ in the opposite direction.  Note that
 $e(p(i))=p(i)\cup\{d(p(i))\}$ is a compulsory set not contained in
 $p(i)$, so we must have $i\in e(p(i))$, but $i\not\in p(i)$ by
 construction, so we must have $d(p(i))=i$.  In the opposite
 direction, suppose we start with a compulsory set $J\in\CC^*$, and
 put $i=d(J)$, so $e(J)=J\amalg\{i\}$.  Now $p(i)\in\CC$ and we have
 seen that $\CC=\CE_J$ so either $e(J)\sse p(i)$ or $p(i)\sse J$.  The
 first of these would imply that $i\in p(i)$, contradicting the
 definition of $p(i)$, so we must instead have $p(i)\sse J$.  On the
 other hand, $J$ is one of the sets in the union that defines $p(i)$,
 so $J\sse p(i)$, so $J=p(i)=p(d(J))$.  This proves that the maps
 $I\xra{p}\CC^*\xra{d}I$ are mutually inverse bijections.  We can thus
 introduce a well-ordering of $I$ by declaring that $i\leq j$ iff
 $p(i)\sse p(j)$.
\end{proof}

We can now start to prove as promised that subgroups of free abelian
groups are free.  It will be enough to prove that every subgroup of
$\Z[I]$ is free, and by Theorem~\ref{thm-well-order} we may assume
that $I$ is well-ordered.  We will need some auxiliary definitions.

\begin{definition}\label{defn-hereditary-aux}
 Let $I$ be a well-ordered set, and let $A$ be a subgroup of $\Z[I]$.
 We put $I_\top=I\amalg\{\top\}$, ordered so that $i\leq\top$ for all
 $i\in I$ (which is again a well-ordering).  Put
 $I_{<j}=\{i\in I\st i<j\}$ and $A_{<j}=A\cap\Z[I_{<j}]$, and
 similarly for $I_{\leq j}$ and $A_{\leq j}$.  Let $\pi_j\:\Z[I]\to\Z$
 be the $j$'th projection, which is characterised by the fact that
 $\pi_j([j])=1$ and $\pi_j([i])=0$ for all $i\neq j$.  Put 
 \[ J = \{j\in I\st \pi_j(A_{\leq j}) \neq 0\} = 
        \{j\in I\st A_{<j} < A_{\leq j}\},
 \]
 and $J_{<k}=\{j\in J\st j<k\}$, and similarly for $J_{\leq k}$.  For
 $j\in J$, we let $d_j$ be the positive integer such that
 $\pi_j(A_{\leq j})=d_j\Z$.  We put 
 \[ B_j = \{x\in A_{\leq j}\st x_j=d_j\} \sse A, \]
 which is nonempty by the definition of $d_j$.  For any $k\in I_\top$,
 we define an \emph{adapted basis for $A_{<k}$} to be a map
 $a\:J_{<k}\to A$ such that for all $j\in J_{<k}$ we have
 $a(j)\in B_j$.  We also put 
 \[ T_k = \{u\in \Z[I_{<k}] \st 
      0\leq u_j<d_j \text{ for all } j \in J_{<k}
    \}.
 \]
 Note that this contains zero but is not a subgroup (unless $d_j=1$
 for all $j\in J$).
\end{definition}
\begin{remark}\label{rem-adapted-choice}
 As each $B(j)$ is nonempty, we see that there exist adapted bases.
 Implicitly we are using the Axiom of Choice here: there exists a
 choice function $c$ for $A$, and then we can take $a(j)=c(B(j))$ for
 all $j$.  However, we will prove as
 Proposition~\ref{prop-adapted-normalised} that there is a
 \emph{unique} adapted basis satisfying a certain normalisation
 condition, which enables us to avoid this use of choice.
\end{remark}

\begin{remark}\label{rem-adapted-trunc}
 Let $a$ be an adapted basis for $A_{<k}$.  We then have a unique
 homomorphism 
 \[ f_k\:\Z[J_{<k}]\to A_{<k} \]
 such that $f_k([j])=a(j)$ for all $j\in J_{<k}$.  If $m<k$ then
 $a|_{J_{<m}}$ is easily seen to be an adapted basis for $A_{<m}$.  It
 therefore gives rise to a homomorphism $f_m\:\Z[J_{<m}]\to A_{<m}$,
 which is just the restriction of the map $f_k\:\Z[J_{<k}]\to A_{<k}$.

 We can also define functions $g_m\:\Z[J_{<m}]\tm T_m\to\Z[I_{<m}]$
 by $g_m(x,y)=f_m(x)+y$.  
\end{remark}

The terminology is justified by the following result:
\begin{lemma}\label{lem-adapted-basis}
 Let $a$ be an adapted basis for $A_{<k}$.  Then the corresponding map
 $f_k\:\Z[J_{<k}]\to A_{<k}$ is an isomorphism, and the map
 $g_k\:\Z[J_{<k}]\tm T_k\to\Z[I_{<k}]$ is a bijection.  In particular,
 the group $A_{<k}$ is free.
\end{lemma}

We will deduce it from the following auxiliary result:
\begin{lemma}\label{lem-adapted-basis-aux}
 Let $a$ be an adapted basis for $A_{<k}$, and suppose that for all
 $m<k$ the maps $f_m$ and $g_m$ are bijective.  Then $f_k$ and $g_k$
 are also bijective.  
\end{lemma}
\begin{proof}
 There are three cases to consider:
 \begin{itemize}
  \item[(a)] $k$ is not a successor.
  \item[(b)] $k=s(m)$ for some $m\not\in J$.
  \item[(c)] $k=s(m)$ for some $m\in J$.
 \end{itemize}
 Suppose that case~(a) holds.  Then for $m<k$ we have $s(m)<k$ and so
 $f_{s(m)}\:\Z[J_{<s(m)}]\to A_{<s(m)}$ is an isomorphism.  Now
 $J_{<k}$ is easily seen to be the union of these sets $J_{<s(m)}$,
 and $A_{<k}$ is the union of the groups $A_{<s(m)}$.  It therefore
 follows that $f_k$ is also an isomorphism $\Z[J_{<k}]\to A_{<k}$, as
 required.  Essentially the same argument proves that $g_k$ is a
 bijection. 

 Next, in case~(b) we see from the definition of $J$ that
 $A_{<k}=A_{\leq m}=A_{<m}$ and similarly $J_{<k}=J_{<m}$ so $f_k=f_m$
 and this is an isomorphism as required.  We also have
 $T_k=T_m\tm\Z.[m]$ and $\Z[I_{<k}]=\Z[I_{<m}]\tm\Z.[m]$ so the
 bijectivity of $g_k$ follows from that of $g_m$.

 Finally, in case~(c), we know that $f_m$ and $g_m$ are bijective by
 assumption.  Suppose that $u\in\Z[I_{<k}]$.  We then have
 $u_m=r\,d_m+s$ for some $s$ with $0\leq s<d_m$.  Put
 $u'=u-r\,a_m-s\,[m]$, so $u'_m=0$, so $u'\in\Z[I_{<m}]$.  As $g_m$ is
 a bijection we see that there is a unique pair
 $(x',y')\in\Z[J_{<m}]\tm T_m$ with $u'=f_m(x')+y'$.  If we put
 $x=x'+r\,[m]\in\Z[J_{<k}]$ and $y=y'+s\,[m]\in T_k$ we find that
 $(x,y)$ is the unique pair with $g_k(x,y)=u$.  It follows that $g_k$
 is a bijection.  In the case where $u\in A_{<k}$ we must have
 $u_m\in d_m\Z$ so $s=0$ and $u'=u-r\,a_m\in A_{<m}$ so $y'=0$; using
 this we see that $f_k$ is also an isomorphism.
\end{proof}

\begin{proof}[Proof of Lemma~\ref{lem-adapted-basis}]
 We actually claim that more generally, the restricted maps
 $f_n\:\Z[J_{<n}]\to A_{<n}$ are isomorphisms for all $n\leq k$.  If
 not, as $I_{\leq k}$ is well-ordered, there must be a smallest $n$
 for which $f_n$ is not an isomorphism.  This means that $f_m$ is an
 isomorphism for $m<n$, so we can apply
 Lemma~\ref{lem-adapted-basis-aux} to $a|_{J_{<n}}$ to see that $f_n$
 is an isomorphism, which is a contradiction.  The claim follows.
\end{proof}

\begin{remark}\label{rem-transfinite-induction}
 The method that we used to deduce Lemma~\ref{lem-adapted-basis}
 from~\ref{lem-adapted-basis-aux} is called \emph{transfinite
  induction}; it is evidently an extension of the usual kind of
 induction over the natural numbers.  We will use transfinite
 induction again below without spelling it out so explicitly.
\end{remark}

As we remarked previously, Theorem~\ref{thm-hereditary} follows from
Theorem~\ref{thm-well-order}, Lemma~\ref{lem-adapted-basis} and
Remark~\ref{rem-adapted-choice}.  Because we need
Theorem~\ref{thm-well-order}, the proof is unavoidably
nonconstructive.  Nonetheless, we can remove one set of arbitrary
choices by pinning down a specific adapted basis, as we now explain.

\begin{definition}\label{defn-adapted-normalised}
 Let $a$ be an adapted basis for $A_{<k}$.  We say that $a$ is
 \emph{normalised} if for all $i,j\in J_{<k}$ with $i<j$ we have
 $0\leq a(j)_i<d_i$.
\end{definition}

\begin{proposition}\label{prop-adapted-normalised}
 There is a unique normalised adapted basis for $A$.
\end{proposition}

This follows by transfinite induction from the following lemma:
\begin{lemma}\label{lem-adapted-normalised}
 Suppose that for all $m<\top$ there is a unique normalised basis for
 $A_{<m}$.  Then there is a unique normalised basis for $A$.
\end{lemma}
\begin{proof}
 We must again separate three cases:
 \begin{itemize}
  \item[(a)] $\top$ is not a successor.
  \item[(b)] $\top=s(m)$ for some $m\not\in J$.
  \item[(c)] $\top=s(m)$ for some $m\in J$.
 \end{itemize}

 We first consider case~(a).  For each $m<\top$ we see that
 $s(m)<\top$, so by the inductive assumption we have a unique
 normalised adapted basis $a_m\:J_{<s(m)}\to A_{<s(m)}$.  Now for
 $n<m$ we see that $a_m|_{J_{<s(n)}}$ is an adapted basis for
 $A_{<s(n)}$ so it must be the same as $a_n$.  It follows that there
 is a unique map $a\:J=J_{<\top}\to A$ such that $a|_{J_{<s(m)}}=a_m$
 for all $m$.  (Explicitly, it is given by $a(m)=a_m(m)$ for all
 $m<\top$.)  It is straightforward to check that this is a normalised
 adapted basis for $A$, and that it is the unique one.

 Now consider instead case~(b).  Here we have $A_{<\top}=A_{<m}$ and
 $J_{<\top}=J_{<m}$ so everything is trivial.

 Finally, consider case~(c).  Let $a\:J_{<m}\to A_{<m}$ be the unique 
 normalised adapted basis for $A_{<m}$.  By the definition of $d_m$,
 we can choose $b\in A$ with $b_m=d_m$.  Put
 $b'=b-d_m\,[m]\in\Z[I_{<m}]$.  As
 $g_m\:\Z[J_{<m}]\tm T_m\to\Z[I_{<m}]$ is a bijection, there is a
 unique pair $(x,y)$ with $f_k(x)+y=b'$.  We put
 $a(m)=b-f_k(x)=y+d_m\,[m]$.  The description $a(m)=b-f_k(x)$ shows
 that $a(m)\in A$, and the description $a(m)=y+d_m\,[m]$ shows that
 $a(m)$ satisfies the conditions for a normalised adapted basis.  Now
 suppose we have another normalised adapted basis for $A$, say $a'$.
 Then $a|_{J_{<m}}$ and $a'|_{J_{<m}}$ are both normalised adapted
 bases for $A_{<m}$, so they are the same by the induction
 hypothesis, so $a(j)=a'(j)$ for all $j<m$.  We also have
 $a(m)_m=d_m=a'(m)_m$, so the element $u=a'(m)-a(m)$ lies in
 $A_{<m}$, so $u=f_m(t)$ for some $t\in\Z[J_{<m}]$.  If $t$ is
 nonzero, then there are only finitely many indices $j$ with
 $t_j\neq 0$, so we can let $k$ be the largest one.  We then find that
 $u_k=a'(m)_k-a(m)_k=t_kd_k$, so $a'(m)_k=a(m)_k\pmod{d_k}$.  On the
 other hand, the normalisation condition means that
 $0\leq a'(m)_k,a(m)_k<d_k$, and this can only be consistent if
 $a'(m)_k=a(m)_k$, so $t_k=0$, contradicting the choice of $k$.  Thus
 $t$ must actually be zero, showing that $a=a'$ as required.
\end{proof}

\section{Tensor and torsion products}

\begin{definition}\label{defn-IA}
 Let $A$ be an abelian group.  We make the free abelian group $\Z[A]$
 into a commutative ring by the rule 
 \[ (\sum_i n_i[a_i]) . (\sum_j m_j[b_j]) = \sum_{i,j}n_im_j[a_i+b_j]
 \]
 (so in particular $[a][b]=[a+b]$).  We define a ring homomorphism
 $\ep\:\Z[A]\to\Z$ by $\ep(\sum_in_i[a_i])=\sum_in_i$, and we define
 $I_A$ to be the kernel of $\ep$.  We write $I^2_A$ for the ideal
 generated by all products $xy$ with $x,y\in I_A$.  We also define a
 group homomorphism $q\:\Z[A]\to A$ by
 $q(\sum_in_i[a_i])=\sum_in_ia_i$.
\end{definition}

\begin{proposition}\label{prop-IA}
 \begin{itemize}
  \item[(a)] The abelian groups $I_A$ and $I_A^2$ are both free.
  \item[(b)] More specifically, the elements $\ip{a}=[a]-[0]$ for
   $a\in A\sm 0$ form a basis for $I_A$.
  \item[(c)] Put 
   \[ \ip{a,b} = \ip{a}\ip{b} = [a+b] - [a] - [b] + [0]
       = \ip{a+b} - \ip{a} - \ip{b}.
   \] 
   Then the set of all elements of this form generates $I^2_A$ as an
   abelian group. 
  \item[(d)] There is a natural short exact sequence
   $I^2_A\xra{j}I_A\xra{q}A$ (where $j$ is just the inclusion).
 \end{itemize}
\end{proposition}
\begin{proof}
 \begin{itemize}
  \item[(a)] Both $I_A$ and $I^2_A$ are subgroups of $\Z[A]$, so they
   are free by Theorem~\ref{thm-hereditary}.
  \item[(b)] In the case of $I_A$ it is easy to be more concrete.
   Suppose we have an element $x=\sum_in_i[a_i]\in I_A$.  Then
   $\sum_in_i=0$, so $x$ can also be written as
   $\sum_in_i([a_i]-[0])$, and it is clearly harmless to omit any
   terms where $a_i=0$, so we see that $x$ is in the subgroup
   generated by the elements $[a]-[0]$ with $a\neq 0$.  It is easy to
   see that all such elements lie in $I_A$ and that they are
   independent over $\Z$, so they form a basis for $I_A$ as claimed.
  \item[(c)] Let $M$ be the subgroup of $I_A$ generated by all
   elements of the form $\ip{a,b}$.  As the elements $[a]-[0]$
   generate $I_A$ as an ideal, it is clear from the description
   $\ip{a,b}=\ip{a}\ip{b}$ that $M$ generates $I^2_A$ as an
   ideal, so it will be enough to check that $M$ itself is already an
   ideal.  This follows easily from the identity
   $[x]\ip{a,b}=\ip{a,b+x}-\ip{a,x}$, which can be verified directly
   from the definitions.
  \item[(d)] It is clear from the definitions that
   $q(\ip{a,b})=a+b-a-b+0=0$, so $qj=0$ by part~(c), so we have an
   induced map $\ov{q}\:I_A/I^2_A\to A$.  In the opposite direction,
   we can define $s\:A\to I_A/I_A^2$ by $s(a)=[a]-[0]+I_A^2$.  We have 
   \[ s(a+b)-s(a)-s(b) = [a+b]-[a]-[b]+[0]+I_A^2 =
        \ip{a,b} + I_A^2 = I_A^2,
   \]
   which means that $s$ is a homomorphism.  It is visible that
   $\ov{q}s=1_A$, so $s$ is injective and $\ov{q}$ is surjective.  It
   is also clear from~(b) that $s(A)$ generates $I_A/I_A^2$ but $s$ is
   a homomorphism so $s(A)$ is already a subgroup of $I_A/I_A^2$, so
   $s$ is surjective.  This means that $s$ is an isomorphism, with
   inverse $\ov{q}$.  As $\ov{q}$ is an isomorphism we see that
   $I^2_A\xra{j}I_A\xra{q}A$ is exact.
 \end{itemize}
\end{proof}

\begin{definition}\label{defn-tensor}
 Let $A$ and $B$ be abelian groups.  We regard $A$ and $B$ as
 subgroups of $A\tm B$ in the obvious way.  In $\Z[A\tm B]$ we let $J$
 be the ideal generated by all elements $\ip{a}$ with $a\in A$, and we
 let $K$ be the ideal generated by all elements $\ip{b}$ with
 $b\in B$.  We then put $A\ot B=JK/(J^2K+JK^2)$, and write $a\ot b$
 for the coset $\ip{a,b}+J^2K+JK^2\in A\ot B$.  We also write
 $\Tor(A,B)=(J^2K\cap JK^2)/(J^2K^2)$. 
\end{definition}

\begin{remark}\label{rem-tensor}
 Note that $A\ot B$ is generated by elements of the form $a\ot b$.
 Moreover, because
 \begin{align*}
  \ip{a+a',b} - \ip{a,b} - \ip{a',b} &=
   \ip{a}\ip{a'}\ip{b} \in J^2K \\
  \ip{a,b+b'} - \ip{a,b} - \ip{a,b'} &=
   \ip{a}\ip{b}\ip{b'} \in JK^2
 \end{align*}
 we see that these satisfy
 \begin{align*}
  a \ot (b+b') &= (a\ot b) + (a\ot b') \\
  (a+a') \ot b &= (a\ot b) + (a'\ot b).
 \end{align*}
 It follows easily that $a\ot 0=0$ for all $a\in A$, and $0\ot b=0$
 for all $b\in B$, and $(na)\ot(mb)=nm(a\ot b)$ for all $n,m\in\Z$.
 In fact, $A\ot B$ can be defined more loosely as the abelian group
 generated by symbols $a\ot b$ subject only to the relations 
 $a\ot(b+b')=(a\ot b)+(a\ot b')$ and
 $(a+a')\ot b=(a\ot b)+(a'\ot b)$.  
\end{remark}

\begin{remark}\label{rem-tensor-functor}
 Suppose we have homomorphisms $f\:A\to A'$ and $g\:B\to B'$.  These
 give a homomorphism $f\tm g\:A\tm B\to A'\tm B'$, which induces a
 ring map $(f\tm g)_\bullet\:\Z[A\tm B]\to\Z[A'\tm B']$.  This sends
 the ideals $J$ and $K$ to the coresponding ideals in $\Z[A'\tm B']$
 and so induces a homomorphism $A\ot B\to A'\ot B'$, which we denote
 by $f\ot g$.  By construction we have $(f\ot g)(a\ot b)=f(a)\ot
 g(b)$.  It is not hard to see that this construction is functorial,
 in the sense that $1_A\ot 1_B=1_{A\ot B}$ and that $(f'\ot g')(f\ot
 g)=(f'f)\ot(g'g)$ for all $f'\:A'\to A''$ and $g'\:B'\to B''$.  It is
 also bilinear in the following sense: if $f_0,f_1\:A\to A'$ and
 $g_0,g_1\:B\to B'$ then
 \[ (f_0+f_1)\ot(g_0+g_1) = 
     (f_0\ot g_0)+(f_0\ot g_1)+(f_1\ot g_0)+(f_1\ot g_1)
 \]
 as homomorphisms from $A\ot B$ to $A'\ot B'$.
\end{remark}

\begin{remark}\label{rem-JK-generators}
 As in Proposition~\ref{prop-IA}, one can check that
 \begin{itemize}
  \item[(a)] $J$ is freely generated as an abelian group by the
   elements $\ip{a}[b]=[a+b]-[b]$ with $a\in A\sm\{0\}$ and $b\in B$.
  \item[(b)] $K$ is freely generated as an abelian group by the
   elements $[a]\ip{b}=[a+b]-[a]$ with $a\in A$ and $b\in B\sm 0$.
  \item[(c)] $JK$ is freely generated as an abelian group by the
   elements $\ip{a,b}$ with $a\in A\sm 0$ and $b\in B\sm 0$.
  \item[(d)] $J^2K$ is generated as an abelian group by the elements
   $\ip{a}\ip{a'}\ip{b}$ with $a,a'\in A$ and $b\in B$.
  \item[(e)] $JK^2$ is generated as an abelian group by the elements
   $\ip{a}\ip{b}\ip{b'}$ with $a\in A$ and $b,b'\in B$.
 \end{itemize}
\end{remark}

We can generalise the identities in Remark~\ref{rem-tensor} as follows:
\begin{definition}\label{defn-bilinear}
 Let $A$, $B$ and $V$ be abelian groups.  We say that a function
 $f\:A\tm B\to V$ is \emph{bilinear} if for all $a,a'\in A$ and all
 $b,b'\in B$ we have
 \begin{align*}
  f(a,b+b') &= f(a,b) + f(a,b') \\
  f(a+a',b) &= f(a,b) + f(a',b).
 \end{align*}
 More generally, we say that a map $g\:A_1\tm\dotsb\tm A_n\to V$ is
 \emph{multilinear} (or more specifically \emph{$n$-linear}) if for
 each $k$ and each $a_1,\dotsc,a_{k-1},a_{k+1},\dotsc,a_n$, the map 
 \[ x\mapsto g(a_1,\dotsc,a_{k-1},x,a_{k+1},\dotsc,a_n) \]
 is a homomorphism from $A_k$ to $V$.
\end{definition}

\begin{example}\label{eg-bilinear}\nl
 \begin{itemize}
  \item[(a)] Matrix multiplication defines a bilinear map
   $\mu\:M_n(\Z)\tm M_n(\Z)\to M_n(\Z)$ by $\mu(M,N)=MN$.
  \item[(b)] The dot product defines a bilinear map
   $\R^3\tm\R^3\to\R$, the cross product defines a bilinear map
   $\R^3\tm\R^3\to\R^3$, and the triple product
   $(u,v,w)\mapsto u.(v\tm w)$ defines a trilinear map
   $\R^3\tm\R^3\tm\R^3\to\R$. 
  \item[(c)] By construction, we have a bilinear map
   $\om\:A\tm B\to A\ot B$ defined by $\om(a,b)=a\ot b$.
 \end{itemize}
\end{example}

Example~(c) above is in a sense the universal example, as explained by
the following result:
\begin{proposition}\label{prop-tensor-bilinear}
 Let $f\:A\tm B\to V$ be a bilinear map.  Then there is a unique
 homomorphism $\ov{f}\:A\ot B\to V$ such that $\ov{f}(a\ot b)=f(a,b)$
 for all $a\in A$ and $b\in B$ (or equivalently $\ov{f}\circ\om=f$).
\end{proposition}
\begin{proof}
 As $A\ot B$ is generated by the elements $a\ot b$, it is clear that
 $\ov{f}$ will be unique if it exists.

 By Lemma~\ref{lem-free-property}, there is a unique homomorphism
 $f_0\:\Z[A\tm B]\to V$ such that $f_0([a,b])=f(a,b)$ for all $a$ and
 $b$.  Note that for $a\in A$ we have $f_0([a])=f_0([a,0])=f(a,0)=0$,
 and similarly $f_0([b])=0$ for $b\in B$, so 
 \[ f_0(\ip{a,b})=f_0([a,b])-f_0([a,0])-f_0([0,b])+f_0([0,0]) = f(a,b). \]
 Note also that 
 \[ f_0(\ip{a}\ip{a'}\ip{b}) = 
     f_0(\ip{a+a',b}) - f_0(\ip{a,b}) - f_0(\ip{a',b}) = 
     f(a+a',b) - f(a,b) - f(a',b) = 0,
 \]
 so (using Remark~\ref{rem-JK-generators}(d)) we see that
 $f_0(J^2K)=0$.  Similarly, we have $f_0(JK^2)=0$, so $f_0$ induces a
 homomorphism 
 \[ \ov{f} \: A\ot B = \frac{JK}{J^2K+JK^2} \to V \]
 with 
 \[ \ov{f}(a\ot b) = \ov{f}(\ip{a,b}+J^2K+JK^2) = 
     f_0(\ip{a,b}) = f(a,b)
 \]
 as required.
\end{proof}

\begin{remark}\label{rem-tensor-bilinear}
 For example, we have a bilinear map 
 $\mu\:M_n(\Z)\tm M_n(\Z)\to M_n(\Z)$ given by $\mu(M,N)=MN$, so there
 is a unique homomorphism $\ov{\mu}\:M_n(\Z)\ot M_n(\Z)\to M_n(\Z)$
 such that $\ov{\mu}(M\ot N)=MN$.  Rather than spelling this out
 explicitly, we will usually just say that
 $\ov{\mu}\:M_n(\Z)\ot M_n(\Z)\to M_n(\Z)$ is defined by
 $\ov{\mu}(M\ot N)=MN$.
\end{remark}

It will be convenient to reformulate
Proposition~\ref{prop-tensor-bilinear} in a different way.
We write $\Hom(A,B)$ for the set of homomorphisms from $A$ to $B$,
considered as a group under pointwise addition.  Similarly, we write
$\Bilin(A,B;V)$ for the group of bilinear maps from $A\tm B$ to $V$.
\begin{proposition}\label{prop-tensor-hom}
 For any abelian groups $A$, $B$ and $V$, there are natural
 isomorphisms 
 \[ \Hom(A\ot B,V) \simeq
     \Bilin(A,B;V) \simeq
     \Hom(A,\Hom(B,V)) \simeq
     \Hom(B,\Hom(A,V)).
 \]
 More specifically, if we have elements 
 \[ f_0\in\Hom(A\ot B,V) \hspace{2em}
    f_1\in\Bilin(A,B;V) \hspace{2em}
    f_2\in\Hom(A,\Hom(B,V)) \hspace{2em}
    f_3\in\Hom(B,\Hom(A,V))
 \]
 then they are related by the above isomorphisms if and only if for
 all $a\in A$ and $b\in B$ we have
 \[ f_0(a\ot b) = f_1(a,b) = f_2(a)(b) = f_3(b)(a). \]
\end{proposition}
\begin{proof}
 This is mostly trivial.  For any map $f_1\:A\tm B\to V$, we can
 define a map $f_2\:A\to\Map(B,V)$ by $f_2(a)(b)=f_1(a,b)$.  If $f_1$
 satisfies the right-linearity condition $f_1(a,b+b')$ we see that
 $f_2(a)(b+b')=f_2(a)(b)+f_2(a)(b')$, so $f_2(a)$ is a homomorphism
 from $B$ to $V$, or in other words $f_2$ is a map from $A$ to
 $\Hom(B,V)$.  If $f_1$ also satisfies the left linearity condition
 $f_1(a+a',b)=f_1(a,b)+f_1(a',b)$ then we see that $f_2(a+a')$ is the
 sum of the homomorphisms $f_2(a)$ and $f_2(a')$, so $f_2$ itself is a
 homomorphism, or in other words $f_2\in\Hom(A,\Hom(B,V))$.  All of
 this is reversible, so we have an isomorphism
 $\Bilin(A,B;V)\simeq\Hom(A,\Hom(B,V))$.  We can define
 $f_3(b)(a)=f_1(a,b)$ to obtain a similar isomorphism 
 $\Bilin(A,B;V)\simeq\Hom(B,\Hom(A,V))$.  Finally,
 Proposition~\ref{prop-tensor-bilinear} gives an isomorphism
 $\Bilin(A,B;V)\simeq\Hom(A\ot B,V)$.
\end{proof}

\begin{proposition}\label{prop-tensor-symmon}
 There are natural isomorphisms as follows:
 \begin{align*}
  \eta_A\: \Z\ot A & \to A & \eta_A(n\ot a) &= na \\
  \tau_{AB}\: A\ot B & \to B\ot A & \tau_{AB}(a\ot b) & =b\ot a \\
  \al_{ABC}\: A\ot(B\ot C) & \to (A\ot B)\ot C & 
   \al_{ABC}(a\ot(b\ot c)) &= (a\ot b)\ot c. 
 \end{align*}
 In other words, the operation $(-)\ot(-)$ is commutative, associative
 and unital up to natural isomorphism.
\end{proposition}
\begin{proof}
 First, we certainly have a bilinear map $\eta'_A\:\Z\tm A\to A$ given
 by $\eta'_A(n,a)=na$, and by Proposition~\ref{prop-tensor-bilinear}
 this gives a homomorphism $\eta_A\:\Z\ot A\to A$ as indicated.  We
 can also define a homomorphism $\zt_A\:A\to\Z\ot A$ by
 $\zt_A(a)=1\ot a$, and it is clear that $\eta_A\zt_A=1_A$.  In the
 opposite direction, we must show that the map 
 \[ \xi=1_{\Z\ot A}-\zt_A\eta_A\:\Z\ot A\to\Z\ot A \]
 is zero.  By Proposition~\ref{prop-tensor-hom}, it will suffice
 to show that the corresponding homomorphism
 $\xi'\:\Z\to\Hom(A,\Z\ot A)$ is zero.  This is given by 
 \[ \xi'(n)(a) = (n\ot a)-(1\ot na) \]
 so visibly $\xi'(1)=0$ but $\xi'$ is a homomorphism so
 $\xi'(n)=n.\xi'(1)=0$ for all $n$ as required.

 Next, Lemma~\ref{lem-free-property} tells us that there is a unique
 homomorphism $\tau'_{AB}\:\Z[A\tm B]\to\Z[B\tm A]$ with
 $\tau'_{AB}[a,b]=[b,a]$.  It is visible that this is an isomorphism
 (with inverse $\tau'_{BA}$) and that $\tau'_{AB}(R_{AB})=R_{BA}$ so
 there is an induced isomorphism $\tau_{AB}\:A\ot B\to B\ot A$, with
 inverse $\tau_{BA}$. 

 Now fix $a\in A$, and define $\al''(a)\:B\tm C\to (A\ot B)\ot C$ by
 $\al''(a)(b,c)=(a\ot b)\ot c$.  This is bilinear, and moreover
 $\al''(a+a')=\al''(a)+\al''(a')$, so we have a homomorphism 
 \[ \al'' \: A \to \Bilin(B,C;(A\ot B)\ot C). \]
 We also have an isomorphism 
 \[ \Bilin(B,C;(A\ot B)\ot C) \simeq \Hom(B\ot C,(A\ot B)\ot C) \]
 and using that we obtain a homomorphism 
 \[ \al' \: A \to \Hom(B\ot C;(A\ot B)\ot C). \]
 characterised by $\al'(a)(b\ot c)=(a\ot b)\ot c$.  As in
 Proposition~\ref{prop-tensor-hom} this corresponds to a 
 homomorphism $\al\:A\ot(B\ot C)\to (A\ot B)\ot C$, given by
 \[ \al(a\ot(b\ot c)) = \al'(a)(b\ot c) = (a\ot b)\ot c. \]
 In the same way, we can construct
 $\bt\:(A\ot B)\ot C\to A\ot(B\ot C)$ with
 $\bt((a\ot b)\ot c)=a\ot(b\ot c)$.  This means that $\bt\al(x)=x$
 whenever $x$ has the form $a\ot(b\ot c)$, but elements of that form
 generate $A\ot(B\ot C)$, so $\bt\al=1$.  A similar argument shows
 that $\al\bt=1$, so $\al$ is an isomorphism as required.
\end{proof}

\begin{proposition}\label{prop-tensor-distrib}
 For any families of abelian groups $(A_i)_{i\in I}$ and
 $(B_j)_{j\in J}$ there is a natural isomorphism
 \[ \left(\bigoplus_{i\in I}A_i\right) \ot
     \left(\bigoplus_{j\in J}B_j\right) 
     \simeq \bigoplus_{(i,j)\in I\tm J} (A_i\ot B_j).
 \]
 In particular, for any $A$, $B$ and $C$ we have
 $A\ot(B\oplus C)\simeq (A\ot B)\oplus(A\ot C)$.
\end{proposition}
\begin{proof}
 For brevity, let $L$ and $R$ be the left and right hand sides of the
 claimed isomorphism.  For each $i$ and $j$ we can define a bilinear
 map $g'_{ij}\:A_i\tm B_j\to L$ by
 $g'_{ij}(a,b)=\iota_i(a)\ot\iota_j(b)$.  This gives a homomorphism
 $g_{ij}\:A_i\ot B_j\to L$, and
 Proposition~\ref{prop-categorical-coproduct} tells us that there is a
 unique $g\:R\to L$ with $g\circ\iota_{ij}=g_{ij}$ for all $i$ and
 $j$.  In the opposite direction, we can define a bilinear map 
 $f'\:(\bigoplus A_i)\tm(\bigoplus B_j)\to R$ by 
 \[ f'(a,b) = \sum_{i\in\supp(a)} \sum_{j\in\supp(b)}
               \iota_{ij}(a_i\ot b_j).
 \]
 This corresponds in the usual way to a homomorphism $f\:L\to R$.  We
 leave it to the reader to check that $fg=1_R$ and $gf=1_L$.
\end{proof}
\begin{corollary}\label{cor-tensor-distrib}
 Given sets $I$ and $J$ and an abelian group $B$, we have natural
 isomorphisms $\Z[I]\ot B\simeq\bigoplus_{i\in I}B$ and
 $\Z[I]\ot\Z[J]\simeq\Z[I\tm J]$.  In particular, this gives
 $\Z^r\ot B\simeq B^r$ and $\Z^r\ot\Z^s\simeq\Z^{rs}$.
\end{corollary}
\begin{proof}
 By applying the Proposition to the family $\{\Z\}_{i\in I}$ and the
  family consisting of the single group $B$, we obtain
 $\Z[I]\ot B\simeq\bigoplus_{i\in I}B$.  If instead we use the family
 $\{\Z\}_{j\in J}$ on the right hand side, we obtain
 $\Z[I]\ot\Z[J]\simeq\Z[I\tm J]$.  
\end{proof}

Another straightforward example is as follows:
\begin{proposition}\label{prop-cyclic-tensor}
 For any integer $n$ and any abelian group $A$ there is a natural
 isomorphism $(\Z/n)\ot A\simeq A/nA$.  In particular, we have
 $(\Z/n)\ot(\Z/m)=\Z/(n,m)$, where $(n,m)$ denotes the greatest common
 divisor of $n$ and $m$.
\end{proposition}
\begin{proof}
 We can define a bilinear map $f'\:(\Z/n)\ot A\to A/nA$ by
 $f'(k+n\Z,a)=ka+nA$, and this induces a homomorphism
 $f\:(\Z/n)\ot A\to A/nA$.  In the opposite direction, we can define
 $g\:A/nA\to(\Z/n)\ot A$ by $g(a+nA)=(1+n\Z)\ot a$.  We leave it to
 the reader to check that these are well-defined and that $fg$ and
 $gf$ are identity maps.  In particular, this gives
 $(\Z/n)\ot(\Z/m)=\Z/(n\Z+m\Z)$, but it is standard that
 $n\Z+m\Z=(n,m)\Z$. 
\end{proof}

We saw in Section~\ref{sec-fin-gen} that every finitely generated
abelian group is a direct sum of groups of the form $\Z$ or $\Z/p^v$.
We can thus use Proposition~\ref{prop-tensor-distrib},
Corollary~\ref{cor-tensor-distrib} and
Proposition~\ref{prop-cyclic-tensor} to understand the tensor product
of any two finitely generated abelian groups.

We next consider the interaction between tensor products and
exactness. 
\begin{proposition}\label{prop-tensor-exact}
 Let $A$, $B$, $C$ and $U$ be abelian groups.
 \begin{itemize}
  \item[(a)] If we have an exact sequence
   \[ A\xra{j}B\xra{q}C\to 0, \] 
   then the resulting sequence
   \[ U\ot A \xra{1\ot j} U\ot B\xra{1\ot q} U\ot C \to 0 \]
   is also exact.  (In other words, tensoring is \emph{right exact}.)
  \item[(b)] If we have an injective map $j\:A\to B$ and $U$ is
   torsion-free, then $1\ot j\:U\ot A\to U\ot B$ is also injective.
  \item[(c)] If we have a short exact sequence
   \[ 0\to A\xra{j}B\xra{q}C\to 0, \] 
   and $U$ is torsion-free, then the resulting sequence
   \[ 0\to U\ot A \xra{1\ot j} U\ot B\xra{1\ot q} U\ot C \to 0 \]
   is also short exact.
 \end{itemize}
\end{proposition}

For the first part, we will use the following criterion:
\begin{lemma}\label{lem-right-test}
 A sequence $A\xra{j}B\xra{q}C\to 0$ is exact iff for every abelian
 group $V$, the resulting sequence
 $0\to\Hom(C,V)\xra{q^*}\Hom(B,V)\xra{j^*}\Hom(A,V)$ is exact.
\end{lemma}
\begin{remark}
 The evident analogous statement with short exact sequences is not
 valid.  We will investigate this in more detail later.
\end{remark}
\begin{proof}
 Let $\CS$ denote the first sequence, and write $\Hom(\CS,V)$ for the
 second one.

 Suppose that $\CS$ is exact, so $q$ is surjective and
 $\ker(q)=\image(j)$.  Suppose that $f\in\ker(q^*)$, so $f\:C\to V$ and
 $fq=0\:B\to V$.  This means that $f(q(b))=0$ for all $b\in B$, but
 $q$ is surjective, so $f(c)=0$ for all $c\in C$, so $f=0$.  This
 proves that $\ker(q^*)=0$, so $q^*$ is injective.  Now suppose that
 $g\in\ker(j^*)$, so $g\:B\to V$ and $gj=0$, or equivalently
 $g(\image(j))=0$, or equivalently $g(\ker(q))=0$.  We thus have a
 well-defined map $f\:C\to V$ given by $f(c)=g(b)$ for any $b$ with
 $q(b)=c$.  Now $f\in\Hom(C,V)$ and $q^*(f)=fq=g$, so
 $g\in\image(q^*)$.  This proves that $\ker(j^*)=\image(q^*)$, so
 $\Hom(\CS,V)$ is exact.

 Conversely, suppose that $\Hom(\CS,V)$ is exact for all $V$.  Take
 $V=\cok(q)=C/q(B)$ and let $f\:C\to V$ be the evident projection
 (which is surjective).  By construction we have $q^*(f)=0$ but $q^*$
 is assumed to be injective so $f$ is zero as well as being
 surjective.  This implies that $C/q(B)=0$ so $q$ is surjective.

 Now instead take $V=C$.  As $\Hom(\CS,C)$ is exact, we certainly have
 $j^*q^*=0\:\Hom(C,C)\to\Hom(A,C)$.  In particular, we see that
 $j^*q^*(1)=0$ in $\Hom(A,C)$, or in other words that $qj=0\:A\to C$.
 This implies that $\image(j)\leq\ker(q)$.

 Finally, take $V=\cok(j)=B/j(A)$, and let $g\:B\to V$ be the
 projection.  Then $j^*(g)=gj=0$, so $g\in\ker(j^*)=\image(q^*)$, so
 there exists $f\:C\to B/j(A)$ with $fq=g\:B\to B/j(A)$.  Now if
 $q(b)=0$ then $b+\image(j)=g(b)=f(q(b))=0$, so $b\in\image(j)$.  This
 proves that $\ker(q)=\image(j)$, so $\CS$ is exact as claimed.
\end{proof}

\begin{proof}[Proof of Proposition~\ref{prop-tensor-exact}]\nl
 \begin{itemize}
  \item[(a)] By Lemma~\ref{lem-right-test}, the sequence
   \[ 0 \to \Hom(C,V) \xra{q^*} \Hom(B,V) \xra{j^*} \Hom(A,V) \]
   is exact for all $V$.  As $V$ is arbitrary we can replace it by
   $\Hom(U,V)$, where $U$ and $V$ are both arbitrary.  This gives an
   exact sequence 
   \[ 0 \to \Hom(C,\Hom(U,V)) \xra{q^*}
            \Hom(B,\Hom(U,V)) \xra{j^*}
            \Hom(A,\Hom(U,V)),
   \]
   and we can use Proposition~\ref{prop-tensor-hom} to rewrite it as
   \[ 0 \to \Hom(U\ot C,V) \xra{(1\ot q)^*}
            \Hom(U\ot B,V) \xra{(1\ot j)^*}
            \Hom(U\ot A,V).
   \]
   Finally we apply Lemma~\ref{lem-right-test} in the opposite
   direction to see that the sequence
   $U\ot A\to U\ot B\to U\ot C\to 0$ is exact.
  \item[(b)] Now suppose instead that we have an injective map
   $j\:A\to B$, and that $U$ is torsion-free.  We must show that
   $(1\ot j)\:U\ot A\to U\ot B$ is injective.  If $U$ is actually free
   then we may assume that $U=\Z[I]$ for some set $I$.  In this case
   Corollary~\ref{cor-tensor-distrib} tells us that $1\ot j$ is just
   a direct sum of copies of $j$ and the claim is clear.  In
   particular, this holds whenever $U$ is finitely generated and
   torsion-free, as we see from Proposition~\ref{prop-flat-free}.  The
   real issue is to deduce the infinitely generated case from the finitely
   generated case.  Suppose we have an element $x\in U\ot A$ with
   $(1\ot j)(x)=0$.  We can write $x$ in the form
   $x=\sum_{i=1}^n u_i\ot a_i$ say.  We then have
   $\sum_{i=1}^nu_i\ot j(a_i)=0$.  Going back to
   Definition~\ref{defn-tensor}, we deduce that $\sum_i[u_i,j(a_i)]$
   can be expressed in $\Z[U\tm B]$ as a finite $\Z$-linear
   combination of terms of the form $[u+u',b]-[u,b]-[u',b]$ or
   $[u,b+b']-[u,b]-[u,b']$.  Choose such an expression, and let $S$ be
   the (finite) set of elements of $U$ that are involved in that
   expression, together with the elements $u_1,\dotsc,u_n$ occuring in
   our expression for $x$.  Let $U_0$ be the subgroup of $U$ generated
   by $S$, which is finitely generated and torsion-free.  We now have
   an element $x_0=\sum_iu_i\ot a_i$ in $U_0\ot A$ and we find that
   $(1\ot j)(x_0)=0$ in $U_0\ot B$.  By the finitely generated case we
   see that $x_0=0$, but $x$ is the image of $x_0$ under the evident
   homomorphism $A\ot U_0\to A\ot U$, so $x=0$ as required.

  \item[(c)] This is just a straightforward combination of~(a)
   and~(b). 
 \end{itemize}
\end{proof}

Note that in part~(b) of the Proposition, it is definitely necessary
to assume that $U$ is torsion-free.  Indeed, we can take $j\:A\to B$
to be $n.1_\Z\:\Z\to\Z$, and then $1\ot j$ is just $n.1_U$, and these
maps are injective for all $n>0$ if and only if $U$ is torsion free.

For various purposes it is important to understand the kernel of 
$1\ot j$ in more detail.  We will first discuss the case $U=\Z/n$,
which is quite straightforward.

\begin{definition}\label{defn-ann-n-A}
 We write $A[n]$ for $\{a\in A\st na=0\}$, which can also be
 identified with $\Hom(\Z/n,A)$.  (A homomorphism $u\:\Z/n\to A$
 corresponds to the element $u(1+n\Z)\in A[n]$.)
\end{definition}

\begin{proposition}\label{prop-Zn-six}
 Fix an integer $n>0$.  For any short exact sequence
 $A\xra{j}B\xra{q}C$, there is a unique homomorphism
 $\dl\:C[n]\to A/nA=(\Z/n)\ot A$ such that $\dl(q(b))=a+nA$ whenever
 $nb=j(a)$.  Moreover, this fits into an exact sequence
 \[ 0 \to A[n] \xra{j} B[n] \xra{q} C[n] \xra{\dl}
       A/n \xra{j} B/n \xra{q} C/n \to 0.
 \]
\end{proposition}
\begin{proof}
 This is just the Snake Lemma (Proposition~\ref{prop-snake-lemma})
 applied to the diagram 
 \begin{center}
  \begin{tikzcd}
     A \arrow[rightarrowtail,r,"j"] \arrow[d,"n.1_A"'] & 
     B \arrow[twoheadrightarrow,r,"q"]  \arrow[d,"n.1_B"'] &
     C                 \arrow[d,"n.1_C"'] \\
     A \arrow[rightarrowtail,r,"j"'] &
     B \arrow[twoheadrightarrow,r,"q"']  &
     C
    \end{tikzcd}
   \end{center}
   
\end{proof}

It turns out that there are similar six-term exact sequences in much
greater generality, involving the groups $\Tor(A,B)$ introduced in
Definition~\ref{defn-tensor}.  We start by recording an obvious fact:
\begin{lemma}\label{lem-tor-twist}
 There is an isomorphism $\tau_{AB}\:\Z[A\tm B]\to\Z[B\tm A]$ given by
 $\tau_{AB}([a,b])=[b,a]$, and this induces an isomorphism
 $\tau_{AB}\:\Tor(A,B)\to\Tor(B,A)$ with inverse $\tau_{BA}$. \qed
\end{lemma}

\begin{lemma}\label{lem-tor-four}
 Let $A$ and $B$ be abelian groups, and let $J$ and $K$ be ideals in
 $\Z[A\tm B]$ as in Definition~\ref{defn-tensor}.  Then there are
 natural exact sequences
 \[ 0 \to \Tor(A,B) \to A\ot I_B^2
        \xra{1\ot j_B} A\ot I_B \xra{1\ot q_B} A\ot B \to 0 
 \]
 and 
 \[ 0 \to \Tor(A,B) \to I^2_A\ot B
        \xra{j_A\ot 1} I_A\ot B \xra{q_A\ot 1} A\ot B \to 0.
 \]
\end{lemma}
\begin{proof}
 Let $J$ and $K$ be as in Definition~\ref{defn-tensor}.  As
 $JK^2\leq JK$ and $J^2K^2\leq J^2K$ we have a map
 $f\:(JK^2)/(J^2K^2)\to(JK)/(J^2K)$ given by $f(x+J^2K^2)=x+J^2K$.
 The cokernel is $JK/(J^2K+JK^2)$, which is $A\ot B$ by definition.
 The kernel is $(JK^2\cap J^2K)/(J^2K^2)$, which is $\Tor(A,B)$ by
 definition.  In other words, we have an exact sequence
 \[ 0 \to \Tor(A,B) \to \frac{JK^2}{J^2K^2} \xra{f} 
       \frac{JK}{J^2K} \to A\ot B \to 0.
 \]
 Next, we have $\Z[A\tm B]=\Z[A]\ot\Z[B]$ by
 Corollary~\ref{cor-tensor-distrib}.  For $p,q\geq 0$ we have ideals
 $I_A^p\leq\Z[A]$ and $I_B^q\leq\Z[B]$.  These are free abelian groups
 by Theorem~\ref{thm-hereditary}, so using
 Proposition~\ref{prop-tensor-exact}(b) we see that the evident
 homomorphisms  
 \begin{center}
  \begin{tikzcd}
      I_A^p\ot I_B^q \arrow[r] \arrow[d] & I_A^p\ot\Z[B] \arrow[d] \\
      \Z[A]\ot I_B^q \arrow[r] & \Z[A]\ot\Z[B]
     \end{tikzcd}
    \end{center}
    
 are all injective.  This means that $I_A^p\ot I_B^q$ can be
 identified with its image in $\Z[A\tm B]$, which is just $J^pK^q$.
 Our exact sequence now takes the form 
 \[ 0 \to \Tor(A,B) \to
       \frac{I_A\ot I_B^2}{I_A^2\ot I_B^2} \xra{f} 
       \frac{I_A\ot I_B}{I_A^2\ot I_B} \to A\ot B \to 0.
 \]
 Next, as tensoring is right exact
 (Proposition~\ref{prop-tensor-exact}(a)) we can identify
 $(I_A\ot I_B^2)/(I_A^2\ot I_B^2)$ with $(I_A/I^2_A)\ot I_B^2$, which
 is $A\ot I_B^2$ by Proposition~\ref{prop-IA}.  Similarly, we can
 identify $(I_A\ot I_B^2)/(I_A^2\ot I_B^2)$ with $A\ot I_B$, so our
 exact sequence becomes
 \[ 0 \to \Tor(A,B) \to A\ot I_B^2
        \xra{} A\ot I_B \xra{} A\ot B \to 0 
 \]
 as claimed.  The other sequence is obtained symmetrically, or by
 appealing to Lemma~\ref{lem-tor-twist}.
\end{proof}

\begin{corollary}\label{cor-flat-tor}
 If $A$ or $B$ is torsion-free then $\Tor(A,B)=0$.  In particular,
 this holds if $A$ or $B$ is free.
\end{corollary}
\begin{proof}
 If $A$ is torsion-free then the homomorphism 
 \[ A\ot I_B^2\to A\ot I_B \]
 is injective by Proposition~\ref{prop-tensor-exact}(b), but
 $\Tor(A,B)$ is the kernel so $\Tor(A,B)=0$.  The other case follows
 symmetrically. 
\end{proof}

We next discuss the functorial properties of Tor groups.  Suppose we
have homomorphisms $f\:A\to A'$ and $g\:B\to B'$.  As discussed in
Remark~\ref{rem-tensor-functor}, these give a homomorphism 
$f\tm g\:A\tm B\to A'\tm B'$, which induces a ring map 
$(f\tm g)_\bullet\:\Z[A\tm B]\to\Z[A'\tm B']$, sending 
$J$ and $K$ to the coresponding ideals in $\Z[A'\tm B']$.  It
therefore induces a homomorphism $\Tor(A,B)\to\Tor(A',B')$, which we
denote by $\Tor(f,g)$.  This makes $\Tor(A,B)$ a functor of the pair
$(A,B)$. 

Now suppose we have two homomorphisms $f_0,f_1\:A\to A'$.  We then
have ring maps $(f_0)_\bullet$, $(f_1)_\bullet$ and
$(f_0+f_1)_\bullet$ from $\Z[A]$ to $\Z[A']$, and it is not true that
$(f_0+f_1)_\bullet=(f_0)_\bullet+(f_1)_\bullet$.  Because of this, it
is not obvious that $\Tor(f_0+f_1,g)=\Tor(f_0,g)+\Tor(f_1,g)$.
However, this does turn out to be true, as we now prove.

\begin{proposition}\label{prop-tor-bilinear}\nl
 \begin{itemize}
  \item[(a)] Given homomorphisms $f_0,f_1\:A\to A'$ and $g\:B\to B'$
   we have $\Tor(f_0+f_1,g)=\Tor(f_0,g)+\Tor(f_1,g)$.
  \item[(b)] Given homomorphisms $f\:A\to A'$ and $g_0,g_1\:B\to B'$
   we have $\Tor(f,g_0+g_1)=\Tor(f,g_0)+\Tor(f,g_1)$.
  \item[(c)] Given families of groups $\{A_i\}_{i\in I}$ and
   $\{B_j\}_{j\in J}$ there is a natural isomorphism
   \[ \bigoplus_{i,j} \Tor(A_i,B_j) \to
       \Tor\left(\bigoplus_i A_i,\bigoplus_jB_j\right).
   \]
 \end{itemize}
\end{proposition}
\begin{proof}
 Lemma~\ref{lem-tor-four} allows us to describe $\Tor(A,B)$ as the
 kernel of the map $1\ot j\:A\ot I_B^2\to A\ot I_B$, and claim~(a)
 follows easily from this.  Claim~(b) is proved similarly.

 For~(c), let us put $A^*=\bigoplus_iA_i$ and $B^*=\bigoplus_jB_j$ and
 $T=\bigoplus_{i,j}\Tor(A_i,A_j)$.
 We have inclusions $\iota_i\:A_i\to A^*$ and $\iota_j\:B_j\to B^*$,
 which give homomorphisms
 $\Tor(\iota_i,\iota_j)\:\Tor(A_i,B_j)\to\Tor(A^*,B^*)$.  We also have
 inclusions $\iota_{ij}\:\Tor(A_i,B_j)\to T$.  By the
 universal property of coproducts
 (Proposition~\ref{prop-categorical-coproduct}) there is a unique
 homomorphism $\phi\:T\to\Tor(A^*,B^*)$ with
 $\phi\circ\iota_{ij}=\Tor(\iota_i,\iota_j)$ for all $i$ and $j$.  It
 is this map that we claim is an isomorphism.  

 For fixed $j$, we can identify $\Tor(A^*,B_j)$ with the kernel of the
 evident map $A^*\ot I^2_{B_j}\to A^*\ot I_{B_j}$.  From this it
 follows easily that $\Tor(A^*,B_j)=\bigoplus_i\Tor(A_i,B_j)$.  A
 similar argument shows that $\Tor(A^*,B^*)=\bigoplus_j\Tor(A^*,B_j)$,
 and by putting these together we obtain $\Tor(A^*,B^*)\simeq T$ as
 claimed.  We leave it to the reader to check that the isomorphism
 arising from this argument is the same as $\phi$. 
\end{proof}

\begin{proposition}\label{prop-tor-six}
 For any abelian group $U$ and any short exact sequence $A\to B\to C$
 there is a natural exact sequence 
 \[ 0 \to \Tor(U,A) \to \Tor(U,B) \to \Tor(U,C) \to 
      U\ot A \to U\ot B \to U \ot C \to 0.
 \]
\end{proposition}
\begin{proof}
 Apply the Snake Lemma (Proposition~\ref{prop-snake-lemma}) to the
 diagram 
 \begin{center}
  \begin{tikzcd}
     I^2_U \ot A \arrow[r] \arrow[d] & 
     I^2_U \ot B \arrow[r] \arrow[d] & 
     I^2_U \ot C        \arrow[d] \\
     I_U \ot A \arrow[r] & 
     I_U \ot B \arrow[r] & 
     I_U \ot C,
    \end{tikzcd}
   \end{center}
   
 noting that the rows are short exact because $I^2_U$ and $I_U$ are
 free.  
\end{proof}

\begin{remark}\label{rem-traditional-tor}
 Let $A$ and $B$ be abelian groups.  We can always choose a free
 abelian group $F$ and a surjective homomorphism $p\:F\to B$.  Indeed,
 one possibility is to use the natural map $q\:I_B\to B$ from
 Proposition~\ref{prop-IA}, but we can also make a less natural
 choice that is typically much smaller.  We then let $F'$ denote the
 kernel of $p$, and note that this is again free by
 Theorem~\ref{thm-hereditary}.  Now the Proposition gives an
 exact sequence 
 \[ 0 \to \Tor(A,F') \to \Tor(A,F) \to \Tor(A,B) \to 
     A\ot F' \xra{1\ot i} A\ot F \to A\ot B \to 0.
 \]
 As $F'$ and $F$ are free we have $\Tor(A,F')=\Tor(A,F)=0$, so we
 conclude that $\Tor(A,B)$ is isomorphic to the kernel of $1\ot i$.  

 A more traditional approach to $\Tor$ groups is to \emph{define}
 $\Tor(A,B)$ as the kernel of $1\ot i$.  In this approach one has to
 work to prove that the resulting group is well-defined up to
 canonical isomorphism, and that $\Tor(A,B)\simeq\Tor(B,A)$ and so
 on.  However, one makes more contact with the general techniques of
 homological algebra, which are important for other reasons.
\end{remark}

\begin{remark}\label{rem-tor-mono}
 Suppose we have subgroups $A'\leq A$ and $B'\leq B$.  Using
 Proposition~\ref{prop-tor-six} we see that the maps
 \begin{center}
  \begin{tikzcd}
     \Tor(A',B') \arrow[r] \arrow[d] & \Tor(A',B) \arrow[d] \\
     \Tor(A,B') \arrow[r] & \Tor(A,B)
    \end{tikzcd}
   \end{center}
   
 are all injective, so we can regard $\Tor(A',B')$ as a subgroup of
 $\Tor(A,B)$.  
\end{remark}

\begin{proposition}\label{prop-tor-finitely-determined}
 $\Tor(A,B)$ is the union of the subgroups $\Tor(A',B')$ for finite
 subgroups $A'\leq A$ and $B'\leq B$.
\end{proposition}
\begin{proof}
 Any element of $x\in\Tor(A,B)$ has the form $x=y+z+J^2K^2$ for some
 $y\in J^2K$ and $z\in JK^2$.  We can write $y$ as
 $\sum_{i=1}^rn_i\ip{a_i}\ip{a'_i}\ip{b_i}$ for some $n_i\in\Z$ and
 $a_i,a'_i\in A$ and $b_i\in B$.  Similarly, we can write $z$ as
 $\sum_{j=1}^sm_j\ip{c_j}\ip{d_j}\ip{d'_j}$ for some $m_j\in\Z$ and
 $c_j\in A$ and $d_j,d'_j\in B$.  Let $A_0$ be the subgroup of $A$
 generated by the elements $a_i,a'_i$ and $c_j$, and let $B_0$ be the
 subgroup of $B$ generated by the elements $b_i,d_j$ and $d'_j$.  It
 is then clear that $x\in\Tor(A_0,B_0)\leq\Tor(A,B)$.  Moreover, $A_0$
 and $B_0$ are finitely generated, so the subgroups $A'=\tors(A_0)$
 and $B'=\tors(B_0)$ are finite.  It follows from
 Theorem~\ref{thm-fin-gen} that $A_0=A'\oplus P$ and
 $B_0=B'\oplus Q$, where $P$ and $Q$ are free.  This means that
 $\Tor(P,B')=\Tor(P,Q)=\Tor(A',Q)=0$ by Corollary~\ref{cor-flat-tor},
 so $\Tor(A_0,B_0)=\Tor(A',B')$.  The claim follows.
\end{proof}

\begin{corollary}\label{cor-tor-torsion}
 $\Tor(A,B)$ is always a torsion group.
\end{corollary}
\begin{proof}
 In view of the proposition, it will suffice to prove this when $A$
 is finite.  In that case there exists $n$ such that $n.1_A=0$, and
 so $\Tor(n.1_A,1_B)=0\:\Tor(A,B)\to\Tor(A,B)$.  However, using
 Proposition~\ref{prop-tor-bilinear} we see that
 $\Tor(n.1_A,1_B)=n.\Tor(1_A,1_B)=n.1_{\Tor(A,B)}$.  This means that
 $nx=0$ for all $x\in\Tor(A,B)$, so $\Tor(A,B)$ is a torsion group.
\end{proof}

We next show how to construct elements of $\Tor(A,B)$ more
explicitly.  

\begin{proposition}\label{prop-en}\nl
 \begin{itemize}
  \item[(a)] For $n>0$ and $a\in A[n]$ and $b\in B[n]$ there is an
   element $e_n(a,b)\in\Tor(A,B)$ given by 
   \[ e_n(a,b) = n\ip{a}\ip{b} + J^2K^2 \in 
       \frac{J^2K\cap JK^2}{J^2K^2} = \Tor(A,B).
   \]
  \item[(b)] The map $e_n\:A[n]\tm B[n]\to\Tor(A,B)$ is bilinear, so
   it induces a homomorphism $A[n]\ot B[n]\to\Tor(A,B)$.
  \item[(c)] Suppose we have elements $a\in A[n]$ and $b\in B[m]$.
   Let $d$ be any common divisor of $n$ and $m$.  Then
   $(n/d)a\in A[m]$ and $(m/d)b\in A[n]$ and we have
   \[ e_n(a,(m/d)b) = e_{nm/d}(a,b) = e_m((n/d)a,b). \]
  \item[(d)] Suppose we have a short exact sequence 
   \[ 0 \to A \xra{j} B \xra{q} C \to 0 \]
   giving rise to an exact sequence 
   \[ 0 \to \Tor(U,A) \to \Tor(U,B) \to \Tor(U,C) \xra{\dl}
        U\ot A \to U\ot B \to U \ot C \to 0
   \]
   as in Proposition~\ref{prop-tor-six}.  Suppose that $u\in U[n]$ and
   $c\in C[n]$.  Then $\dl(e_n(u,c))$ can be described as follows: we
   choose $b\in B$ with $q(b)=c$, then there is a unique $a\in A$ with
   $j(a)=nb$, and $\dl(e_n(u,c))=u\ot a$.
 \end{itemize}
\end{proposition}
\begin{proof}\nl
 \begin{itemize}
  \item[(a)] First, as $na=0$ in $A\simeq I_A/I_A^2$ we see that
   $n\ip{a}\in I_A^2$ and so $n\ip{a}\ip{b}\in I_A^2\ot I_B=J^2K$.
   Similarly we have $n\ip{b}\in I_B^2$ so
   $n\ip{a}\ip{b}=\ip{a}.(n\ip{b})\in I_A\ot I_B^2=JK^2$.  Thus, we
   have $n\ip{a}\ip{b}\in J^2K\cap JK^2$ and the definition of $e_n$
   is meaningful.
  \item[(b)] Next, recall that
   $\ip{a+a'}-\ip{a}-\ip{a'}=\ip{a}\ip{a'}\in I_A^2$, so 
   \[ n\ip{a+a'}\ip{b} - n\ip{a}\ip{b} - n\ip{a'}\ip{b} = 
       (\ip{a}\ip{a'}).(n\ip{b}) \in I_A^2\ot I_B^2=J^2K^2,
   \]
   so $e_n(a+a',b)=e_n(a,b)+e_n(a',b)$.  By a symmetrical argument we
   see that $e_n(a,b)$ is also linear in $b$, as required.
  \item[(c)] Suppose we have elements $a\in A[n]$ and $b\in B[m]$.
   Let $d$ be any common divisor of $n$ and $m$, so $n=pd$ and $m=qd$
   for some $p,q>0$.  Now $m.(n/d)a=pqda=q.na=0$, so
   $(n/d)a=pa\in A[m]$, and similarly $(m/d)b=qb\in B[m]$.  Next note
   that 
   \[ pd\ip{a}\ip{qb}-pqd\ip{a}\ip{b} = (n\ip{a})(\ip{qb}-q\ip{b}) \in 
       I_A^2\ot I_B^2 = J^2K^2,
   \]
   so $e_{pd}(a,qb)=e_{pqd}(a,b)$, or $e_n(a,(m/d)b)=e_{nm/d}(a,b)$.
   By a symmetrical argument, we also have
   $e_{nm/d}(a,b)=e_m((n/d)a,b)$. 
  \item[(d)] For $u\in U[n]$ and $c\in C[n]$ we note that
   $n\ip{u}\in I_U^2$ so we can put
   $e'_n(u,c)=(n\ip{u})\ot c\in I_U^2\ot C$.  Note that it is not
   legitimate to rewrite this as $\ip{u}\ot nc$ because
   $\ip{u}\not\in I_U^2$.  However this rewriting becomes legitimate
   if we work in $I_U\ot C$, and the result is zero because $nc=0$.
   In other words, $e'_n(u,c)$ lies in the kernel of the map
   $I_U^2\ot A\to I_U\ot A$, which was identified with $\Tor(U,A)$ in
   Lemma~\ref{lem-tor-four}. It is not hard to check that $e'_n(u,c)$
   corresponds to $e_n(u,c)$ under this identification.  Now consider
   the diagram 
   \begin{center}
    \begin{tikzcd}
       I^2_U \ot A \arrow[r,"1\ot j"] \arrow[d,"i\ot 1"'] & 
       I^2_U \ot B \arrow[r,"1\ot q"] \arrow[d,"i\ot 1"'] & 
       I^2_U \ot C        \arrow[d,"i\ot 1"] \\
       I_U \ot A \arrow[r,"1\ot j"'] & 
       I_U \ot B \arrow[r,"1\ot q"'] & 
       I_U \ot C.
      \end{tikzcd}
     \end{center}
     
   By inspecting the proof of Proposition~\ref{prop-snake-lemma} we
   find the following prescription for the connecting map
   $\dl\:\Tor(U,C)\to A\ot C$.  Given $z\in I^2_U\ot C$ with
   $(i\ot 1)(z)=0$ we choose $y\in I^2_U\ot B$ with $(1\ot q)(y)=z$,
   then we check that $(i\ot 1)(y)\in\ker(1\ot q)=\image(1\ot j)$ so
   there exists $x\in I_U\ot A$ with $(1\ot j)(x)=(i\ot 1)(y)$, and
   the image of $x$ in $U\ot A=(I_U/I_U^2)\ot A$ is by definition
   $\dl(z)$.  Now suppose again that we are given $u\in U[n]$ and
   $c\in C[n]$ and take $z=e'_n(u,c)$.  As $q$ is surjective we can
   choose $b\in B$ with $q(b)=a$.  We then have $q(nb)=na=0$ so
   $nb\in\ker(q)=\image(j)$, so there exists $a\in A$ with $j(a)=nb$.
   We can thus put $y=(n\ip{u})\ot b\in I^2_U\ot B$ and
   $x=\ip{u}\ot a\in I_U\ot A$, and we find that $(1\ot q)(y)=z$ and
   $(i\ot 1)(y)=\ip{u}\ot nb=(1\ot j)(\ip{u}\ot a)$.  This means that
   $\dl(e_n(u,a))$ is the image of $\ip{u}\ot a$ in $U\ot A$, which is
   just $u\ot a$ as claimed.
 \end{itemize}
\end{proof}

\begin{corollary}\label{cor-tor-Zn}
 For any abelian group $U$ and $n>0$, the map $u\mapsto e_n(u,1+n\Z)$
 gives an isomorphism $U[n]\mapsto\Tor(U,\Z/n)$.
\end{corollary}
\begin{proof}
 Consider the short exact sequence $\Z\xra{n}\Z\to\Z/n$.  As
 $\Tor(U,\Z)=0$ and $U\ot\Z=\Z$ the six-term exact sequence reduces to
 \[ 0 \to \Tor(U,\Z) \xra{\dl} U \xra{n} U \xra{} U/n \xra{} 0, \]
 which means that $\dl$ gives an isomorphism
 $\dl\:\Tor(U,\Z)\to U[n]$.  Part~(d) of the proposition tells us
 that $\dl(e_n(u,1+n\Z))=u$ for all $u\in U[n]$, and the claim follows
 from this. 
\end{proof}

\begin{corollary}\label{cor-en-generate}
 For any abelian groups $A$ and $B$, the group $\Tor(A,B)$ is
 generated by all the elements $e_n(a,b)$ for $n>0$ and $a\in A[n]$
 and $b\in B[n]$.
\end{corollary}
\begin{proof}
 In view of Proposition~\ref{prop-tor-finitely-determined}, it is
 enough to prove this when $A$ and $B$ are finite.  Now $A$ and $B$
 can be written as direct sums of finite cyclic groups, and using
 Proposition~\ref{prop-tor-bilinear}(c) we reduce to the case where
 $A$ and $B$ are themselves cyclic.  That case is immediate from
 Corollary~\ref{cor-tor-Zn}.  
\end{proof}

\begin{corollary}\label{cor-tor-QZ}
 For any abelian group $B$, there is a natural isomorphism
 $\Tor(\QZ,B)\simeq\tors(B)$. 
\end{corollary}
\begin{proof}
 Let $A_m$ be the subgroup of $\QZ$ generated by $1/m+\Z$, and
 define $\ep_m\:B[m]\to\Tor(A_m,B)\leq\Tor(\QZ,B)$ by
 $\ep_m(b)=e_n(1/m+\Z,b)$.  Note that $A_m$ is cyclic of order $m$, so
 the previous result tells us that $\ep_m$ is an isomorphism.  Now
 suppose that $m$ divides $n$, so $B[m]\leq B[n]$.  We can take $d=m$
 in Proposition~\ref{prop-en}(c) to see that $e_n(a,b)=e_m((n/m)a,b)$
 whenever $na=0$ and $mb=0$.  Taking $a=1/n+\Z$ we deduce that
 $\ep_n(b)=\ep_m(b)$ whenever $mb=0$, so $\ep_n|_{B[m]}=\ep_m$.  It
 follows that there is a unique homomorphism 
 \[ \ep \: \tors(B) = \bigcup_{n>0} B[n] \to \Tor(\QZ,B). \]
 Every finite subgroup of $\QZ$ has the form $A_n$ for some $n$, and
 it follows from this using
 Proposition~\ref{prop-tor-finitely-determined} that $\ep$ is an
 isomorphism.  

 There is an explicit construction of the inverse which is quite
 instructive.  An element $x\in\QZ$ is a coset of $\Z$ in $\Q$, and
 any such coset intersects the interval $[0,1)$ in a single point,
 which we call $\lm(x)$.  The function $\lm\:\QZ\to\Q$ is not a
 homomorphism, but $\lm(x+x')$ and $\lm(x)+\lm(x')$ are both
 representatives of the same coset $x+x'$, so we at least have
 \[ \lm(x+x') - \lm(x) - \lm(x') \in \Z. \] We can extend $\lm$ to
 give a group homomorphism $\Z[\QZ]\to\Q$ by
 $\lm(\sum_in_i[x_i])=\sum_in_i\lm(x_i)$.  We then find that 
 \[ \lm(\ip{x}\ip{x'}) = \lm([x+x']-[x]-[x']+[0]) =
     \lm(x+x')-\lm(x)-\lm(x') \in \Z,
 \]
 and thus that $\lm(I_{\QZ}^2)\leq\Z$.  We therefore have a
 homomorphism $\lm\ot 1\:I_{\QZ}^2\ot B\to \Z\ot B=B$.
 Lemma~\ref{lem-tor-four} allows us to regard $\Tor(\QZ,B)$ as a
 subgroup of $I^2_{\QZ}\ot B$, so we can restrict to this subgroup
 to get a homomorphism $\mu\:\Tor(\QZ,B)\to B$.  We leave it to the
 reader to check that the image of $\mu$ is $\tors(B)$, and that
 $\mu\:\Tor(\QZ,B)\to\tors(B)$ is inverse to $\ep$.
\end{proof}

\begin{proposition}\label{prop-tor-coeq}
 Define maps
 \begin{center}
  \begin{tikzcd}
   \bigoplus_{n,m,d}(A[nd]\ot B[md])
    \arrow[r,shift left=0.1cm,"\lm"]
    \arrow[r,shift right=0.1cm,"\rho"'] &
    \bigoplus_p(A[p]\ot B[p]) 
     \arrow[r,"\ep"] &
    \Tor(A,B)
   \end{tikzcd}
  \end{center}
  
 by 
 \begin{align*}
   \lm(i_{n,m,d}(a\ot b)) &= i_{nd}(a\ot mb) \\
  \rho(i_{n,m,d}(a\ot b)) &= i_{md}(na\ot b) \\
  \ep(i_p(a\ot b)) &= e_p(a,b).
 \end{align*}
 Then the sequence
 \begin{center}
  \begin{tikzcd}
    \bigoplus_{n,m,d}(A[nd]\ot B[md]) 
     \arrow[r,"\lm-\rho"] &
    \bigoplus_p(A[p]\ot B[p]) 
     \arrow[r,"\ep"] &
    \Tor(A,B) \arrow[r] & 
    0
   \end{tikzcd}
  \end{center}
  
 is exact, so $\Tor(A,B)$ is the cokernel of $\lm-\rho$.
\end{proposition}
\begin{proof}
 Let $T(A,B)$ be the cokernel of $\lm-\rho$, and let $e'_p(a,b)$ be
 the image of $i_p(a\ot b)$ in $T(A,B)$, so by construction we have
 $e'_{nd}(a,mb)=e'_{md}(na,b)$ whenever $nda=0$ and $mdb=0$.  Part~(c)
 of Proposition~\ref{prop-en} shows that $\ep\lm=\ep\rho$, so
 $\img(\lm-\rho)\leq\ker(\ep)$, so there is a unique homomorphism
 $\ov{\ep}\:T(A,B)\to\Tor(A,B)$ such that
 $\ov{\ep}(e'_p(a,b))=e_p(a,b)$ for all $p\geq 1$ and $a\in A[p]$ and
 $b\in B[p]$.  This is surjective by Corollary~\ref{cor-en-generate}.
 We must show that this is also injective.

 Consider the special case where $A$ and $B$ are finitely generated.
 It is easy to see that there are natural splittings
 \begin{align*}
  T(A\oplus A',B\oplus B') &\simeq 
   T(A,B)\oplus T(A,B')\oplus T(A',B)\oplus T(A',B') \\
  \Tor(A\oplus A',B\oplus B') &\simeq 
   \Tor(A,B)\oplus \Tor(A,B')\oplus \Tor(A',B)\oplus \Tor(A',B')
 \end{align*}
 so we can reduce to the case where $A$ and $B$ are cyclic.  If $A=\Z$
 or $B=\Z$ it is clear that $T(A,B)=0=\Tor(A,B)$, so we may assume
 that $A=\Z/r$ and $B=\Z/s$ say.  Let $t$ be the least common multiple
 of $r$ and $s$, so we have an element $u=e'_t(1+r\Z,1+s\Z)\in T(A,B)$.
 Note that $T(A,B)$ is generated by elements of the form
 $v=e_p(a+r\Z,b+s\Z)$ with $pa\in r\Z$ and $pb\in s\Z$, which implies
 that $pab=tc$ for some $c\in\Z$.  We thus have an element 
 \[ x = i_{c,1,t}(1+r\Z,1+s\Z) + i_{1,b,ap}(1+r\Z,1+s\Z)
        - i_{a,1,p}(1+r\Z,b+s\Z) \in
     \bigoplus_{n,m,d}(A[nd]\ot B[md]),
 \]
 and we find that 
 \[ (\lm-\rho)(x) = i_p((a+r\Z)\ot(b+s\Z)) - i_t(c+r\Z,1+s\Z) = 
      i_p((a+r\Z)\ot(b+s\Z)) - c\,i_t(1+r\Z,1+s\Z)
 \]
 so $v=cu$.  This proves that $u$ generates $T(A,B)$.  It is also
 clear that $ru=su=0$, so if we put $d=(r,s)$ we have $du=0$.  This
 means that $T(A,B)$ is cyclic of order dividing $d$ but the map
 $T(A,B)\to\Tor(A,B)\simeq\Z/d$ is surjective so it must be an
 isomorphism.  

 We now revert to the general case, where $A$ and $B$ may be
 infinitely generated.  Consider an element
 $w\in\ker(\ov{\ep})\leq T(A,B)$.  We can write $w$ as
 $\sum_{i=1}^Ne'_{p_i}(a_i,b_i)$ say, and then let $A'$ be the
 subgroup of $A$ generated by $a_1,\dotsc,a_N$, and let $B'$ be the
 subgroup of $B$ generated by $b_1,\dotsc,b_N$.  There is then an
 evident element $w'\in T(A',B')$ that maps to $w$ in $T(A,B)$.  Now
 consider the diagram
 \begin{center}
  \begin{tikzcd}
   T(A',B')
    \arrow[r,"\ov{\ep}","\simeq"']
    \arrow[d] &
   \Tor(A',B') \arrow[rightarrowtail,d] \\
   T(A,B) \arrow[twoheadrightarrow,r,"\ov{\ep}"'] &
   \Tor(A,B).
  \end{tikzcd}
 \end{center}
 The top map is an isomorphism by the special case considered above,
 and the right hand map is injective by Remark~\ref{rem-tor-mono}.  By
 chasing $w'$ around the diagram we see that $w=0$.  Thus, the map
 $\ov{\ep}\:T(A,B)\to\Tor(A,B)$ is injective as required.
\end{proof}

\section{Ext groups}
\label{sec-ext}

\begin{definition}\label{defn-ext}
 Let $A$ and $B$ be abelian groups, and let $j$ be the inclusion
 $I^2_A\to I_A$.  We define $\Ext(A,B)$ to be the kernel of the map
 $j^*\:\Hom(I_A,B)\to\Hom(I^2_A,B)$.  Now suppose we have
 homomorphisms $f\:A'\to A$ and $g\:B\to B'$.  We then have compatible
 homomorphisms $f_\bullet\:I_{A'}\to I_A$ and
 $f_\bullet\:I^2_{A'}\to I^2_A$, and we can use these in an evident
 way to construct a commutative square of maps
 \begin{center}
  \begin{tikzcd}
  \Ext(A,B) \arrow[r,"g_*"] \arrow[d,"f^*"'] & \Ext(A,B') \arrow[d,"f^*"] \\
  \Ext(A',B) \arrow[r,"g_*"'] & \Ext(A',B')
 \end{tikzcd}
\end{center}

\end{definition}
\begin{remark}\label{rem-ext-four}
 Let $q$ be the usual map $I_A\to A$, with kernel $I_A^2$.  Consider
 the sequence
 \[ 0 \to \Hom(A,B) \xra{q^*} \Hom(I_A,B) \xra{j^*}
       \Hom(I^2_A,B) \xra{} \Ext(A,B) \to 0.
 \]
 By combining Lemma~\ref{lem-right-test} with
 Definition~\ref{defn-ext}, we see that this is exact.
\end{remark}
\begin{remark}\label{rem-ext-additive}
 If we have two homomorphisms $g_0,g_1\:B\to B'$, it is clear that
 \[ (g_0+g_1)_* = g_{0*} + g_{1*} \: \Ext(A,B) \to \Ext(A,B'). \]
 If we have two homomorphisms $f_0,f_1\:A'\to A$, it is not obvious
 from the definitions that $(f_0+f_1)^*=f_0^*+f_1^*$, but later we
 will give a different description of the Ext groups that makes this
 fact clear.
\end{remark}

\begin{lemma}\label{lem-free-projective}
 Let $F$ be a free abelian group.  Then for every surjective
 homomorphism $q\:B\to C$, the resulting map
 $q_*\:\Hom(F,B)\to\Hom(F,C)$ is also surjective.  Equivalently, for
 every pair of homomorphisms $f$ and $q$ as shown with $q$ surjective,
 there exists $g$ with $qg=f$.
 \begin{center}
  \begin{tikzcd}
   & F \arrow[dotted,dl,"g"'] \arrow[d,"f"] \\
   B \arrow[twoheadrightarrow,r,"q"'] & C.
  \end{tikzcd}
 \end{center}
 Conversely, every group $F$ with this property is free.
\end{lemma}
\begin{proof}
 Firstly, there exists a map $g$ as shown if and only if the element
 $f\in\Hom(F,C)$ lies in the image of $q_*\:\Hom(F,B)\to\Hom(F,C)$; so
 the two versions of the statement are indeed equivalent.  To prove
 them, we may assume that $F=\Z[I]$ for some $I$.  We then have
 elements $f([i])\in C$ and $q\:B\to C$ is surjective so we can choose
 $b_i\in B$ with $q(b_i)=f([i])$.  Now there is a unique homomorphism
 $g\:\Z[I]\to B$ with $g([i])=b_i$ for all $i$, and $qg=f$ as
 required.  

 Conversely, let $F$ be any abelian group that has the property under
 consideration.  Take $C=F$ and $B=I_F$, let $q\:B\to C$ be the usual
 surjection $I_F\to F$, and let $f\:F\to C$ be the identity.  Then
 there must exist $g\:F\to I_F$ with $qg=1_F$.  This means that $g$ is
 injective, so $F$ is isomorphic to a subgroup of the free group
 $I_F$, so $F$ is free by Theorem~\ref{thm-hereditary}.
\end{proof}
\begin{remark}
 In the more general context of modules over an arbitrary ring, the
 property in the lemma is called \emph{projectivity}.  Thus, we have
 shown that an abelian group is projective if and only if it is free.
 The same argument shows that an $R$-module is projective if and only
 if it is a direct summand in a free $R$-module.  The analogue of
 Theorem~\ref{thm-hereditary} is not valid for modules over a general
 ring, so projective modules need not be free.
\end{remark}

\begin{proposition}\label{prop-ext-six}
 Let $U$ be an abelian group, and let $A\xra{j}B\xra{q}C$ be a short
 exact sequence of abelian groups.  Then there is a natural exact
 sequence 
 \[ 0 \to \Hom(U,A) \xra{j_*} \Hom(U,B) \xra{q_*} \Hom(U,C) 
    \xra{\dl} \Ext(U,A) \xra{j_*} \Ext(U,B) \xra{q_*} \Ext(U,C) \to 0.
 \] 
\end{proposition}
\begin{proof}
 We first claim that the sequence 
 \[ 0 \to \Hom(U,A) \xra{j_*} \Hom(U,B) \xra{q_*} \Hom(U,C) \]
 is  exact.  Indeed, if $j_*(\al)=0$ then $j(\al(u))=0$ for all $u$,
 but $j$ is injective so $\al(u)=0$ for all $u$, so $\al=0$; this
 proves that $j_*$ is injective.  Next, suppose that $q_*(\bt)=0$, so
 $q(\bt(u))=0$ for all $u\in U$, so $\bt(u)\in\ker(q)=\image(j)$, so
 there exists $\al(u)\in A$ with $\bt(u)=j(\al(u))$.  This element
 $\al(u)$ is in fact unique, because $j$ is injective.  We also have 
 \[ j(\al(u+u')-\al(u)-\al(u')) = 
    j(\al(u+u'))-j(\al(u))-j(\al(u')) = 
    \bt(u+u')-\bt(u)-\bt(u') = 0,
 \]
 but $j$ is injective so $\al(u+u')-\al(u)-\al(u')=0$, so the map
 $\al\:U\to A$ is a homomorphism.  Clearly $j_*(\al)=\bt$, so we see
 that $\ker(q_*)=\image(j_*)$ as required.

 Now consider the sequence
 \[ 0 \to \Hom(I_U,A) \xra{j_*} \Hom(I_U,B)
        \xra{q_*} \Hom(I_U,C) \to 0.
 \]
 As $U$ was arbitrary we can replace it by $I_U$ to see that $j_*$ is
 injective and $\image(j_*)=\ker(q_*)$.  As $I_U$ is free we also see
 from Lemma~\ref{lem-free-projective} that $q_*$ is surjective, so the
 sequence is short exact.  The same argument applies with $I_U$
 replaced by $I_U^2$. 

 Now consider the diagram 
 \begin{center}
  \begin{tikzcd}
   \Hom(I_U,A) \arrow[d] \arrow[rightarrowtail,r,"j_*"] &
   \Hom(I_U,B) \arrow[d] \arrow[twoheadrightarrow,r,"q_*"] &
   \Hom(I_U,C) \arrow[d] \\
   \Hom(I^2_U,A) \arrow[rightarrowtail,r,"j_*"'] &
   \Hom(I^2_U,B) \arrow[twoheadrightarrow,r,"q_*"'] &
   \Hom(I^2_U,C). 
  \end{tikzcd}
 \end{center}  
 We have just seen that the rows are short exact, so we can apply the
 Snake Lemma (Proposition~\ref{prop-snake-lemma}) to get a six-term
 exact sequence involving the kernels and cokernels of the vertical
 maps.  Remark~\ref{rem-ext-four} identifies these kernels and
 cokernels as $\Hom$ and $\Ext$ groups, as required.
\end{proof}

\begin{corollary}\label{cor-free-ext}
 An abelian group $F$ is free if and only if $\Ext(F,A)=0$ for all
 $A$. 
\end{corollary}
\begin{proof}
 First suppose that $F$ is free.  By applying
 Lemma~\ref{lem-free-projective} to the diagram 
 \begin{center}
  \begin{tikzcd}
   & F \arrow[dotted,dl,"s"'] \arrow[d,"1"] \\
   I_F \arrow[twoheadrightarrow,r,"q"'] & F
  \end{tikzcd}
 \end{center}  
 we obtain a homomorphism $s\:F\to I_F$ with $qs=1_F$.  This gives a
 splitting in the usual way, so there is a unique map
 $r\:I_F\to\ker(q)=I_F^2$ with $jr=1-sq$ and we have $rj=1_{I^2_F}$.
 Now for any $A$ we have homomorphisms 
 \[ \Hom(I^2_F,A) \xra{r^*} \Hom(I_F,A) \xra{j^*} \Hom(I^2_F,A) \]
 with $j^*r^*=(rj)^*=1^*=1$.  This implies that $j^*$ is surjective,
 so the cokernel is zero.  But $\Ext(F,A)$ is defined to be
 $\cok(j^*)$, so $\Ext(F,A)=0$ as claimed.

 Conversely, suppose that $\Ext(F,A)=0$ for all $A$.  Let $q\:B\to C$
 be a surjective homomorphism.  If we put $A=\ker(q)$, then
 Proposition~\ref{prop-ext-six} gives an exact sequence
 \[ 0 \to \Hom(U,A) \xra{j_*} \Hom(U,B) \xra{q_*} \Hom(U,C) 
    \xra{\dl} 0 \xra{j_*} 0 \xra{q_*} 0 \to 0.
 \] 
 From this we deduce that $q_*$ is surjective.  By the last part of
 Lemma~\ref{lem-free-projective}, this means that $F$ is free.
\end{proof}

To complete our study of $\Ext$ groups, we will need to understand
groups $D$ with the ``dual'' property that $\Ext(A,D)=0$ for all $A$.
These will turn out to be the divisible groups, as defined below.

\begin{definition}\label{defn-divisible}
 Let $V$ be an abelian group.  We say that $V$ is \emph{divisible} if
 for all integers $n>0$ and all $v\in V$ there exists $u\in V$ with
 $nu=v$.  Equivalently, $V$ is divisible iff all the maps
 $n.1_V\:V\to V$ (for $n>0$) are surjective.
\end{definition}
\begin{example}\label{eg-divisible}
 The groups $\Q$, $\R$ and $\QZ$ are all divisible.  The only
 divisible finite group is the trivial group.
\end{example}
\begin{remark}\label{rem-divisible-quotient}
 It is clear that any quotient of a divisible group is divisible.
\end{remark}
\begin{remark}\label{rem-division-system}
 If $A$ is divisible, then (using the Axiom of Choice) we can choose
 functions $d_n\:A\to A$ for $n>0$ such that $n.d_n(a)=a$ for all $n$
 and $a$ (so in particular $d_1(a)=a$), and we can also ensure that
 $d_n(0)=0$ for all $n$.  Of course, $d_n$ need not be a homomorphism.
 We will call such a system of maps a \emph{division system} for $A$.
 In many cases one can make an explicit choice for $d_n$.  For $\Q$ or
 $\R$ we just have $d_n(a)=a/n$.  For $A=\Q/\Z$, every element has a
 unique representation as $a=x+\Z$ with $0\leq x<1$, and we put 
 $d_n(a)=x/n+\Z$.
\end{remark}

\begin{proposition}\label{prop-divisible-injective}
 Let $D$ be a divisible group.  Then for any injective homomorphism
 $j\:A\to B$, the resulting homomorphism $j^*\:\Hom(B,D)\to\Hom(A,D)$
 is surjective.  Equivalently, given homomorphisms $j$ and $f$ as
 shown below, there exists $g\:B\to D$ with $gj=f$.
 \begin{center}
  \begin{tikzcd}
   A \arrow[rightarrowtail,r,"j"] \arrow[d,"f"'] &
   B \arrow[dotted,dl,"g"] \\
   D.
  \end{tikzcd}
 \end{center}
 Conversely, any group $D$ that has this property is divisible.
\end{proposition}
\begin{proof}
 Firstly, there exists a map $g$ as shown if and only if the element
 $f\in\Hom(A,D)$ lies in the image of $j^*\:\Hom(B,D)\to\Hom(A,D)$; so
 the two versions of the statement are indeed equivalent.  

 Next, we claim that if $D$ has the above extension property then it
 is divisible.  This is immediate from the special case of the
 extension property where $A=B=\Z$ and $j=n.1_{\Z}$ for some $n>0$.
 
 For the main part of the proof, suppose that $D$ is divisible, and
 that we are given $j$ and $f$ as shown.  It will be harmless to
 replace $A$ by the isomorphic group $j(A)\leq B$ and so assume that
 $A\leq B$ and that $j$ is just the inclusion.  We then need to find
 $g\:B\to D$ with $g|_A=f$.  For this we choose a division system
 $(d_n)_{n>0}$ for $D$ and a well-ordering of $B$.  For $b\in B$ we
 let $B_{<b}$ denote the subgroup generated by
 $A\cup\{x\in B\st x<b\}$, and similarly for $B_{\leq b}$.  We then
 have $\{k\in\Z\st kb\in B_{<b}\}=\Z.n_b$ for some $n_b\geq 0$.  Note
 that if $n_b=0$ then $B_{\leq b}=B_{<b}\oplus\Z$, with the $\Z$
 summand generated by $b$.  We say that a homomorphism
 $g\:B_{\leq b}\to D$ is \emph{admissible} if
 \begin{itemize}
  \item $g|_A=f$
  \item Whenever $x\leq b$ with $n_x=0$ we have $g(x)=0$
  \item Whenever $x\leq b$ with $n_x>0$ we have
   $g(x)=d_{n_x}(g(n_xx))$. 
 \end{itemize}
 We claim that for all $b$ there is a unique admissible homomorphism
 $g_b\:B_{\leq b}\to D$.  If not, let $b$ be the least element for
 which this is false (which is meaningful because $B$ is
 well-ordered).  For all $x<b$, we have a unique admissible map
 $g_x\:B_{\leq x}\to D$.  By uniqueness, we see that $g_x$ agrees with
 $g_y$ on $B_{\leq y}$ whenever $y\leq x<b$.  It follows that the maps
 $g_x$ can be combined to give a map $g'\:B_{<b}\to D$.  If $n_b=0$,
 we find that there is a unique extension $g_b\:B_{\leq b}\to D$
 satisfying $g_b(u+kb)=g'(u)$ for all $u\in B_{<b}$ and $k\in\Z$.  If
 $n_b>0$, then we have an element $z=d_{n_b}(g'(n_bb))\in D$ with
 $nz=g'(n_bb)$, and we see that there is a unique extension
 $g_b\:B_{\leq b}\to D$ satisfying $g_b(u+kb)=g'(u)+kz$ for all $u\in
 B_{<b}$ and $k\in\Z$.  Either way, we find that $g_b$ is the unique
 admissible map on $B_{\leq b}$, contrary to assumption.  Thus, we
 have $g_b$ for all $b$, and we again see that $g_b$ agrees with $g_c$
 on $B_{\leq c}$ whenever $c\leq b$, so the maps $g_b$ fit together to
 give a homomorphism $g\:B\to D$.  This clearly satisfies $g|_A=f$ as
 required. 
\end{proof}
\begin{corollary}\label{cor-divisible-injective}
 If $D$ is divisible, then $\Ext(A,D)=0$ for all $A$.
\end{corollary}
\begin{remark}
 The converse statement is also true, but it will be more convenient
 to prove that later.
\end{remark}
\begin{proof}
 $\Ext(A,D)$ is by definition the cokernel of
 $j^*\:\Hom(I_A,D)\to\Hom(I^2_A,D)$, but $j^*$ is surjective by the
 proposition. 
\end{proof}

We next want to show that every abelian group can be embedded in a
divisible group.  This will be proved after some preliminaries.

\begin{proposition}\label{prop-QZ-cogen}
 Let $A$ be an abelian group, and let $a$ be a nontrivial element of
 $A$.  Then there is a homomorphism $f\:A\to\QZ$ with $f(a)\neq 0$. 
\end{proposition}
\begin{proof}
 Put $A_0=\Z a\leq A$.  If $a$ has infinite order then $A_0\simeq\Z$
 and we can define $f_0\:A_0\to\QZ$ by $f_0(ka)=k/2+\Z$.  If $a$ has
 finite order $n$ we can instead define $f_0\:A_0\to\QZ$ by
 $f_0(ka)=k/n+\Z$.  Either way we have $f_0(a)\neq 0$.  Next, as
 $\QZ$ is divisible, Proposition~\ref{prop-divisible-injective}
 tells us that the restriction map $\Hom(A,\QZ)\to\Hom(A_0,\QZ)$ is
 surjective, so we can choose $f\:A\to\QZ$ with $f|_{A_0}=f_0$.  In
 particular, this means that $f(a)\neq 0$, as required.
\end{proof}

\begin{lemma}\label{lem-MapIQZ}
 For any set $I$, the group $\Map(I,\QZ)$ is divisible.
\end{lemma}
\begin{proof}
 Let $(d_n)_{n>0}$ be the standard division system for $\QZ$ as in
 Remark~\ref{rem-division-system}.   The maps $u\mapsto d_n\circ u$
 then give a division system for $\Map(I,\QZ)$.
\end{proof}

Now suppose we have a family of homomorphisms $f_i\:A\to\QZ$ for $i\in
I$.  We can combine them to give a single homomorphism
$j\:A\to\Map(I,\QZ)$ by the rule $j(a)(i)=f_i(a)$.  Note that the
kernel of $j$ is the intersection of the kernels of all the
homomorphisms $f_i$.  Thus, if the family is large enough we can hope
that $j$ will be injective.

The most canonical thing to do is to consider the family of \emph{all}
homomorphisms $f\:A\to\QZ$, and thus to take $I=\Hom(A,\QZ)$.  This
leads us to the following definitions.

\begin{definition}\label{defn-EA}
 We put $E_A=\Map(\Hom(A,\QZ),\QZ)$, and define $j\:A\to E_A$ by the
 rather tautological rule $j(a)(f)=f(a)$.  We write $E^2_A$ for the
 cokernel of $j$, and $q$ for the quotient map $E_A\to E^2_A$.
\end{definition}

\begin{proposition}\label{prop-EA}
 The groups $E_A$ and $E^2_A$ are divisible, and the sequence 
 \[ A \xra{j} E_A \xra{q} E^2_A \]
 is short exact.
\end{proposition}
\begin{proof}
 The group $E_A$ is divisible by Lemma~\ref{lem-MapIQZ}, and $E^2_A$
 is a quotient of $E_A$ so it is also divisible.  It is clear by
 construction that $q$ is surjective with $\image(j)=\ker(q)$.
 Finally, if $a\in A$ is nonzero then Proposition~\ref{prop-QZ-cogen}
 gives us a homomorphism $f\in\Hom(A,\QZ)$ with $f(a)\neq 0$ or
 equivalently $j(a)(f)\neq 0$, so $j(a)\neq 0$.  This shows that $j$
 is injective, so the sequence is short exact as claimed.
\end{proof}

\begin{remark}\label{rem-double-dual}
 Note that for $a\in A$ and $f,g\in\Hom(A,\QZ)$ we have 
 \[ j(a)(f+g) = (f+g)(a) = f(a) + g(a) = j(a)(f) + j(a)(g), \]
 so the map $j(a)\:\Hom(A,\QZ)\to\QZ$ is a homomorphism.  In other
 words, $j$ can be regarded as an injective homomorphism
 \[ j\: A \to \Hom(\Hom(A,\QZ),\QZ). \]
 If we use the briefer notation $A^*$ for $\Hom(A,\QZ)$, then
 $j\:A\to A^{**}\leq E_A$.  The group $A^{**}$ is usually not
 divisible, so $E_A$ is more useful for our immediate applications to
 $\Ext$ groups.  However, the group $A^{**}$ will reappear later in
 other contexts.
\end{remark}

\begin{corollary}\label{cor-ext-four-b}
 For all $A$ and $B$, there is a natural exact sequence
 \[ 0 \to \Hom(A,B) \xra{j_*} \Hom(A,E_B) \xra{q_*}
       \Hom(A,E^2_B) \to \Ext(A,B) \to 0.
 \]
\end{corollary}
\begin{proof}
 Apply Proposition~\ref{prop-ext-six} to the sequence
 $B\to E_B\to E^2_B$, noting that $\Ext(A,E_B)=\Ext(A,E^2_B)=0$ by
 Corollary~\ref{cor-divisible-injective}. 
\end{proof}
\begin{corollary}\label{cor-ext-coadditive}
 For any $f_0,f_1\:A'\to A$ we have 
 $(f_0+f_1)^*=f_0^*+f_1^*\:\Ext(A,B)\to\Ext(A',B)$.
\end{corollary}
\begin{proof}
 The corresponding statement is clearly true for the induced maps on
 $\Hom(A,E^2_B)$, and Corollary~\ref{cor-ext-four-b} identifies
 $\Ext(A,B)$ in a natural way as a quotient group of $\Hom(A,E^2_B)$.
\end{proof}
\begin{corollary}\label{cor-ext-distrib}
 The natural maps
 \begin{align*}
  \Ext(\bigoplus_iA_i,B) &\to \prod_i\Ext(A_i,B) \\
  \Ext(A,\prod_jB_j) &\to \prod_j\Ext(A,B_j).
 \end{align*}
 are isomorphisms.
\end{corollary}
\begin{proof}
 The functors $\Hom(-,U)$ (for $U\in\{B,E_B,E^2_B\}$) convert
 coproducts to products, and the cokernel of a product is the product
 of the cokernels, so the first statement follows from
 Corollary~\ref{cor-ext-four-b}.  Similarly, the functors $\Hom(T,-)$
 (for $T\in\{A,I_A,I^2_A\}$) preserve products, so the second
 statement follows from our original definition of $\Ext$.
\end{proof}

\begin{proposition}\label{prop-ext-six-b}
 Let $A\xra{i}B\xra{p}C$ be a short exact sequence of abelian groups,
 and let $V$ be an abelian group.  Then there is a natural exact
 sequence 
 \[ 0 \to \Hom(C,V) \xra{p^*} \Hom(B,V) \xra{i^*} \Hom(A,V) 
     \xra{\dl} \Ext(C,V) \xra{p^*} \Ext(B,V) \xra{i^*} \Ext(A,V) 
      \to 0.
 \]
\end{proposition}
\begin{proof}
 Consider the diagram 
 \begin{center}
  \begin{tikzcd}
     \Hom(C,E_V) \arrow[d,"q_*"'] \arrow[rightarrowtail,r,"p^*"] &
     \Hom(B,E_V) \arrow[d,"q_*"'] \arrow[twoheadrightarrow,r,"i^*"]  &
     \Hom(A,E_V) \arrow[d,"q_*"] \\
     \Hom(C,E^2_V) \arrow[rightarrowtail,r,"p^*"'] &
     \Hom(B,E^2_V) \arrow[twoheadrightarrow,r,"i^*"']  &
     \Hom(A,E^2_V) 
    \end{tikzcd}
   \end{center}
   
 The rows are short exact by Lemma~\ref{lem-right-test} together with
 Proposition~\ref{prop-divisible-injective}.  The Snake Lemma
 therefore gives us a six-term exact sequence involving the kernels
 and cokernels of the vertical maps $q_*$.
 Corollary~\ref{cor-ext-four-b} identifies these kernels and cokernels
 with $\Hom$ and $\Ext$ groups as required.
\end{proof}
\begin{corollary}\label{cor-ext-epi}
 For any groups $B$ and $V$, and any subgroup $A\leq B$, the
 restriction $\Ext(B,V)\to\Ext(A,V)$ is surjective. \qed
\end{corollary}

\begin{corollary}\label{cor-cyclic-ext}
 There are natural isomorphisms $\Ext(\Z/n,B)\simeq B/nB$ for $n>0$,
 and $\Ext(\Z,B)=0$.
\end{corollary}
Note that in conjunction with Corollary~\ref{cor-ext-distrib} this
allows us to calculate $\Ext(A,B)$ whenever $A$ is finitely
generated.  
\begin{proof}
 Corollary~\ref{cor-free-ext} tells us that $\Ext(\Z,B)=0$ for all
 $B$.  Now consider the short exact sequence 
 \[ \Z \xra{n.1_{\Z}} \Z \xra{} \Z/n. \]
 Using Proposition~\ref{prop-ext-six-b} we obtain an exact sequence 
 \[ 0 \to \Hom(\Z/n,B) \to \Hom(\Z,B) \to \Hom(\Z,B) \to 
     \Ext(\Z/n,B) \to \Ext(\Z,B) \to \Ext(\Z,B) \to 0.
 \]
 Now $\Ext(\Z,B)=0$, while $\Hom(\Z,B)$ is easily identified with $B$,
 and $\Hom(\Z/n,B)$ with $B[n]=\{b\in B\st nb=0\}$.  We thus have an
 exact sequence 
 \[ 0 \to B[n] \to B \xra{n.1_B} B \xra{\dl} \Ext(\Z/n,B) \to 0. \]
 From this it is clear that $\Ext(\Z/n,B)=\cok(n.1_B)=B/nB$.
\end{proof}

\begin{proposition}\label{prop-ext-AZ}
 Let $A$ be a torsion group.  Then there is a canonical isomorphism
 \[ \Ext(A,\Z) = \Hom(A,\QZ) = \prod_p\Hom(\tors_p(A),\QZ), \]
 and this maps surjectively to $\prod_p\Hom(A[p],\Z/p)$.  In
 particular, if $\Ext(A,\Z)=0$ then $A=0$.
\end{proposition}
\begin{proof}
 As $A$ is torsion, and both $\Z$ and $\Q$ are torsion free, we see
 that $\Hom(A,\Z)=\Hom(A,\Q)=0$.  As $\Q$ and $\QZ$ are divisible, we
 see from Proposition~\ref{prop-divisible-injective} that
 $\Ext(A,\Q)=\Ext(A,\QZ)=0$.  Thus, if we apply
 Proposition~\ref{prop-ext-six} to the short exact sequence
 $\Z\to\Q\to\QZ$ we just get an isomorphism
 $\dl\:\Ext(A,\Z)\to\Hom(A,\QZ)$.  Next,
 Proposition~\ref{prop-tors-split} gives $A=\bigoplus_p\tors_p(A)$,
 so $\Hom(A,\QZ)=\prod_p\Hom(\tors_p(A),\QZ)$.  Now $A[p]$ is a
 subgroup of $\tors_p(A)$ and $\QZ$ is divisible so the restriction
 \[ \Hom(\tors_p(A),\QZ)\to\Hom(A[p],\QZ) \]
 is surjective.  Moreover, any homomorphism from $A[p]\to\QZ$
 necessarily lands in $(\QZ)[p]$, which is a copy of $\Z/p$ generated
 by the element $(1/p)+\Z$.  By taking the product over all $p$, we
 get a surjection
 \[ \prod_p\Hom(\tors_p(A),\QZ) \to \prod_p\Hom(A[p],\Z/p) \]
 as claimed.  The last statement follows from the isomorphism
 $\Ext(A,\Z)=\Hom(A,\QZ)$ together with
 Proposition~\ref{prop-QZ-cogen}. 
\end{proof}
\begin{example}\label{eg-ext-QZ}
 We can take $A=\QZ$, and we find that $\Ext(\QZ,\Z)=\End(\QZ)$.  It
 is easy to see that $(\QZ)[n]$ is a copy of $\Z/n$, generated by
 $(1/n)+\Z$.  It follows that $\Hom((\QZ)[p],\Z/p)\simeq\Z/p$, and
 thus that $\prod_p\Hom((\QZ)[p],\Z/p)$ is uncountable, and thus that
 $\End(\QZ)$ is uncountable.  Next, we can apply $\Hom(-,\Z)$ and
 $\Ext(-,\Z)$ to the sequence $\Z\to\Q\to\QZ$ to get a six term exact
 sequence.  As $\Q$ and $\QZ$ are divisible, we find that
 $\Hom(\Q,\Z)=\Hom(\QZ,\Z)=0$.  We also have $\Hom(\Z,\Z)=\Z$ and
 $\Ext(\Z,\Z)=0$.  The six term sequence therefore reduces to a short
 exact sequence $\Z\to\End(\QZ)\to\Ext(\Q,\Z)$.  As $\Z$ is countable
 and $\End(\QZ)$ is uncountable, we deduce that $\Ext(\Q,\Z)$ is
 uncountable.  In particular, it is nonzero.
\end{example}

\begin{definition}\label{defn-extension}
 An \emph{extension of $C$ by $A$} is a short exact sequence
 $A\xra{i}B\xra{p}C$.  We say that two extensions
 $(A\xra{i_0}B_0\xra{p_0}C)$ and $(A\xra{i_1}B_1\xra{p_1}C)$ are
 \emph{equivalent} if there exists a map $f\:B_0\to B_1$ with
 $fi_0=i_1$ and $p_1f=p_0$.  (Any such $f$ is an isomorphism by a
 straightforward diagram chase, and using this we see that this notion
 of equivalence is reflexive, symmetric and transitive.)
\end{definition}

\begin{proposition}\label{prop-ext-pullback}
 Let $E=(A\xra{i}B\xra{p}C)$ be an extension of $C$ by $A$.  
 \begin{itemize}
  \item[(a)] For any homomorphism $h\:C'\to C$, we have an extension
   $h^*E=(A\xra{i'}B'\xra{p'}C')$ given by 
   \begin{align*}
    B' &= \{(b,c')\in B\oplus C'\st p(b)=h(c')\} \\
    i'(a) &= (i(a),0) \\
    p'(b,c') &= c'.
   \end{align*}
   (We call this the \emph{pullback of $E$ along $h$.})
  \item[(b)] Suppose we have another extension
   $E^*=(A\xra{i^*}B^*\xra{p^*}C')$ and a commutative diagram 
   \begin{center}
    \begin{tikzcd}
        A \arrow[rightarrowtail,r,"i^*"] \arrow[equal,d] &
        B^* \arrow[twoheadrightarrow,r,"p^*"] \arrow[d,"g"] &
        C' \arrow[d,"h"] \\
        A \arrow[rightarrowtail,r,"i"'] &
        B \arrow[twoheadrightarrow,r,"p"'] &
        C
       \end{tikzcd}
      \end{center}
      
   Then $E^*$ is equivalent to $h^*E$.
  \item[(c)] For any map $m\:C'\to B$, the extension $(h+pm)^*E$ is
   equivalent to $h^*E$.
  \item[(d)] If $E_0$ and $E_1$ are equivalent extensions of $C$ by
   $A$, then $h^*E_0$ and $h^*E_1$ are also equivalent.
 \end{itemize}
\end{proposition}
\begin{proof}\nl
 \begin{itemize}
  \item[(a)] We can certainly define a group $B'$ and homomorphisms
   $i'$ and $p'$ by the given formulae.  As $i$ is injective, it is
   clear that $i'$ is injective.  Now suppose that $c'\in C'$.  We
   then have $f(c')\in C$ and $p\:B\to C$ is surjective by assumption,
   so we can choose $b\in B$ with $p(b)=f(c')$.  This gives a point
   $b'=(b,c')\in B'$ with $p'(b')=c'$, so we see that $p'$ is
   surjective.  It is immediate that $p'i'=0$, so
   $\image(i')\leq\ker(p')$.  A general element of $\ker(p')$ has the
   form $b'=(b,0)$ with $p(b)=f(0)=0$, so $b\in\ker(p)=\image(i)$, so
   $b=i(a)$ for some $a\in A$.  This means that
   $b'=i'(a)\in\image(i')$.  We conclude that the sequence $h^*E$ is
   indeed short exact, so it gives an extension of $C'$ by $A$.
  \item[(b)] Now suppose we have a commutative diagram as indicated.
   As $hp^*=pg$ we can define $g'\:B^*\to B'$ by
   $g'(b^*)=(g(b^*),p^*(b^*))$.  We then have $p'g'(b^*)=p^*(b^*)$ and
   \[ g'(i^*(a))=(g(i^*(a)),p^*(i^*(a)))=(i(a),0)=i'(a), \]
   so $p'g'=p^*$ and $g'i^*=i'$.  Thus, $g'$ gives an equivalence
   between $E^*$ and $h^*E$.
  \item[(c)] By the construction in part~(a), we have a commutative
   diagram 
   \begin{center}
    \begin{tikzcd}
     A \arrow[rightarrowtail,r,"i'"] \arrow[equal,d] &
     B' \arrow[twoheadrightarrow,r,"p'"] \arrow[d,"g"] &
     C' \arrow[d,"h"] \\
     A \arrow[rightarrowtail,r,"i"'] &
     B \arrow[twoheadrightarrow,r,"p"'] &
     C        
    \end{tikzcd}
   \end{center}
   (where $g(b,c')=b$).  It follows easily that there is also a
   commutative diagram 
   \begin{center}
    \begin{tikzcd}
     A \arrow[rightarrowtail,r,"i'"] \arrow[equal,d] &
     B' \arrow[twoheadrightarrow,r,"p'"] \arrow[d,"g+mp'"] &
     C' \arrow[d,"h+pm"] \\
     A \arrow[rightarrowtail,r,"i"'] &
     B \arrow[twoheadrightarrow,r,"p"'] &
     C.
    \end{tikzcd}
   \end{center}   
   Now part~(b) tells us that the top row is equivalent to $(h+pm)^*$
   of the bottom row, or in other words $h^*E\simeq (h+pm)^*E$ as
   claimed. 
  \item[(d)] Suppose we have equivalent extensions
   $E_k=(A\xra{i_k}B_k\xra{p_k}C)$ for $k=0,1$, so there is an
   isomorphism $s\:B_0\to B_1$ with $si_0=i_1$ and $p_1s=p_0$.  We can
   then define $s'\:B'_0\to B'_1$ by $s'(b_0,c')=(s(b_0),c')$, and we
   find that $s'i'_0=i'_1$ and $p'_1s'=p'_0$; this shows that $h^*E_0$
   and $h^*E_1$ are equivalent as claimed.
 \end{itemize}
\end{proof}

\begin{proposition}\label{prop-ext-pushout}
 Let $E=(A\xra{i}B\xra{p}C)$ be an extension of $C$ by $A$.  
 \begin{itemize}
  \item[(a)] For any homomorphism $f\:A\to A'$, we have an extension
   $f_*E=(A'\xra{i'}B'\xra{p'}C)$ given by 
   \begin{align*}
    R &= \{(f(a),-i(a))\st a\in A\} \leq A'\oplus B \\
    B' &= (A'\oplus B)/R \\
    i'(a') &=  (a',0)+R  \\
    p'((a',b)+R) &= p(b).
   \end{align*}
   (We call this the \emph{pushout of $E$ along $f$.})
  \item[(b)] Suppose we have another extension
   $E^*=(A\xra{i^*}B^*\xra{p^*}C')$ and a commutative diagram 
   \begin{center}
    \begin{tikzcd}
     A \arrow[rightarrowtail,r,"i"] \arrow[d,"f"'] q &
     B \arrow[twoheadrightarrow,r,"p"] \arrow[d,"g"'] &
     C \arrow[equal,d] \\
     A^* \arrow[rightarrowtail,r,"i^*"'] &
     B^* \arrow[twoheadrightarrow,r,"p^*"'] &
     C
    \end{tikzcd}
   \end{center}
   Then $E^*$ is equivalent to $f_*E$.
  \item[(c)] For any map $n\:B\to A'$, the extension $(f+ni)_*E$ is
   equivalent to $f_*E$.
  \item[(d)] If $E_0$ and $E_1$ are equivalent extensions of $C$ by
   $A$, then $f_*E_0$ and $f_*E_1$ are also equivalent.
 \end{itemize}
\end{proposition}
\begin{proof}\nl
 \begin{itemize}
  \item[(a)] We can certainly define groups $R$ and $B'$, and a
   homomorphism $i'$, by the given formulae.  If $(a',b)\in R$ then
   there exists $a\in A$ with $a'=f(a)$ and $b=-i(a)$, so
   $p(b)=-p(i(a))=0$.  Given this, we see that the formula
   $p'((a',b)+R)=p(b)$ also gives a well-defined map $B'\to C$.  

   If $i'(a')=0$ we must have $(a',0)\in R$, so there exists $a\in A$
   with $f(a)=a'$ and $i(a)=0$.  As $i$ is injective this gives $a=0$
   and then $a'=f(0)=0$.  This proves that $i'$ is injective.  As $p$
   is surjective, it is immediate that $p'$ is also surjective.  Next,
   we have $p'i'(a')=p'((a',0)+R)=p(0)=0$, so
   $\image(i')\leq\ker(p')$.  Conversely, suppose we have an element
   $b'=(a',b)+R\in B'$ with $p'(b')=0$.  This means that $p(b)=0$, so
   $b=i(a)$ for some $a\in A$.  We then find that $(f(a),-i(a))\in R$
   \[ b' = (a',b)+R = (a',b)+(f(a),-i(a))+R = (a'+f(a),0) + R =
       i'(a'+f(a)) \in \image(i').
   \] 
   We conclude that the sequence $f_*E$ is indeed short exact, so it
   gives an extension of $C$ by $A'$.
  \item[(b)] Now suppose we have a commutative diagram as indicated.
   Define $g''\:A'\oplus B\to B^*$ by $g''(a',b)=i^*(a')+g(b)$.  We
   then have $g''(f(a),-i(a))=(i^*f-gi)(a)=0$, so $g''(R)=0$, so there
   is a unique homomorphism $g'\:B'\to B^*$ given by
   $g'(x+R)=g''(x)$.  This means that
   $g'i'(a)=g'((a',0)+R)=g''(a',0)=i^*(a')$, so $g'i'=i^*$.  We also
   have 
   \[ p^*g'((a',b)+R)=p^*i^*(a')+p^*g(b)=0+p(b) = p(b), \]
   so $p^*g'=p$.  Thus, $g'$ gives the required equivalence from
   $f_*E$ to $E^*$.
  \item[(c)] By the construction in part~(a), we have a commutative
   diagram 
   \begin{center}
    \begin{tikzcd}
     A \arrow[rightarrowtail,r,"i"] \arrow[d,"f"'] &
     B \arrow[twoheadrightarrow,r,"p"] \arrow[d,"g"'] &
     C \arrow[equal,d] \\
     A' \arrow[rightarrowtail,r,"i'"'] &
     B' \arrow[twoheadrightarrow,r,"p'"'] &
     C        
    \end{tikzcd}
   \end{center}
   (where $g(b)=(0,b)+R$).  It follows easily that there is also a
   commutative diagram 
   \begin{center}
    \begin{tikzcd}
     A \arrow[rightarrowtail,r,"i"] \arrow[d,"f+ni"'] &
     B \arrow[twoheadrightarrow,r,"p"] \arrow[d,"g+i'n"'] &
     C \arrow[equal,d] \\
     A' \arrow[rightarrowtail,r,"i'"'] &
     B' \arrow[twoheadrightarrow,r,"p'"'] &
     C        
    \end{tikzcd}
   \end{center}
   Now part~(b) tells us that the bottom row is equivalent to $(f+ni)_*$
   of the top row, or in other words $f_*E\simeq (f+ni)_*E$ as
   claimed. 
  \item[(d)] Suppose we have equivalent extensions
   $E_k=(A\xra{i_k}B_k\xra{p_k}C)$ for $k=0,1$, so there is an
   isomorphism $s\:B_0\to B_1$ with $si_0=i_1$ and $p_1s=p_0$.  We can
   then define $s'\:B'_0\to B'_1$ by 
   \[ s'((a',b_0)+R_0) = (a',s(b_0)) + R_1. \]
   We find that $s'i'_0=i'_1$ and $p'_1s'=p'_0$; this shows that
   $h^*E_0$ and $h^*E_1$ are equivalent as claimed.
 \end{itemize}
\end{proof}

\begin{proposition}\label{prop-ext-classification}
 Let $A$ and $C$ be abelian groups, and let $\Ext'(C,A)$ denote the
 set of equivalence classes of extensions of $C$ by $A$.  Let $Q$
 denote the extension $(I_C^2\xra{j}I_C\xra{q}C)$.  Then there is a
 well-defined bijection 
 \[ \zt\:\Ext(C,A) = \frac{\Hom(I_C^2,A)}{j^*(\Hom(I_C,A))} \to
           \Ext'(C,A)
 \]
 given by $\zt(\al+\image(j^*))=[\al_*(Q)]$.
\end{proposition}
\begin{proof}
 Using Proposition~\ref{prop-ext-pushout} (especially part~(c)) we see
 that there is a well-defined map $\zt$ as described.  We must show
 that it is a bijection.  Consider an arbitrary extension
 $E=(A\xra{i}B\xra{p}C)$.  As $I_C$ is free and $p$ is surjective,
 Lemma~\ref{lem-free-projective} gives us a homomorphism
 $\bt\:I_C\to B$ with $p\bt=q$.  This means that $p\bt j=qj=0$, so
 $\image(\bt j)\leq\ker(p)=\image(i)$, so there is a unique map
 $\al\:I^2_C\to A$ with $i\al=\bt j$.  We now have a commutative
 diagram 
 \begin{center}
  \begin{tikzcd}
   I^2_C \arrow[rightarrowtail,r,"j"] \arrow[d,"\al"'] &
   I_C \arrow[twoheadrightarrow,r,"q"] \arrow[d,"\bt"'] &
   C \arrow[equal,d] \\
   A \arrow[rightarrowtail,r,"i"'] &
   B \arrow[twoheadrightarrow,r,"p"'] &
   C, 
  \end{tikzcd}
 \end{center}     
 so Proposition~\ref{prop-ext-pushout}(b) tells us that
 $E\simeq\al_*Q$, or in other words $[E]=\zt(\al+\image(j^*))$.  This
 proves that $\zt$ is surjective.  Suppose we also have
 $[E]=\zt(\al'+\image(j^*))$.  There is then another commutative
 diagram 
 \begin{center}
  \begin{tikzcd}
   I^2_C \arrow[rightarrowtail,r,"j"] \arrow[d,"\al'"'] &
   I_C \arrow[twoheadrightarrow,r,"q"] \arrow[d,"\bt'"'] &
   C \arrow[equal,d] \\
   A \arrow[rightarrowtail,r,"i"'] &
   B \arrow[twoheadrightarrow,r,"p"'] &
   C.
  \end{tikzcd}
 \end{center}     
 In particular we have $p\bt'=q=p\bt$ so $p(\bt'-\bt)=0$, so
 $\bt'-\bt$ factors through $\ker(p)=\image(i)$, so there is a unique
 homomorphism $\gm\:I_C\to A$ with $\bt'=\bt+i\gm$.  Now $\bt'j=i\al'$
 and $\bt j=i\al$ so the equation $\bt'=\bt+i\gm$ yields
 $i\al'=i\al+i\gm j$, or $i(\al'-\al-\gm j)=0$.  As $i$ is injective
 we conclude that $\al'=\al+j^*(\gm)$, so $\al'$ and $\al$ have the
 same image in $\cok(j^*)=\Ext(C,A)$.  This proves that $\zt$ is also
 injective. 
\end{proof}

The above proposition gives a bijection from the group $\Ext(C,A)$ to
the set $\Ext'(C,A)$.  There is thus a unique group structure on
$\Ext'(C,A)$ for which this bijection is a homomorphism.  We would
like to understand this more intrinsically.

\begin{definition}\label{defn-baer-sum}
 Suppose we have two extensions $E_k=(A\xra{i_k}B_k\xra{p_k}C)$ for
 $k=0,1$.  The \emph{Baer sum} of $E_0$ and $E_1$ is the sequence 
 $E_2=(A\xra{i_2}B_2\xra{p_2}C)$ where 
 \begin{align*}
  U &= \{(b_0,b_1)\in B_0\oplus B_1\st p_0(b_0)=p_1(b_1)\} \\
  V &= \{(i_0(a),-i_1(a))\st a\in A\} \\
  B_2 &= U/V \\
  i_2(a) &= (i_0(a),0) + V = (0,i_1(a)) + V \\
  p_2((b_0,b_1)+V) &= p_0(b_0) = p_1(b_1).
 \end{align*}
\end{definition}

\begin{proposition}\label{prop-baer-sum}
 In the above context, the sequence $E_2$ is an extension of $C$ by
 $A$, with $[E_2]=[E_0]+[E_1]$ in $\Ext'(C,A)$.  Moreover, the zero
 element in $\Ext'(C,A)$ is the equivalence class consisting of all
 split extensions.
\end{proposition}
\begin{proof}
 First, if $i_2(a)=0$ we must have $(i_0(a),0)\in V$, so
 $(i_0(a),0)=(i_0(a'),-i_1(a'))$ for some $a'\in A$.  As $i_1$ is
 injective and $i_1(a')=0$ we have $a'=0$, so the equation
 $i_0(a)=i_0(a')$ gives $i_0(a)=0$ and then $a=0$.  This shows that
 $i_2$ is injective.  Next, suppose we have $c\in C$.  As both $p_0$
 and $p_1$ are surjective we can choose $b_0\in B_0$ and $b_1\in B_1$
 with $p_0(b_0)=p_1(b_1)=c$.  The element $b_2=(b_0,b_1)+V\in B_2$
 then satisfies $p_2(b_2)=c$, so $p_2$ is surjective.  Next, as
 $p_0i_0=0=p_1i_1$ we see from the definitions that $p_2i_2=0$, so
 $\image(i_2)\leq\ker(p_2)$.  Now suppose we have an element
 $b_2=(b_0,b_1)+V\in B_2$ with $p_2(b_2)=0$.  This means that
 $p_0(b_0)=0=p_1(b_1)$, so there is a unique element $a_0\in A$ with
 $b_0=i_0(a_0)$, and also a unique element $a_1\in A$ with
 $b_1=i_1(a_1)$.  Put $a_2=a_0+a_1$ and note that
 \[ i_2(a_2) = i_2(a_0) + i_2(a_1) =
     (i_0(a_0),0) + (0,i_1(a_1)) + V = (b_0,b_1)+V = b_2.
 \]
 This proves that $\ker(p_2)=\image(i_2)$, so we have an extension as
 claimed.  Now suppose that $[E_k]=\zt(\al_k+\image(j^*))$ for
 $k=0,1$, so there are commutative diagrams 
 \begin{center}
  \begin{tikzcd}
        I^2_C \arrow[rightarrowtail,r,"j"] \arrow[d,"\al_k"'] &
        I_C \arrow[twoheadrightarrow,r,"q"] \arrow[d,"\bt_k"'] &
        C \arrow[equal,d] \\
        A \arrow[rightarrowtail,r,"i_k"'] &
        B_k \arrow[twoheadrightarrow,r,"p_k"'] &
        C
       \end{tikzcd}
      \end{center}
      
 for $k=0,1$.  We define $\al_2\:I^2_C\to A$ by
 $\al_2(y)=\al_0(y)+\al_1(y)$, and we define $\bt_2\:I_C\to B_2$ by
 $\bt_2(x)=(\bt_0(x),\bt_1(x))+V$.  It is straightforward to check
 that this gives a commutative diagrams as above, showing that 
 \[ [E_2] = \zt(\al_2+\image(j^*)) = 
     \zt(\al_0+\image(j^*)) + \zt(\al_1+\image(j^*)) = 
      [E_0]+[E_1].
 \]
 Thus, the sum in $\Ext'(C,A)$ is the Baer sum, as claimed.  The zero
 element is $\zt(0)$, which is the pushout of the extension
 $Q=(I_C^2\xra{j}I_C\xra{q}C)$ along the map $0\:I^2_C\to A$.  If we
 use the notation of Proposition~\ref{prop-ext-pushout} in this
 context we have 
 \begin{align*}
  R &= \{(0,-j(y))\st y\in I_C^2\} = 0\oplus I_C^2 \leq A\oplus I_C \\
  B' &= \frac{A\oplus I_C}{R} = \frac{A\oplus I_C}{0\oplus I_C^2} = 
        A \oplus (I_C/I^2_C) \simeq A\oplus C \\
  i'(a) &=  (a,0) \\
  p'(a,c) &= c.
 \end{align*}
 Thus, $0_*Q$ is just the obvious split extension of $C$ by $A$.
\end{proof}

We now present a result that will help us relate homology groups to
cohomology groups.  There are very standard theorems that deduce
information about cohomology from information about homology.  To go
in the opposite direction we need the following proposition, which is
less well-known.

\begin{proposition}\label{prop-hom-ext-fg}
 Suppose that $\Hom(A,\Z)$ and $\Ext(A,\Z)$ are finitely generated.
 Then $A$ is finitely generated.
\end{proposition}

The proof will follow after some lemmas.

\begin{lemma}\label{lem-free-summand}
 Suppose that $\Hom(A,\Z)$ is finitely generated.  Then $A=B\oplus F$
 for some subgroups $B$ and $F$ such that $F$ is free and finitely
 generated, and $\Hom(B,\Z)=0$.
\end{lemma}
\begin{proof}
 Choose maps $f_1,\dotsc,f_r\:A\to\Z$ that generate $\Hom(A,\Z)$, and
 define $f\:A\to\Z^r$ by $f(a)=(f_1(a),\dotsc,f_r(a))$.  Now $f(A)$ is
 a subgroup of $\Z^r$, so it is free, with basis
 $f(a_1),\dotsc,f(a_s)$ say.  Put $B=\ker(f)$, and let $F\leq A$ be
 the subgroup generated by $a_1,\dotsc,a_s$.  We find that
 $f\:F\to f(A)\simeq\Z^s$ is an isomorphism, and thus that $A=B\oplus F$.
 Consider a homomorphism $g\:B\to\Z$.  Then the composite 
 \[ A = B\oplus F \xra{\text{proj}} B \xra{g} \Z \]
 is an element of the group $\Hom(A,\Z)$, which is generated by the
 maps $f_i$, but $f_i(B)=0$, so we see that $g=0$.  This proves
 that $\Hom(B,\Z)=0$.
\end{proof}

\begin{lemma}\label{lem-hom-ext-zero}
 Suppose that $\Hom(A,\Z)=\Ext(A,\Z)=0$.  Then $A=0$. 
\end{lemma}
\begin{proof}
 Corollary~\ref{cor-ext-epi} implies that $\Ext(\tors(A),\Z)=0$, so
 Proposition~\ref{prop-ext-AZ} gives $\tors(A)=0$, so $A$ is torsion
 free.  We thus have short exact sequences $A\xra{n}A\to A/n$ for all
 $n>0$, and using the resulting six term sequences we deduce that
 $\Ext(A/n,\Z)=0$.  As $A/n$ is torsion we can use
 Proposition~\ref{prop-ext-AZ} again to see that $A/n=0$, so $n.1_A$ is
 surjective. It is also injective because $\tors(A)=0$, so it is an
 isomorphism.  We can thus make $A$ into a vector space over $\Q$ by
 the rule $(m/n).a=(n.1_A)^{-1}(ma)$.  Linear algebra therefore tells
 us that either $A$ is zero, or it has $\Q$ as a summand.  In the
 latter case $\Ext(A,\Z)$ would contain the uncountable group
 $\Ext(\Q,\Z)$ as a summand, which is impossible as $\Ext(A,\Z)=0$.
 We therefore have $A=0$ as claimed.
\end{proof}

\begin{proof}[Proof of Proposition~\ref{prop-hom-ext-fg}]
 Using Lemma~\ref{lem-free-summand} we reduce to the case where
 $\Hom(A,\Z)=0$.  We next claim that there are only finitely many
 primes $p$ for which $A[p]\neq 0$.  Indeed, for any such $p$ we see
 (by linear algebra over the field $\Z/p$) that $\Hom(A[p],\Z/p)\neq
 0$.  If there are infinitely many such primes, we deduce that the
 group $P=\prod_p\Hom(A[p],\Z/p)$ is uncountable, which is impossible
 as Proposition~\ref{prop-ext-AZ} tells us that $P$ is a quotient of
 the finitely generated group $\Ext(A,\Z)$.  We can thus choose $p$
 such that $A[p]=0$, so we have a short exact sequence
 $A\xra{p}A\xra{}A/p$.  Proposition~\ref{prop-ext-six-b} then tells us
 that multiplication by $p$ is surjective on $\Ext(A,\Z)$.  By the
 structure theory of finitely generated groups, we see that
 $\Ext(A,\Z)$ must be finite, of order $n$ say, and that $n$ must be
 coprime to $p$.

 Next we have short exact sequences $A[n]\xra{i}A\xra{f}nA$ and
 $nA\xra{j}A\xra{g}A/n$, where $f(a)=na$ and $g$ is the quotient map.
 By assumption we have $\Hom(A,\Z)=0$, and also
 $\Hom(A[n],\Z)=\Hom(A/n,\Z)=0$ because $\Z$ is torsion free.  From
 the six term sequences we find that $\Hom(nA,\Z)=0$.  We also find
 that $j^*\:\Ext(A,\Z)\to\Ext(nA,\Z)$ is surjective and
 $f^*\:\Ext(nA,\Z)\to\Ext(A,\Z)$ is injective, but the composite
 $f^*j^*=(jf)^*$ is just multiplication by $n$.  As $n$ was defined to
 be the order of $\Ext(A,\Z)$ we deduce that $f^*j^*=0$, which implies
 that $\Ext(nA,\Z)=0$.  Lemma~\ref{lem-hom-ext-zero} therefore tells
 us that $nA=0$, so $A$ is torsion and $\Ext(A,\Z)=\Hom(A,\QZ)=A^*$.
 This is finitely generated by assumption.  Moreover, we have $nA=0$
 and therefore $nA^*=0$, so $A^*$ is a finite group.  It follows that
 $A^{**}$ is finite but Remark~\ref{rem-double-dual} gives an
 embedding $A\to A^{**}$ so $A$ is finite, as required.
\end{proof}

\section{Localisation}

\begin{definition}
 A \emph{multiplicative set} is a set $S$ of positive integers that
 contains $1$ and is closed under multiplication.
\end{definition}

\begin{definition}\label{defn-localisation}
 Let $A$ be an abelian group, and let $S$ be a multiplicative set.  We
 introduce an equivalence relation on the set $A\tm S$ by declaring
 that $(a,s)\sim (b,t)$ iff $bsx=atx$ for some $x\in S$.  We write
 $a/s$ for the equivalence class of the pair $(a,s)$, and we write
 $A[S^{-1}]$ for the set of equivalence classes.  We make this into an
 abelian group by the rule
 \[ a/s + b/t = (at+bs)/st. \]
\end{definition}

\begin{remark}\label{rem-localisation}
 Various checks are required to ensure that this definition is
 meaningful.  First, we must show that the given relation really is an
 equivalence relation.  It is clearly reflexive (as we can take $x=1$)
 and symmetric.  Suppose that $(a,s)\sim (b,t)\sim (c,u)$, so there
 are elements $x,y\in S$ with $atx=bsx$ and $buy=cty$.  Put
 $z=txy\in S$ and note that
 \[ auz = (au)(txy) = (atx)(uy) = (bsx)(uy) = (buy)(sx)
        = (cty)(sx) = (cs)(txy) = csz, \]
 so $(a,s)\sim(c,u)$, as required.

 Next, we should check that addition is well-defined.  More
 specifically, suppose that $(a_0,s_0)\sim (a_1,s_1)$ and
 $(b_0,t_0)\sim (b_1,t_1)$.  Put $(c_k,u_k)=(a_kt_k+b_ks_k,s_kt_k)$;
 we must show that $(c_0,u_0)\sim (c_1,u_1)$.  By hypothesis there are
 elements $x,y\in S$ such that $a_0s_1x=a_1s_0x$ and
 $b_0t_1y=b_1t_0y$.  We multiply these two equations by $t_0t_1y$ and
 $s_0s_1x$ respectively, and then add them together to get
 \[ a_0s_1xt_0t_1y + b_0t_1ys_0s_1x = 
    a_1s_0xt_0t_1y + b_1t_0ys_0s_1x, 
 \]
 or equivalently
 \[ (a_0t_0+b_0s_0)(s_1t_1)xy = (a_1t_1+b_1s_1)(s_0t_0)xy. \]
 If we put $z=xy\in S$ this can be rewritten as 
 $c_0u_1z=c_1u_0z$, so $(c_0,u_0)\sim(c_1,u_1)$ as required.

 Finally, we should show that addition is commutative and associative,
 that the element $0/1$ is an additive identity, and that $(-a)/s$ is
 an additive inverse for $a/s$.  All this is left to the reader.
\end{remark}

\begin{remark}\label{rem-local-zero}
 From the definitions we see that $a/s=0$ in $A[S^{-1}]$ if and only
 if there exists $t\in S$ with $at=0$.
\end{remark}

\begin{remark}\label{rem-ZS-Q}
 In the case $A=\Z$ it is not hard to see that $\Z[S^{-1}]$ can be
 identified with the set $\{n/s\in\Q\st n\in\Z,\;s\in S\}$, which is a
 subring of $\Q$.
\end{remark}

\begin{remark}\label{rem-local-split}
 Consider the case where $A$ is finite, so $A$ is the direct sum of
 its Sylow subgroups, say $A=A_1\oplus\dotsb\oplus A_r$ with
 $|A_i|=p_i^{v_i}$ for some primes $p_1,\dotsc,p_r$ and integers
 $v_i>0$.  We then find that $A[S^{-1}]=\bigoplus_kA_k[S^{-1}]$.
 Suppose that there exists $n\in S$ that is divisible by $p_k$.  In
 $A_k[S^{-1}]$ we then have
 $a/s=(an^{v_k})/(sn^{v_k})=0/(sn^{v_k})=0$, so $A_k[S^{-1}]=0$.  On the
 other hand, if there is no such $n$ then for each $s\in S$ we see
 that $s.1_{A_k}$ is invertible for all $s\in S$, and using this we
 will see later that $A_k[S^{-1}]=A_k$.  Thus, $A[S^{-1}]$ is just the
 direct sum of some subset of the Sylow subgroups. 
\end{remark}

\begin{definition}\label{defn-special-localisations}
 We use different notation for the most popular cases, as follows:
 \begin{itemize}
  \item[(a)] If $S=\{n^k\st k\geq 0\}$ we write $A[n^{-1}]$ or
   $A[1/n]$ for $A[S^{-1}]$.
  \item[(b)] If $p$ is prime and
   $S=\{n>0\st n\neq 0\pmod{p}\}=\N\sm p\N$ then we write $A_{(p)}$
   for $A[S^{-1}]$.  This is called the \emph{$p$-localisation} of
   $A$.
  \item[(c)] If $S=\{n\in\N\st n>0\}$ then we write $A\Q$ or $A_{(0)}$
   for $A[S^{-1}]$.  This is called the \emph{rationalisation} of $A$.
 \end{itemize}
\end{definition}

\begin{proposition}\label{prop-localisation-functor}
 Let $S$ be a multiplicative set.  Then any homomorphism $f\:A\to B$
 gives a homomorphism $f[S^{-1}]\:A[S^{-1}]\to B[S^{-1}]$ by the rule
 $f[S^{-1}](a/s)=f(a)/s$.  This construction gives an additive
 functor, and there is a natural map $\eta\:A\to A[S^{-1}]$ given by
 $\eta(a)=a/1$.  
\end{proposition}
\begin{proof}
 First, we see from the definitions that if $(a_0,s_0)\sim(a_1,s_1)$
 then $(f(a_0),s_0)\sim(f(a_1),s_1)$.  This shows that $f[S^{-1}]$ is
 well-defined.  It also follows directly from the definitions that it
 is a homomorphism.  Next, if we have maps $A\xra{f}B\xra{g}C$ then 
 \[ (gf)[S^{-1}](a/s)=gf(a)/s = g[S^{-1}](f(a)/s) =
      g[S^{-1}](f[S^{-1}](a/s)),
 \]
 which shows that our construction is functorial.  We also claim that
 $\eta$ is natural, which means that for any $f\:A\to B$ the square 
 \begin{center}
  \begin{tikzcd}
   A \arrow[rr,"f"] \arrow[d,"\eta"'] && B \arrow[d,"\eta"] \\
   A[S^{-1}] \arrow[rr,"{f[S^{-1}]}"'] && B[S^{-1}]
  \end{tikzcd}
 \end{center}
 commutes.  This is again straightforward.
\end{proof}

\begin{remark}
 When there is no danger of confusion, we will just write $f$ rather
 than $f[S^{-1}]$ for the induced map $A[S^{-1}]\to B[S^{-1}]$.
\end{remark}

\begin{proposition}\label{prop-localisation-exact}
 If $A\xra{f}B\xra{g}C$ is exact (or short exact), then so is the
 localised sequence 
 \[ A[S^{-1}]\xra{f[S^{-1}]} B[S^{-1}]\xra{g[S^{-1}]} C[S^{-1}]. \]
\end{proposition}
\begin{proof}
 First, as $gf=0$ and localisation is functorial we see that
 $g[S^{-1}]f[S^{-1}]=(gf)[S^{-1}]=0$, so
 $\image(f[S^{-1}])\leq\ker(g[S^{-1}])$.  Now consider an element
 $b/s\in\ker(g[S^{-1}])$.  We then have $g(b)/s=0/1$ in $B[S^{-1}]$, or
 equivalently $g(b)\,x=0$ for some $x\in S$.  This means that
 $g(xb)=0$ so $xb\in\ker(g)=\image(f)$, so there exists $a\in A$ with
 $f(a)=xb$.  This implies that
 $f[S^{-1}](a/(xs))=f(a)/(xs)=xb/xs=b/s$, so
 $b/s\in\image(f[S^{-1}])$.  We now see that the sequence 
 $A[S^{-1}]\xra{f[S^{-1}]} B[S^{-1}]\xra{g[S^{-1}]} C[S^{-1}]$ is
 exact as claimed.  Now suppose that the original sequence is short
 exact.  It is equivalent to say that the sequences $0\to A\to B$ and
 $A\to B\to C$ and $B\to C\to 0$ are all exact, and it follows from
 this that the sequences $0\to A[S^{-1}]\to B[S^{-1}]$ and
 $A[S^{-1}]\to B[S^{-1}]\to C[S^{-1}]$ and
 $B[S^{-1}]\to C[S^{-1}]\to 0$ are also exact.  We can then reassemble
 these pieces to see that the sequence
 $A[S^{-1}]\xra{f[S^{-1}]} B[S^{-1}]\xra{g[S^{-1}]} C[S^{-1}]$ is
 again short exact.
\end{proof}

\begin{proposition}\label{prop-local-tensor}
 There is a natural isomorphism $\mu\:\Z[S^{-1}]\ot A\to A[S^{-1}]$
 given by $\mu((n/s)\ot a)=(na)/s$.
\end{proposition}
\begin{proof}
 First, it is straightforward to check that there is a well-defined
 bilinear map $\mu_0\:\Z[S^{-1}]\tm A\to A[S^{-1}]$ given by
 $\mu_0(n/s,a)=(na)/s$.  By the universal property of tensor products,
 this gives a homomorphism $\mu\:\Z[S^{-1}]\ot A\to A[S^{-1}]$ with
 $\mu((n/s)\ot a)=(na)/s$.  In the opposite direction, we would like
 to define $\nu\:A[S^{-1}]\to \Z[S^{-1}]\ot A$ by
 $\nu(a/s)=(1/s)\ot a$.  To see that this is well-defined, suppose
 that $a/s=b/t$, so $xta=xsb$ for some $x\in S$.  From the definition
 of tensor products, we have $mu\ot v=u\ot mv$ in $U\ot V$ for all
 $u\in U$, $v\in V$ and $m\in\Z$.  We can apply this with
 $u=1/(stx)\in\Z[S^{-1}]$ and $v=a$ and $m=tx$ to get
 $(1/s)\ot a=(1/(stx))\ot xta$.  By a symmetrical argument, we have
 $(1/t)\ot b=(1/(stx))\ot xsb$, but $xta=xsb$ so we find that
 $(1/s)\ot a=(1/t)\ot b$, as required.  It is clear that
 $\mu\nu=1_{A[S^{-1}]}$.  The other way around, we have 
 \[ \nu\mu((n/s)\ot a) = \nu((na)/s) = (1/s)\ot na = (n/s)\ot a. \]
 Thus, $\nu$ is inverse to $\mu$.
\end{proof}

\begin{definition}\label{defn-S-local}
 Let $S$ be a multiplicative set, and let $A$ be an abelian group.  We
 say that $A$ is \emph{$S$-torsion} if for all $a\in A$ there exists
 $s\in S$ with $sa=0$.  We say that $A$ is \emph{$S$-local} if for
 each $s\in S$, the endomorphism $s.1_A\:A\to A$ is invertible.
\end{definition}

Some care is needed in relating the above definition to the
traditional terminology in the most popular cases:
\begin{definition}\label{defn-local-special}\nl
 \begin{itemize}
  \item[(a)] Definition~\ref{defn-torsion}(c) is equivalent to the
   following: we say that $A$ is a \emph{torsion group} if it is
   $S_0$-torsion, where $S_0=\{n\in\N\st n>0\}$.
  \item[(b)] We say that $A$ is \emph{rational} if it is $S_0$-local.
   (We will see that in this case, $A$ can be regarded as a vector
   space over $\Q$.)
  \item[(c)] Now let $p$ be a prime number.  We say that $A$ is
   \emph{$p$-torsion} if it is $p^{\N}$-torsion, where
   $p^{\N}=\{p^n\st n\in\N\}$.
  \item[(d)] However, we say that $A$ is \emph{$p$-local} if it is
   $S_p$-local, where $S_p=\N\sm p\N$.
 \end{itemize}
\end{definition}

\begin{proposition}\label{prop-local-omni}\nl
 \begin{itemize}
  \item[(a)] The group $A[S^{-1}]$ is always $S$-local.
  \item[(b)] The map $\eta\:A\to A[S^{-1}]$ is an isomorphism if and
   only if $A$ is $S$-local.
  \item[(c)] If $A$ is $S$-local then it can be regarded as a module
   over the ring $\Z[S^{-1}]\leq\Q$ by the rule
   $(n/s).a=(s.1_A)^{-1}(na)$.
  \item[(d)] Suppose that $f\:A\to B$ is a homomorphism, and that $B$
   is $S$-local.  Then there is a unique homomorphism
   $f'\:A[S^{-1}]\to B$ such that $f'\circ\eta=f\:A\to B$.
 \end{itemize}
\end{proposition}
\begin{proof}\nl
 \begin{itemize}
  \item[(a)] One can check that for each $s\in S$ there is a
   well-defined map $d_s\:A[S^{-1}]\to A[S^{-1}]$ given by
   $d_s(a/t)=a/(st)$.  This is inverse to $s.1_{A[S^{-1}]}$.
  \item[(b)] If $\eta$ is an isomorphism then it follows from~(a) that
   $A$ is $S$-local.  Conversely, if $A$ is $S$-local one can check
   that the formula $\zt(a/s)=(s.1_A)^{-1}(a)$ gives a well-defined
   map $\zt\:A[S^{-1}]\to A$, and that this is inverse to $\eta$.
  \item[(c)] First we must check that the multiplication rule is
   well-defined.  Suppose that $n/s=m/t$, so $ntx=msx$ for some
   $x\in S$.  As this is an equation in $\Z$ and $x>0$ it reduces to
   $tn=ms$.  If we write $n_A$ for $n.1_A$ and so on, we deduce that
   $t_An_A=m_As_A\:A\to A$.  We can compose on the left by $t_A^{-1}$
   and on the right by $s_A^{-1}$ to get $n_As_A^{-1}=t_A^{-1}m_A$.  As
   $s_A^{-1}$ is a homomorphism, it commutes with multiplication by
   $n$, so $s_A^{-1}n_A=t_A^{-1}m_A$.  This means that the definition
   of multiplication is consistent.  We will leave it to the reader
   that it has the usual associativity and distributivity properties.
  \item[(d)] We define $f'\:A[S^{-1}]\to B$ by
   $f'(a/s)=(s.1_B)^{-1}(f(a))$.  We leave it to the reader to show
   that this is well-defined and is a homomorphism.  Note that
   $f'(a/1)=f(a)$, so $f'\eta=f$.  If $f''$ is another homomorphism
   with $f''\eta=f$, we have 
   \[ s_A(f''(a/s)) = s.f''(a/s) = f''(s.(a/s)) = f''(a/1) =
       f''(\eta(a)) = f(a).
   \]
   As $s_A$ is invertible we can rewrite this as
   $f''(a/s)=s_A^{-1}(f(a))=f'(a/s)$.  As $a/s$ was arbitrary this
   means that $f''=f'$, which gives that claimed uniqueness statement.   
 \end{itemize}
\end{proof}

\begin{remark}\label{rem-idempotent-monad}
 As a consequence of~(b), we can identify $A[S^{-1}]$ with
 $A[S^{-1}][S^{-1}]$.  There is a slight subtlety here: there are two
 apparently different isomorphisms $A[S^{-1}]\to A[S^{-1}][S^{-1}]$,
 and to keep everything straight it is necessary to prove that they
 are the same.  Indeed, for any $B$, we have a map
 $\eta_B\:B\to B[S^{-1}]$.  We can specialise to the case $B=A[S^{-1}]$
 to get a map $\eta_{A[S^{-1}]}\:A[S^{-1}]\to A[S^{-1}][S^{-1}]$,
 given by $\eta_{A[S^{-1}]}(a/s)=(a/s)/1$.  Alternatively, we can
 apply Proposition~\ref{prop-localisation-functor} to the map
 $\eta_A\:A\to A[S^{-1}]$ to get another map
 $\eta_A[S^{-1}]\:A[S^{-1}]\to A[S^{-1}][S^{-1}]$, given by
 $(\eta_A[S^{-1}])(a/s)=(a/1)/s$.  It is clear that
 \[ s.\eta_{A[S^{-1}]}(a/s) = (a/1)/1 =
    s.(\eta_A[S^{-1}])(a/s),
 \]
 and multiplication by $s$ is an isomorphism on $A[S^{-1}][S^{-1}]$,
 so $\eta_{A[S^{-1}]}=\eta_A[S^{-1}]$.
\end{remark}

\begin{proposition}\label{prop-local-category}
 \begin{itemize}
  \item[(a)] If we have a short exact sequence $A\to B\to C$ in which
   two of the three terms are $S$-local, then so is the third.
  \item[(b)] Direct sums, products and retracts of $S$-local groups are
   $S$-local. 
  \item[(c)] The kernel, cokernel and image of any homomorphism
   between $S$-local groups are $S$-local.
  \item[(d)] $p$-torsion groups are $p$-local.
 \end{itemize}
\end{proposition}
\begin{proof}\nl
 \begin{itemize}
  \item[(a)] Note that $U$ is $S$-local iff for each $n\in S$ we have
   $U[n]=0$ and $U/n=0$.  Recall also that
   Proposition~\ref{prop-Zn-six} gives exact sequences
   \[ 0 \to A[n] \xra{j} B[n] \xra{q} C[n] \xra{\dl}
          A/n \xra{j} B/n \xra{q} C/n \to 0.
   \]
   The claim follows by diagram chasing.
  \item[(b)] This is clear, because direct sums, products and retracts
   of isomorphisms are isomorphisms.
  \item[(c)] Let $f\:A\to B$ be a homomorphism between $S$-local
   groups.  Then $\img(f)$ is a subgroup of $B$ so for $n\in S$ we
   have $\img(f)[n]\leq B[n]=0$.  Similarly, $\img(f)$ is a quotient
   of $A$ so $\img(f)/n$ is a quotient of $A/n$ and so is zero.  It
   follows that $\img(f)$ is $S$-local.  We can therefore apply~(a) to
   the short exact sequences $\ker(f)\to A\xra{f}\img(f)$ and
   $\img(f)\to B\to\cok(f)$ to see that $\ker(f)$ and $\cok(f)$ are
   also $S$-local.
  \item[(d)] Let $A$ be a $p$-torsion group.  Consider
   $m\in\Z\sm p\Z$; we must show that $m.1_A$ is an isomorphism.  For
   any $a\in A$ we can choose $k\geq 0$ such that $p^ka=0$, and $p^k$
   is coprime with $m$ so we can choose $r,s\in\Z$ with $p^kr+ms=1$.
   It follows that $a=msa=(m.1_A)(sa)$.  Using this we see that
   $m.1_A$ is surjective.  Also, if $ma=0$ then $a=sma=0$, so
   $A[m]=0$, so $m.1_A$ is also injective.
 \end{itemize}
\end{proof}

\begin{proposition}\label{prop-local-torsion}
 If $A$ is a torsion group, then $A_{(0)}=0$ and
 $A_{(p)}=\tors_p(A)$.  
\end{proposition}
\begin{proof}
 First suppose that $A$ is a $p$-torsion group, so
 $A=\bigcup_kA[p^k]$.  If $m$ is not divisible by $p$ then we can find
 $n>0$ with $mn=1\pmod{p^k}$, so $n.1_{A[p^k]}$ is inverse to
 $m.1_{A[p^k]}$.  It follows that $m.1_A$ is also an isomorphism.
 This holds for all $m\in\Z\sm p\Z$, so $A$ is $p$-local, so
 $A=A_{(p)}$.  Now suppose instead that $A$ is a $q$-torsion group for
 some prime $q\neq p$.  For any element $a\in A$ we have $q^va=0$ for
 some $v\geq 0$, so in $A_{(p)}$ we have
 $a/m=(q^va)/(q^vm)=0/(q^vm)=0$.  This shows that $A_{(p)}=0$.
 Finally, for a general torsion group $A$ we have
 $A=\bigoplus_q\tors_q(A)$ by Proposition~\ref{prop-tors-split}, so
 $A_{(p)}=\bigoplus_q\tors_q(A)_{(p)}$.  By the special cases that we
 have just discussed, this sum contains only the single factor
 $\tors_p(A)_{(p)}=\tors_p(A)$, as claimed.  It is also clear from
 Remark~\ref{rem-local-zero} that $A_{(0)}=0$. 
\end{proof}

\section{Colimits of sequences}

\begin{definition}\label{defn-sequence}
 By a \emph{sequence} we mean a diagram of the form
 \[ A_0 \xra{f_0} A_1 \xra{f_1} A_2 \xra{f_2} A_3 \xra{f_3} \dotsb.
 \]
 Given such a sequence and natural numbers $i\leq j$, we write
 $f_{ij}$ for the composite
 \[ A_i \xra{f_i} A_{i+1} \xra{} \dotsb \xra{f_{j-1}} A_j. \]
 In particular, $f_{ii}$ is the identity map of $A_i$, and
 $f_{i,i+1}=f_i$.  
\end{definition}

\begin{definition}\label{defn-colimit}
 Given such a sequence, we consider the group $A_+=\bigoplus_iA_i$.
 For each $k$ we have an inclusion $\iota_k\:A_k\to A_+$ and also a
 homomorphism $\iota_{k+1}\circ f_k\:A_k\to A_+$.  We let $R_k$ denote
 the image of $(\iota_k-\iota_{k+1}f_k)\:A_k\to A_+$, and put
 $R_+=\sum_kR_k\leq A_+$ and $\colim_iA_i=A_+/R_+$.  This group is
 called the \emph{colimit} of the sequence. 

 Next, we write $\ib_k$ for the composite 
 \[ A_k \xra{\iota_k} A_+ \xra{} A_+/R_+ = \colim_iA_i. \]
 By construction we have $\ib_k=\ib_{k+1}f_k$, so the following
 diagram commutes:
 \begin{center}
  \begin{tikzcd}
   A_0 \arrow[d,"\ib_0"'] \arrow[r,"f_0"] &
   A_1 \arrow[d,"\ib_1"'] \arrow[r,"f_1"] &
   A_2 \arrow[d,"\ib_2"'] \arrow[r,"f_2"] &
   A_3 \arrow[d,"\ib_3"'] \arrow[r,"f_3"] &
   \dotsb \\
   \colim_iA_i \arrow[equal,r] &
   \colim_iA_i \arrow[equal,r] &
   \colim_iA_i \arrow[equal,r] &
   \colim_iA_i \arrow[equal,r] &
   \dotsb
  \end{tikzcd}
 \end{center}  
 This implies that $\ib_k=\ib_m f_{k,m}$ whenever $k\leq m$.
\end{definition}

\begin{remark}\label{rem-shift}
 The definition can be reformulated slightly as follows.  We can
 define an endomorphism $S$ of $A_+$ by  
 \[ S(a_0,a_1,a_2,a_3,\dotsc) = 
      (0,f_0(a_0),f_1(a_1),f_2(a_2),f_3(a_3),\dotsb).
 \]
 Equivalently, $S$ is the unique map such that the following diagram
 commutes for all $k$:
 \begin{center}
  \begin{tikzcd}
   A_k \arrow[rightarrowtail,d,"\iota_k"'] \arrow[r,"f_k"] &
   A_{k+1} \arrow[d,"\iota_{k+1}"] \\
   A_+ \arrow[r,"S"'] &
   A_+
  \end{tikzcd}
 \end{center}  
 We can then say that $\colim_iA_i$ is the cokernel of
 $1-S\:A_+\to A_+$.  Note that if $a_i$ is the first nonzero entry in
 $a$ then the $i$'th entry in $(1-S)(a)$ is again $a_i$, so
 $(1-S)(a)\neq 0$.  This shows that $1-S$ is injective, so we actually
 have a short exact sequence
 \begin{center}
  \begin{tikzcd}
   A_+ \arrow[rightarrowtail,r,"1-S"] &
   A_+ \arrow[twoheadrightarrow,r] &
   \colim_iA_i.
  \end{tikzcd}
 \end{center}
\end{remark}

\begin{remark}\label{rem-cocone}
 The colimit can also be characterised by a universal property, as
 follows.  A \emph{cone} for the sequence is a group $B$ with a
 collection of maps $u_k\:A_k\to B$ such that $u_{k+1}f_k=u_k$ for all
 $k\geq 0$.  By construction, the maps $\ib_k\:A_k\to\colim_iA_i$ form
 a cone.  We claim that for any cone $\{A_k\xra{u_k}B\}_{k\in\N}$
 there is a unique homomorphism $u_\infty\:\colim_iA_i\to B$ such that
 $u_\infty\ib_k=u_k$ for all $k$.  Indeed,
 Proposition~\ref{prop-categorical-coproduct} gives us a unique map
 $v\:A_+\to B$ with $vi_k=u_k$ for all $k$, and the cone property
 tells us that $v(R_k)=0$ for all $k$, so $v(R_+)=0$, so $v$ induces a
 map $u_\infty\:\colim_iA_i=A_+/R_+\to B$.  This is easily seen to be
 the unique map such that $u_\infty\ib_k=u_k$ for all $k$.
\end{remark}

In many cases colimits are just unions, as we now explain.

\begin{proposition}\label{prop-colimit-structure}
 We have
 \[ \ib_0(A_0) \leq \ib_1(A_1) \leq \ib_2(A_2) \leq \dotsb
      \leq \colim_iA_i.
 \]
 Moreover, $\colim_iA_i$ is the union of the groups $\ib_k(A_k)$, and
 we have $\ib_k(a)=0$ iff $f_{k,m}(a)=0$ for some $m\geq k$.
\end{proposition}
\begin{proof}
 First, when $k\leq m$ we have $\ib_k=\ib_m f_{k,m}$, and this implies
 that $\ib_k(A_k)\leq\ib_m(A_m)$.  Next, as $\colim_iA_i$ is a
 quotient of $A_+$, we see that every element $a\in\colim_iA_i$ can be
 written as $a=\sum_{k=0}^N\ib_k(a_k)$ for some $N\geq 0$ and
 $a_k\in A_k$.  Now $\ib_k(a_k)\in\ib_k(A_k)\leq\ib_N(A_N)$ for all
 $k$, so $a\in\ib_N(A_N)$.  Thus $\colim_iA_i=\bigcup_N\ib_N(A_N)$ as
 claimed.  Now suppose we have $a\in A_k$ and that $f_{km}(a)=0$ for
 some $m\geq k$.  Using $\ib_k=\ib_mf_{km}$ we deduce that
 $\ib_k(a)=0$.  Conversely, suppose that $\ib_k(a)=0$, so
 $i_k(a)\in R_+$, so 
 \[ i_k(a) = \sum_{m=0}^{N-1} (i_m(b_m)-i_{m+1}(f_m(b_m))) \]
 for some $N>k$ and some $b_0,\dotsc,b_{N-1}$ with $b_i\in A_i$.  Now
 let $h\:\bigoplus_{m=0}^NA_m\to A_N$ be the map given by $f_{mN}$ on
 $A_m$, or more formally the unique map with $hi_m=f_{mN}$.  We note
 that 
 \[ h(i_m(b_m)-i_{m+1}(f_m(b_m))) = 
      f_{m,N}(b_m) - f_{m+1,N}(f_m(b_m)) = 0,
 \] 
 so we can apply $h$ to the above equation for $i_k(a)$ to get
 $f_{kN}(a)=0$ as required.
\end{proof}

\begin{corollary}\label{cor-cone-universal}
 Suppose we have a sequence $\{A_i\}_{i\in\N}$ and a cone
 $\{u_i\:A_i\to B\}_{i\in\N}$ giving rise to a homomorphism
 $u_\infty\:\colim_iA_i\to B$.  Then 
 \begin{itemize}
  \item[(a)] The image of $u_\infty$ is the union of the subgroups
   $u_k(A_k)$.
  \item[(b)] The map $u_\infty$ is injective iff whenever $u_k(a)=0$,
   there exists $m\geq k$ with $f_{km}(a)=0$.
 \end{itemize}
\end{corollary}
\begin{proof}
 This is clear from the Proposition.
\end{proof}

\begin{example}\label{eg-union-colimit}
 Let $A$ be any abelian group.  Suppose we have a chain of subgroups 
 \[ A_0 \leq A_1 \leq A_2 \leq A_3 \leq \dotsb \leq A. \]
 Let $f_k\:A_k\to A_{k+1}$ be the inclusion.  Then the colimit of the
 resulting sequence is just $\bigcup_iA_i$.  Indeed, the inclusions
 $A_n\to\bigcup_iA_i$ form a cone, which clearly has both the
 properties in Corollary~\ref{cor-cone-universal}.
\end{example}

The following examples are instructive as well as useful.
\begin{proposition}\label{prop-rational-colimit}
 Let $A$ be an arbitrary abelian group.  Then the colimit of the
 sequence 
 \[ A \xra{n} A \xra{n} A \xra{n} A \xra{n} A \xra{} \dotsb \]
 is $A[1/n]$, whereas the colimit of the sequence
 \[ A \xra{1} A \xra{2} A \xra{3} A \xra{4} A \xra{} \dotsb \]
 is the rationalisation $A_{(0)}$.
\end{proposition}
\begin{proof}
 We will prove the second statement; the first is similar but easier.
 Let $C$ denote the colimit, so we have maps $\ib_k\:A\to C$ with
 $\ib_k(a)=(k+1)\ib_{k+1}(a)$.  Define $u_n\:A\to A_{(0)}$ by
 $u_n(a)=a/n!$.  As $((n+1)a)/(n+1)!=a/n!$ we see that these maps form
 a cone, so there is a unique map $u_\infty\:C\to A_{(0)}$ with
 $u_\infty\ib_k=u_k$ for all $k$.  Any element $a\in A_{(0)}$ can be
 written as $a=a'/n$ for some $a'\in A$ and $n>0$, so
 $a=((n-1)!a')/n!=u_n((n-1)!a')$, so $A_{(0)}$ is the union of the
 images of the maps $u_n$.  Now suppose that $u_n(a')=0$.  By the
 definition of $A_{(0)}$ this just means that $ma'=0$ for some $m>0$.
 The map $f_{n,n+m}$ in our sequence is multiplication by the integer
 \[ p = (n+1)(n+2)\dotsb(n+m) = \frac{(n+m)!}{n!} = m!\bcf{n+m}{n},
 \]
 which is divisible by $m$, so $f_{n,n+m}(a')=0$.  The claim follows
 by Corollary~\ref{cor-cone-universal}.
\end{proof}

\begin{proposition}\label{prop-colimit-maps}
 Suppose we have a commutative diagram as shown
 \begin{center}
  \begin{tikzcd}
   A_0 \arrow[r,"f_0"] \arrow[d,"p_0"'] &
   A_1 \arrow[r,"f_1"] \arrow[d,"p_1"'] &
   A_2 \arrow[r,"f_2"] \arrow[d,"p_2"'] &
   A_3 \arrow[r,"f_3"] \arrow[d,"p_3"'] &
   \dotsb \\
   B_0 \arrow[r,"g_0"'] &
   B_1 \arrow[r,"g_1"'] &
   B_2 \arrow[r,"g_2"'] &
   B_3 \arrow[r,"g_3"'] &
   \dotsb
  \end{tikzcd}
 \end{center}  
 Let $\ib_k$ be the canonical map $A_k\to\colim_iA_i$, and let $\jb_k$
 be the canonical map $B_k\to\colim_iB_i$.  Then there is a unique map
 $p_\infty$ such that the diagram 
 \begin{center}
  \begin{tikzcd}
   A_k \arrow[d,"p_k"'] \arrow[r,"\ib_k"] & \colim_iA_i \arrow[d,"p_\infty"] \\
   B_k \arrow[r,"\jb_k"'] & \colim_iB_i
  \end{tikzcd}
 \end{center}
 
 commutes for all $k$.  Moreover, if all the maps $p_k$ are injective,
 or surjective, or bijective, then $p_\infty$ has the same property.
\end{proposition}
\begin{proof}
 The maps $\jb_k p_k\:A_k\to\colim_iB_i$ satisfy 
 \[ \jb_k p_k = \jb_{k+1}g_k p_k = \jb_{k+1}p_{k+1}f_k, \]
 so they form a cone for the sequence $\{A_i\}$.  There is thus a
 unique map $p_\infty\:\colim_iA_i\to\colim_iB_i$ with
 $p_\infty\ib_k=\jb_kp_k$ for all $k$, as claimed. 

 \begin{itemize} 
  \item[(a)] Now suppose that all the maps $p_k$ are injective.
   Consider an element $a\in\ker(p_\infty)$.  By
   Proposition~\ref{prop-colimit-structure} we have $a=\ib_k(a')$ for
   some $k$ and some $a'\in A_k$.  We then have
   $\jb_k(p_k(a'))=p_\infty(\ib_k(a'))=p_\infty(a)=0$.  The same
   proposition therefore tells us that $g_{km}(p_k(a'))=0$ for some
   $m\geq k$.  Now $g_{km}p_k=p_mf_{km}$ and $p_m$ is injective, so
   $f_{km}(a')=0$.  This means that $a=i_k(a')=i_m(f_{km}(a'))=0$.
   Thus, $p_\infty$ is injective as claimed.
  \item[(b)] Suppose instead that all the maps $p_k$ are surjective.
   Consider an element $b\in\colim_iB_i$.  By
   Proposition~\ref{prop-colimit-structure} we have $b=\jb_k(b')$ for
   some $k$ and some $b'\in B_k$.  As $p_k$ is surjective, we can
   choose $a'\in A_k$ with $p_k(a')=b'$, and then put $b=\ib_k(a)$;
   we find that $p_\infty(a)=b$.  Thus, $p_\infty$ is also
   surjective.
  \item[(c)] If the maps $p_k$ are all isomorphisms, then~(a) and~(b)
   together imply that $p_\infty$ is an isomorphism.
 \end{itemize}
\end{proof}

\begin{proposition}\label{prop-colimit-exact}
 Suppose we have a commutative diagram as shown, in which all the
 columns are exact:
 \begin{center}
  \begin{tikzcd}
   A_0 \arrow[r,"f_0"] \arrow[d,"p_0"'] &
   A_1 \arrow[r,"f_1"] \arrow[d,"p_1"'] &
   A_2 \arrow[r,"f_2"] \arrow[d,"p_2"'] &
   A_3 \arrow[r,"f_3"] \arrow[d,"p_3"'] &
   \dotsb \\
   B_0 \arrow[r,"g_0"'] \arrow[d,"q_0"'] &
   B_1 \arrow[r,"g_1"'] \arrow[d,"q_1"'] &
   B_2 \arrow[r,"g_2"'] \arrow[d,"q_2"'] &
   B_3 \arrow[r,"g_3"'] \arrow[d,"q_3"'] &
   \dotsb \\
   C_0 \arrow[r,"h_0"'] &
   C_1 \arrow[r,"h_1"'] &
   C_2 \arrow[r,"h_2"'] &
   C_3 \arrow[r,"h_3"'] &
   \dotsb
  \end{tikzcd}
 \end{center}  
 Then the resulting sequence
 \[ \colim_iA_i \xra{p_\infty}
    \colim_iB_i \xra{q_\infty}
    \colim_iC_i
 \]
 is also exact.
\end{proposition}
\begin{proof}
 First, any element $a\in\colim_iA_i$ has the form $a=\ib_n(a')$ for
 some $n$ and $a'$.  We can then chase $a'$ around the diagram 
 \begin{center}
  \begin{tikzcd}
   A_n \arrow[r,"\ib_n"] \arrow[d,"p_n"'] & 
   \colim_iA_i \arrow[d,"p_\infty"] \\
   B_n \arrow[r,"\jb_n"] \arrow[d,"q_n"'] & 
   \colim_iA_i \arrow[d,"q_\infty"] \\
   C_n \arrow[r,"\kb_n"]  & 
   \colim_iA_i
  \end{tikzcd}
 \end{center}  
 to see that $q_\infty(p_\infty(a))=0$.  Conversely, suppose we have
 an element $b\in\ker(q_\infty)$.  We then have $b=\jb_n(b')$ for some
 $n$ and some $b'\in B_n$.  We then have
 $\kb_n(q_n(b'))=q_\infty(\jb_n(b'))=q_\infty(b)=0$, so
 $h_{n,m}(q_n(b'))=0$ for some $m\geq n$.  Now
 $h_{n,m}(q_n(b'))=q_m(g_{n,m}(b'))$, so
 $g_{n,m}(b')\in\ker(q_m)=\image(p_m)$, so we can find $a'\in A_m$
 with $p_m(a')=g_{n,m}(b')$.  Now put $a=\ib_m(a')$.  We find that 
 \[ p_\infty(a)=\jb_m(p_m(a'))=\jb_m(g_{n,m}(b'))=\jb_n(b')=b, \]
 so $b\in\image(p_\infty)$.  The claim follows.
\end{proof}

\begin{proposition}\label{prop-colimit-cofinal}
 Suppose we have a sequence
 \[ A_0 \xra{f_0} A_1 \xra{f_1} A_2 \xra{f_2} A_3 \xra{f_3} \dotsb \]
 and a nondecreasing function $u\:\N\to\N$ such that $u(i)\to\infty$
 as $i\to\infty$.  Put 
 \[ g_i = f_{u(i),u(i+1)} = 
     (A_{u(i)} \xra{f_{u(i)}} A_{u(i)+1} \xra{f_{u(i)+1}} \dotsb 
       \xra{f_{u(i+1)-1}} A_{u(i+1)}),
 \]
 so we have a sequence
 \[ A_{u(0)} \xra{g_0} A_{u(1)} \xra{g_1} A_{u(2)}
     \xra{g_2} A_{u(3)} \xra{g_3} \dotsb
 \]
 Then there is a canonical isomorphism $\colim_jA_{u(j)}=\colim_iA_i$.
\end{proposition}
\begin{proof}
 Let $\ib_n\:A_n\to\colim_iA_i$ and
 $\jb_n\:A_{u(n)}\to\colim_jA_{u(j)}$ be the usual maps.  As
 $\ib_n=\ib_{n+1}f_n$ for all $n$ we find by induction that
 $\ib_n=\ib_mf_{n,m}$ for all $n\leq m$.  By applying this to the pair
 $u(k)\leq u(k+1)$, we see that $\ib_{u(k)}=\ib_{u(k+1)}g_k$, so the
 maps $\ib_{u(k)}$ form a cone for the sequence
 $\{A_{u(j)}\}_{j\in\N}$, so there is a unique map
 $p\:\colim_jA_{u(j)}\to\colim_iA_i$ with $p\jb_k=\ib_{u(k)}$ for all
 $k$.  In the opposite direction, suppose we have $n\in\N$.  As
 $u(j)\to\infty$ as $j\to\infty$, we can choose $k$ such that
 $n\leq u(k)$, and form the composite 
 \[ q_{nk} = (A_n\xra{f_{n,u(k)}}A_{u(k)}\xra{\jb_k}\colim_jA_{u(j)}).
 \]
 As $\jb_k=\jb_{k+1}g_k=\jb_{k+1}f_{u(k),u(k+1)}$ and
 $f_{u(k),u(k+1)}f_{n,u(k)}=f_{n,u(k+1)}$ we see that
 $q_{n,k}=q_{n,k+1}$.  Thus $q_{nk}$ is independent of $k$, so we can
 denote it by $q_n$.  We also find that $q_n=q_{n+1}f_n$, so the maps
 $q_n$ form a cone for the sequence $\{A_i\}_{i\in\N}$, so there is a
 unique map $q\:\colim_iA_i\to\colim_jA_{u(j)}$ with $q\ib_n=q_n$ for
 all $n$.  

 Now note that any element $a\in\colim_iA_i$ has the form $a=\ib_n(a')$
 for some $n$ and $a'\in A_n$.  If we choose $k$ with $u(k)\geq n$ we
 have 
 \[ pq(a) = pq_n(a') = p\jb_k f_{n,u(k)}(a') = 
     \ib_{u(k)} f_{n,u(k)}(a')= \ib_n(a') = a,
 \]
 so $pq$ is the identity.  A similar argument shows that $qp$ is the
 identity. 
\end{proof}

\begin{proposition}\label{prop-colimit-fubini}
 Suppose we have a commutative diagram 
 \begin{center}
  \begin{tikzcd}
   A_{00} \arrow[r,"f_{00}"] \arrow[d,"g_{00}"'] &
   A_{01} \arrow[r,"f_{01}"] \arrow[d,"g_{01}"'] &
   A_{02} \arrow[r,"f_{02}"] \arrow[d,"g_{02}"'] &
   A_{03} \arrow[r,"f_{03}"] \arrow[d,"g_{03}"'] &
   \dotsb \\
   A_{10} \arrow[r,"f_{10}"] \arrow[d,"g_{10}"'] &
   A_{11} \arrow[r,"f_{11}"] \arrow[d,"g_{11}"'] &
   A_{12} \arrow[r,"f_{12}"] \arrow[d,"g_{12}"'] &
   A_{13} \arrow[r,"f_{13}"] \arrow[d,"g_{13}"'] &
   \dotsb \\
   A_{20} \arrow[r,"f_{20}"] \arrow[d,"g_{20}"'] &
   A_{21} \arrow[r,"f_{21}"] \arrow[d,"g_{21}"'] &
   A_{22} \arrow[r,"f_{22}"] \arrow[d,"g_{22}"'] &
   A_{23} \arrow[r,"f_{23}"] \arrow[d,"g_{23}"'] &
   \dotsb \\
   A_{30} \arrow[r,"f_{30}"] \arrow[d,"g_{30}"'] &
   A_{31} \arrow[r,"f_{31}"] \arrow[d,"g_{31}"'] &
   A_{32} \arrow[r,"f_{32}"] \arrow[d,"g_{32}"'] &
   A_{33} \arrow[r,"f_{33}"] \arrow[d,"g_{33}"'] &
   \dotsb \\
   \dotsb & \dotsb & \dotsb & \dotsb 
  \end{tikzcd}
 \end{center} 
 and thus sequences
 \[ \colim_jA_{0j} \xra{g_{0\infty}}
    \colim_jA_{1j} \xra{g_{1\infty}}
    \colim_jA_{2j} \xra{g_{2\infty}}
    \colim_jA_{3j} \xra{g_{3\infty}} \dotsb
 \]
 and
 \[ \colim_iA_{i0} \xra{f_{\infty 0}}
    \colim_iA_{i1} \xra{f_{\infty 1}}
    \colim_iA_{i2} \xra{f_{\infty 2}}
    \colim_iA_{i3} \xra{f_{\infty 3}} \dotsb
 \]
 Suppose we also put 
 \[ h_i = g_{i,j+1}f_{ij} = f_{i+1,j}g_{ij} \: A_{ii} \to A_{i+1,i+1}, \]
 giving a third sequence 
 \[ A_{00} \xra{h_0} A_{11} \xra{h_1} A_{22}
     \xra{h_2} A_{33} \xra{h_3} \dotsb
 \]
 Then there are canonical isomorphisms
 \[ \colim_i\colim_j A_{ij} \simeq
    \colim_iA_{ii} \simeq 
    \colim_j\colim_i A_{ij}.
 \]
\end{proposition}
\begin{example}
 In conjunction with Proposition~\ref{prop-rational-colimit}, this
 will give 
 \[ A[\frac{1}{n}][\frac{1}{m}] =
    A[\frac{1}{nm}] =
    A[\frac{1}{m}][\frac{1}{n}]
 \]
\end{example}
\begin{proof}
 Put $A_{++}=\bigoplus_{n,m}A_{nm}$, and let
 $i_{nm}\:A_{nm}\to A_{++}$ be the canonical inclusion.  Put 
 \begin{align*}
  P_{nm} &= \image(i_{nm}-i_{n,m+1}f_{nm}\:A_{nm}\to A_{++}) \\
  Q_{nm} &= \image(i_{nm}-i_{n+1,m}g_{nm}\:A_{nm}\to A_{++}).
 \end{align*}
 From the definitions we have
 \begin{align*}
  \bigoplus_n\colim_jA_{nj} &= A_{++}/\sum_{n,m}P_{nm} \\
  \bigoplus_m\colim_iA_{im} &= A_{++}/\sum_{n,m}Q_{nm}.
 \end{align*}
 It follows easily that 
 \[ \colim_i\colim_j A_{ij} = 
     A_{++}/\left(\sum_{n,m}P_{nm} + \sum_{n,m}Q_{nm}\right) = 
    \colim_j\colim_i A_{ij}.
 \]
 We write $A_{\infty\infty}$ for this group, and we write $\ib_{nm}$
 for the obvious map $A_{nm}\to A_{\infty\infty}$.  By construction,
 the following diagram commutes:
 \begin{center}
  \begin{tikzcd}
  A_{nm} \arrow[r,"f_{nm}"] \arrow[d,"g_{nm}"'] \arrow[dr,"\ib_{nm}"] &
  A_{n,m+1} \arrow[d,"\ib_{n,m+1}"] \\
  A_{n+1,m} \arrow[r,"\ib_{n+1,m}"'] &
  A_{\infty\infty}.
 \end{tikzcd}
\end{center}

 It follows that $\ib_{kk}=\ib_{k+1,k+1}h_k$, so the maps $\ib_{kk}$
 form a cone for the sequence $\{A_{ii}\}_{i\in\N}$.  Thus, if we
 write $\jb_k$ for the usual map $A_{kk}\to\colim_iA_{ii}$, we find
 that there is a unique map $p\:\colim_iA_{ii}\to A_{\infty\infty}$
 with $p\jb_k=\ib_{kk}$ for all $k$.  In the opposite direction,
 suppose we have $n,m,k\in\N$ with $n,m\leq k$.  By composing $f$'s
 and $g$'s in various orders we can form a number of maps
 $A_{nm}\to A_{kk}$ but they are all the same because the original
 diagram is commutative.  We write $u_{nmk}$ for this map, and put
 $q_{nmk}=\jb_k u_{nmk}\:A_{nm}\to\colim_iA_{ii}$.  Now
 $\jb_k=\jb_{k+1}h_k$ and $h_ku_{nmk}=u_{n,m,k+1}$ so
 $q_{n,m,k}=q_{n,m,k+1}$.  Thus $q_{nmk}$ is independent of $k$
 (provided that $k\geq\max(n,m)$) so we can denote it by $q_{nm}$.  We
 can now put the maps $q_{nm}$ together to give a map
 $q'\:A_{++}\to\colim_iA_{ii}$ with $q'i_{nm}=q_{nm}$.  When
 $k>\max(n,m)$ we have $u_{nmk}=u_{n,m+1,k}f_{nm}=u_{n+1,m,k}g_{nm}$,
 and using this we see that $q'(P_{nm})=0=q'(Q_{nm})$.  There is thus
 an induced map $A_{\infty\infty}\to\colim_iA_{ii}$ with
 $q\ib_{nm}=q_{nm}$.  We leave it to the reader to check that $q$ is
 inverse to $p$. 
\end{proof}

\section{Limits and derived limits of towers}
\label{sec-towers}

\begin{definition}\label{defn-tower}
 A \emph{tower} is a diagram of the form 
 \[ B_0 \xla{f_0} B_1 \xla{f_1} B_2 \xla{f_2} B_3 \xla{f_3} \dotsb \]
 Given such a tower an integers $i\geq j$, we write $f_{ij}$ for the
 composite 
 \[ B_i \xra{f_{i-1}} B_{i-1} \xra{f_{i-2}} \dotsb \xra{f_j} B_j. \]
 Note that $f_{ii}$ is the identity map, and $f_{i+1,i}=f_i$, and
 $f_{jk}f_{ij}=f_{ik}$ whenever $i\geq j\geq k$.
\end{definition}

\begin{definition}\label{defn-invlim}
 Suppose we have a tower as above.  The \emph{limit} (or \emph{inverse
 limit}) of the tower is the group 
 \[ \invlim_i B_i =
      \{a\in\prod_iB_i \st a_i=f_i(a_{i+1}) \text{ for all } i\}.
 \]
 Equivalently, if we define $D\:\prod_iB_i\to\prod_iB_i$ by
 $D(a)_i=a_i-f_i(a_{i+1})$, then $\invlim_iB_i=\ker(D)$.  We also
 write $\invlim{}_i^1B_i$ for the cokernel of $D$.  We define
 $p_n\:\invlim_iB_i\to B_n$ by $p_n(a)=a_n$, so $p_n=f_np_{n+1}$.
\end{definition}
\begin{remark}\label{rem-cone}
 The limit can also be characterised by a universal property, as
 follows.  A \emph{cone} for the tower is a group $A$ with a
 collection of maps $u_k\:A\to B_k$ such that $u_k=f_ku_{k+1}$ for all
 $k$.  Tautologically, the maps $p_k\:\invlim_iB_i\to B_k$ form a
 cone.  Moreover, for any cone $\{A\xra{u_k}B_k\}_{k\in\N}$ we can
 define $u_\infty\:A\to\invlim_iB_i$ by 
 \[ u_\infty(a) = (u_0(a),u_1(a),u_2(a),\dotsc), \]
 and this is the unique map with $p_ku_\infty=u_k$ for all $k$. 
\end{remark}

\begin{example}\label{eg-invlim-intersect}
 Suppose we have a chain of subgroups 
 \[ B_0 \geq B_1 \geq B_2 \geq \dotsb \]
 and we take the maps $f_i$ to be the inclusion maps.  Then we see
 from the definitions that 
 \[ \invlim_iB_i = 
     \{(b,b,b,\dotsb)\st b\in\bigcap_iB_i\} 
      \simeq\bigcap_iB_i.
 \]
\end{example}
\begin{example}\label{eg-invlim-projections}
 Suppose we have a system of groups $C_i$ (for $i\in\N$).  Put 
 $B_k=\prod_{i=0}^kC_i$ and let $f_k\:B_{k+1}\to B_k$ be the
 obvious projection map.  If $b\in\invlim_kB_k$ then
 $b_k\in\prod_{i=0}^kC_i$ so $b_{kk}\in C_k$.  We can thus define
 \[ d \: \invlim_kB_k \to \prod_kC_k \]
 by 
 \[ d(b) = (b_{00},b_{11},b_{22},\dotsc). \]
 By the definition of $\invlim_kB_k$, we have $b_{kj}=b_{jj}$ for
 $j\leq k$.  Using this, we see that $d$ is an isomorphism. 
\end{example}
\begin{example}\label{eg-invlim-zero}
 Suppose we have a tower in which the groups are arbitrary but the
 maps are all zero.  Then the map $D$ is the identity, so both
 $\invlim$ and $\invlim^1$ are zero.
\end{example}
\begin{example}\label{eg-lim-adjunction}
 Suppose we have a sequence
 \[ A_0 \xra{f_0} A_1 \xra{f_1} A_2 \xra{f_2} A_3 \xra{f_3} \dotsb,
 \]
 and another group $B$.  This gives us a tower
 \[ \Hom(A_0,B) \xla{f_0^*} 
    \Hom(A_1,B) \xla{f_1^*} 
    \Hom(A_2,B) \xla{f_2^*} 
    \Hom(A_3,B) \xla{f_3^*} \dotsb
 \]
 The elements of $\invlim_i\Hom(A_i,B)$ are precisely the cones from
 the sequence $\{A_i\}_{i\in\N}$ to $B$, which biject with
 homomorphisms from $\colim_iA_i$ to $B$.  In other words, we have
 \[ \invlim_i\Hom(A_i,B) = \Hom(\colim_iA_i,B). \] 
\end{example}
\begin{example}\label{eg-padic-tower}
 Fix a prime $p$.  We then have a tower
 \[ \Z/p \xla{} \Z/p^2 \xla{} \Z/p^3 \xla{} \Z/p^4 \xla{} \dotsb. \]
 The inverse limit is called \emph{the ring of $p$-adic integers}, and
 is denoted by $\Zp$.  We will investigate it in more detail in
 Section~\ref{sec-completion}.   We can also form a tower
 \[ \Z/0! \xla{} \Z/1! \xla{} \Z/2! \xla{} \Z/3! \xla{} \dotsb \]
 The inverse limit is called \emph{the profinite completion of $\Z$},
 and is denoted by $\widehat{\Z}$.  Using the Chinese Remainder
 Theorem (Proposition~\ref{prop-chinese}) one can show that
 $\widehat{\Z}=\prod_p\Zp$.  For yet another description, recall that 
 \[ \QZ = \bigcup_n(\QZ)[n!] = \colim_n(\QZ)[n!], \]
 so 
 \[ \End(\QZ) = \invlim_n\Hom((\QZ)[n!],\QZ). \]
 Now we have a map $m\:\Z\to\End(\QZ)$ defined by $m(k)=k.1_{\QZ}$,
 and this fits in a commutative diagram
 \begin{center}
  \begin{tikzcd}
      \Z \arrow[rr,"m"] \arrow[twoheadrightarrow,d] && \End(\QZ) \arrow[d,"\text{restrict}"] \\
      \Z/n! \arrow[rr,"m_n"'] && \Hom((\QZ)[n!],\QZ).
     \end{tikzcd}
    \end{center}
    
 Using the fact that $(\QZ)[n!]$ is generated by $(1/n!)+\Z$ we see
 that $m_n$ is an isomorphism.  By passing to inverse limits, we
 obtain an isomorphism $m_\infty\:\widehat{\Z}\to\End(\QZ)$.
\end{example}

\begin{proposition}
 Suppose we have a commutative diagram as shown
 \begin{center}
  \begin{tikzcd}
   A_0 \arrow[d,"p_0"'] &
   A_1 \arrow[d,"p_1"'] \arrow[l,"f_0"'] &
   A_2 \arrow[d,"p_2"'] \arrow[l,"f_1"'] &
   A_3 \arrow[d,"p_3"'] \arrow[l,"f_2"'] &
   \dotsb           \arrow[l,"f_3"'] \\
   B_0 &
   B_1              \arrow[l,"g_0"] &
   B_2              \arrow[l,"g_1"] &
   B_3              \arrow[l,"g_2"] &
   \dotsb           \arrow[l,"g_3"]
  \end{tikzcd}
 \end{center}  
 and we define $p=\prod_ip_i\:\prod_iA_i\to\prod_iB_i$.  Then the
 central square below commutes, so there are induced maps $p_\infty$
 and $p^1_\infty$ as shown.
 \begin{center}
  \begin{tikzcd}
   \invlim_iA_i \arrow[d,"p_\infty"'] \arrow[rightarrowtail,r] &
   \prod_iA_i \arrow[d,"p"'] \arrow[r,"D"] &
   \prod_iA_i \arrow[d,"p"] \arrow[twoheadrightarrow,r] &
   \invlim{}_i^1A_i \arrow[d,"p^1_\infty"] \\
   \invlim_iB_i \arrow[rightarrowtail,r] &
   \prod_iB_i \arrow[r,"D"'] &
   \prod_iB_i \arrow[twoheadrightarrow,r] &
   \invlim_i{}^1B_i.
  \end{tikzcd}
 \end{center}
 
\end{proposition}
\begin{proof}
 Clear from the definitions.
\end{proof}

\begin{proposition}\label{prop-limit-exact}
 Suppose we have a commutative diagram as shown, in which all the
 columns are short exact:
 \begin{center}
  \begin{tikzcd}
   A_0 \arrow[rightarrowtail,d,"p_0"'] &
   A_1 \arrow[rightarrowtail,d,"p_1"'] \arrow[l,"f_0"'] &
   A_2 \arrow[rightarrowtail,d,"p_2"'] \arrow[l,"f_1"'] &
   A_3 \arrow[rightarrowtail,d,"p_3"'] \arrow[l,"f_2"'] &
   \dotsb            \arrow[l,"f_3"'] \\
   B_0 \arrow[twoheadrightarrow,d,"q_0"'] &
   B_1 \arrow[twoheadrightarrow,d,"q_1"'] \arrow[l,"g_0"'] &
   B_2 \arrow[twoheadrightarrow,d,"q_2"'] \arrow[l,"g_1"'] &
   B_3 \arrow[twoheadrightarrow,d,"q_3"'] \arrow[l,"g_2"'] &
   \dotsb            \arrow[l,"g_3"'] \\
   C_0 &
   C_1               \arrow[l,"h_0"] &
   C_2               \arrow[l,"h_1"] &
   C_3               \arrow[l,"h_2"] &
   \dotsb            \arrow[l,"h_3"]
  \end{tikzcd}
 \end{center}  
 Then there is an associated exact sequence
 \begin{center}
  \begin{tikzcd}
   \invlim_iA_i \arrow[rightarrowtail,r,"p_\infty"] &
   \invlim_iB_i \arrow[r,"q_\infty"] &
   \invlim_iC_i \arrow[r,"\dl"] &
   \invlim^1_iA_i \arrow[r,"p^1_\infty"] &
   \invlim^1_iB_i \arrow[twoheadrightarrow,r,"q^1_\infty"] &
   \invlim^1_iC_i
  \end{tikzcd}
 \end{center}  
\end{proposition}
\begin{proof}
 Apply the Snake Lemma to the diagram
 \begin{center}
  \begin{tikzcd}
   \prod_iA_i \arrow[rightarrowtail,r,"\prod_ip_i"] \arrow[d,"D"'] &
   \prod_iB_i \arrow[twoheadrightarrow,r,"\prod_iq_i"] \arrow[d,"D"'] &
   \prod_iC_i                      \arrow[d,"D"] \\
   \prod_iA_i \arrow[rightarrowtail,r,"\prod_ip_i"'] &
   \prod_iB_i \arrow[twoheadrightarrow,r,"\prod_iq_i"'] &
   \prod_iC_i
  \end{tikzcd}
 \end{center}
 in which the rows are easily seen to be short exact. 
\end{proof}

In practice the groups $\invlim^1_iA_i$ are usually either zero, or
enormous and untractable.  We will thus be very interested in results
that force them to be zero.  

\begin{proposition}\label{prop-surjective-tower}
 Suppose we have a tower in which the maps $f_i\:A_{i+1}\to A_i$ are
 all surjective.  Then $\invlim^1_iA_i=0$, and the projection maps
 $p_k\:\invlim_iA_i\to A_k$ are all surjective.  
\end{proposition}
\begin{proof}
 Consider an element $a\in\prod_iA_i$.  We will choose elements
 $b_k\in A_k$ recursively as follows: we start with $b_0=0$, and then
 take $b_n$ to be any element with $f_{n-1}(b_n)=b_{n-1}-a_{n-1}$.
 These elements $b_k$ give an element $b\in\prod_iA_i$ with $D(b)=a$,
 so $D$ is surjective and $\invlim^1_iA_i=\cok(D)=0$.

 Now suppose we have an element $a\in A_k$.  Define $c_i=f_{k,i}(a)$
 for all $i\leq k$.  Then define $c_i\in A_i$ recursively for $i>k$ by 
 choosing $c_i$ to be any element with $f_{i-1}(c_i)=c_{i-1}$.  This
 gives an element $c\in\invlim_iA_i$ with $p_k(c)=a$.
\end{proof}

\begin{definition}\label{defn-nilpotent-tower}
 We say that a tower $A_0\xla{f_0}A_1\xla{f_1}\dotsb$ is
 \emph{nilpotent} if for all $i$ there exists $j>i$ such that
 $f_{ji}=0\:A_j\to A_i$.
\end{definition}

\begin{proposition}\label{prop-nilpotent-tower}
 For a nilpotent tower as above, we have
 $\invlim_iA_i=\invlim_i^1A_i=0$. 
\end{proposition}
\begin{proof}
 Define $E\:\prod_iA_i\to\prod_iA_i$ by 
 \[ E(a)_i = \sum_{j=i}^\infty f_{j,i}(a_j). \]
 Although the sum is formally infinite, the nilpotence hypothesis
 means that there are only finitely many nonzero terms, so the
 expression is meaningful.  It is then not hard to check that
 $DE=ED=1$, so the kernel and cokernel of $D$ are zero.
\end{proof}

\begin{definition}\label{defn-mittag-leffler}
 Consider a tower $A_0\xla{f_0}A_1\xla{f_1}\dotsb$, so for each $i$ we
 have a descending chain of subgroups
 \[ A_i \geq
     f_{i+1,i}(A_{i+1}) \geq
     f_{i+2,i}(A_{i+2}) \geq
     f_{i+3,i}(A_{i+3}) \geq \dotsb 
 \]
 We say that the tower is \emph{Mittag-Leffler} if for each $i$ there
 exists $j\geq i$ such that $f_{ki}(A_k)=f_{ji}(A_j)$ for all
 $k\geq j$ (so the above chain is eventually constant).
\end{definition}

\begin{example}\label{eg-mittag-leffler}
 Towers of surjections are Mittag-Leffler, as are nilpotent towers.
\end{example}

\begin{proposition}\label{prop-finite-ml}
 If all the groups $A_i$ are finite, then the tower is
 Mittag-Leffler.  Similarly, if the groups $A_i$ are
 finite-dimensional vector spaces over a field $K$, and the maps $f_i$
 are all $K$-linear, then the tower is Mittag-Leffler.
\end{proposition}
\begin{proof}
 In the first case, we just choose $j\geq i$ such that the order
 $|f_{ji}(A_j)|$ is as small as possible; it then follows that
 $f_{ki}(A_k)=f_{ji}(A_j)$ for $k\geq i$.  In the second case, use
 dimensions instead of orders.
\end{proof}

\begin{proposition}\label{prop-mittag-leffler}
 If $A$ is a Mittag-Leffler tower, we have $\invlim_i^1A_i=0$.
\end{proposition}
\begin{proof}
 By the Mittag-Leffler condition, there is a subgroup $A'_i\leq A_i$
 such that $f_{ji}(A_j)=A'_i$ for all sufficiently large $j$.  Thus,
 for $j$ very large we have both $A'_{i+1}=f_{j,i+1}(A_j)$ and
 $A'_i=f_{ji}(A_j)=f_i(f_{j,i+1}(A_j))=f_i(A'_{i+1})$.  Thus, the
 groups $A'_i$ form a subtower of $A$, with surjective maps
 $f'_i\:A'_{i+1}\to A'_i$.  Now put $A''_i=A_i/A'_i$, so there are
 induced maps $f''_i\:A''_{i+1}\to A''_i$, giving a third tower.  If
 $j$ is much larger than $i$ we have $f_{ji}(A_j)=A'_i$ and so
 $f''_{ji}=0\:A''_j\to A''_i$; this shows that the tower $A''$ is
 nilpotent.  Now apply Proposition~\ref{prop-limit-exact} to the
 short exact sequence $A'\to A\to A''$ to give an exact sequence 
 \begin{center}
  \begin{tikzcd}
   \invlim_iA'_i \arrow[rightarrowtail,r] &
   \invlim_iA_i \arrow[r] &
   \invlim_iA''_i \arrow[r,"\dl"] &
   \invlim^1_iA'_i \arrow[r] &
   \invlim^1_iA_i \arrow[twoheadrightarrow,r] &
   \invlim^1_iA''_i
  \end{tikzcd}
 \end{center}  
 Here $\invlim_i^1A'_i=0$ because the maps in $A'$ are surjective, and
 $\invlim_i^1A''_i=0$ because $A''$ is nilpotent, so
 $\invlim_i^1A_i=0$ as claimed.  (We also have $\invlim_iA''_i=0$ and
 so $\invlim_iA'_i=\invlim_iA_i$.)
\end{proof}

We also have the following result analogous to
Proposition~\ref{prop-colimit-cofinal}, which again indicates that
$\invlim_iA_i$ and $\invlim^1_iA_i$ only depend on the asymptotic
behaviour of the tower.

\begin{proposition}\label{prop-limit-cofinal}
 Suppose we have a tower
 \[ A_0 \xla{f_0} A_1 \xla{f_1} A_2 \xla{f_2} A_3 \xla{f_3} \dotsb \]
 and a nondecreasing function $u\:\N\to\N$ such that $u(i)\to\infty$
 as $i\to\infty$.  Put 
 \[ g_i = f_{u(i+1),u(i)} = 
     (A_{u(i+1)} \xra{f_{u(i+1)-1}} A_{u(i+1)-1}  \dotsb 
       \xra{f_{u(i)}} A_{u(i)}),
 \]
 so we have a tower
 \[ A_{u(0)} \xla{g_0} A_{u(1)} \xla{g_1} A_{u(2)}
     \xla{g_2} A_{u(3)} \xla{g_3} \dotsb
 \]
 Then there are canonical isomorphisms
 $\invlim_jA_{u(j)}=\invlim_iA_i$ and 
 $\invlim^1_jA_{u(j)}=\invlim^1_iA_i$.
\end{proposition}
\begin{proof}
 Define $v\:\N\to\N$ by $v(i)=\min\{j\st u(j)\geq i\}$.  We will
 construct a diagram as follows:
 \begin{center}
  \begin{tikzcd}
   \prod_jA_{u(j)} \arrow[r,"\phi"] \arrow[d,"D'"'] &
   \prod_iA_i      \arrow[r,"\psi"] \arrow[d,"D"']    &
   \prod_jA_{u(j)} \arrow[d,"D'"] \\
   \prod_jA_{u(j)} \arrow[r,"\lm"']  &
   \prod_iA_i      \arrow[r,"\mu"']  &
   \prod_jA_{u(j)} 
  \end{tikzcd}
 \end{center}  
 The maps are:
 \begin{align*}
  D'(b)_j &= b_j - f_{u(j+1),u(j)}(b_{j+1}) &
  D(a)_i &= a_i - f_{i+1,i}(a_{i+1}) \\
  \phi(b)_i &= f_{uv(i),i}(b_{v(i)}) &
  \psi(a)_j &= a_{u(j)} \\
  \lm(b)_i &= \sum_{u(j)=i}b_j &
  \mu(a)_j &= \sum_{u(j)\leq i<u(j+1)} f_{i,u(j)}(a_i).
 \end{align*}
 Thus $D$ and $D'$ are the usual maps whose kernels and cokernels are
 the $\invlim$ and $\invlim^1$ groups under consideration.  We claim
 that the diagram commutes.  To see this, consider a point
 $b\in\prod_jA_{u(j)}$.  We then have 
 \begin{align*}
  D(\phi(b))_i &= \phi(b)_i - f_{i+1,i}(\phi(b)_{i+1}) \\
               &= f_{uv(i),i}(b_{v(i)}) - f_{uv(i+1),i}(b_{v(i+1)}) \\
  \lm(D'(b))_i &= \sum_{u(j)=i} D'(b)_j 
               = \sum_{u(j)=i}(b_j-f_{u(j+1),u(j)}b_{j+1}).
 \end{align*}
 If $u^{-1}\{i\}=\emptyset$ we find that $v(i+1)=v(i)$ and so
 $D(\phi(b))_i=0=\lm(D'(b))_i$.  If $u^{-1}\{i\}$ is nonempty then it
 will be an interval, say $u^{-1}\{i\}=\{j_0,\dotsc,j_1-1\}$.  In our
 expression for $\lm(D'(b))_i$, the map $f_{u(j+1),u(j)}$ is just the
 identity except when $j=j_1-1$.  The expression therefore cancels
 down to $b_{j_0}-f_{u(j_1),i}(b_{j_1})$.  On the
 other hand, we also find that $v(i)=j_0$ and $v(i+1)=j_1$, so
 \[ D(\phi(b))_i = f_{ii}(b_{j_0}) - f_{u(j_1),i}(b_{j_1}) =
     \lm(D'(b))_i.
 \]
 Thus, the left square commutes.  For the right square, consider an
 element $a\in\prod_iA_i$.  We have
 \begin{align*}
  \mu(D(a))_j
    &= \sum_{u(j)\leq i<u(j+1)} f_{i,u(j)}(D(a)_i) \\
    &= \sum_{u(j)\leq i<u(j+1)}
           (f_{i,u(j)}(a_i)-f_{i+1,u(j)}(a_{i+1})) \\
    &= a_{u(j)} - f_{u(j+1),u(j)}(a_{u(j+1)}) \\
    &= \psi(a)_j - f_{u(j+1),u(j)}(\psi(a)_{j+1}) = D'(\psi(a))_j
 \end{align*}
 as required.  We therefore have induced maps
 \[ \invlim_jA_{u(j)} \xra{\phi'} \invlim_iA_i
      \xra{\psi'} \invlim_jA_{u(j)}
 \]
 and 
 \[ \invlim^1_jA_{u(j)} \xra{\lm'} \invlim^1_iA_i
      \xra{\mu'} \invlim^1_jA_{u(j)}.
 \]
 It is straightforward to check that $\phi'\psi'$ and $\psi'\phi'$ are
 the respective identity maps.  Now define
 $\sg\:\prod_iA_i\to\prod_iA_i$ by 
 \[ \sg(a)_i = \sum_{i\leq h<uv(i)}f_{h,i}(a_h). \]
 We claim that $\lm\mu+D\sg=1$ (which implies that $\lm'\mu'=1$).
 The proof that $\lm(\mu(a))_i+D(\sg(a))_i=a_i$ splits into two cases,
 depending on whether $i\in\image(u)$ or not.  If $i\not\in\image(u)$
 we find that the inequality $i\leq u(v(i))$ cannot be an equality, so
 $v(i+1)=v(i)$.  Using this we deduce that $D(\sg(a))_i$ is the
 difference between two sums that mostly have the same terms, and thus
 that $D(\sg(a))_i=a_i$.  On the other hand, we have
 $\lm(\mu(a))_i=\sum_{u(j)=i}\mu(a)_j$ which is zero as the sum has no
 terms.  Now consider instead the case where $i\in\image(u)$.  Let $j_1$
 be the largest integer such that $u(j_1)=i$, and put $i'=u(j_1+1)>i$.
 From the definitions we have 
 \[ \lm(\mu(a))_i =
      \sum_{u(j)=i}\sum_{u(j)\leq h<u(j+1)}f_{h,u(j)}(a_h).
 \]
 However, the inner summation is empty unless $j=j_1$, so the formula
 reduces to $\lm(\mu(a))_i=\sum_{i\leq h<i'}f_{h,i}(a_h)$.  We also
 have $uv(i)=i$, so $\sg(a)_i=0$, and $uv(i+1)=u(j_1+1)=i'$, so 
 $\sg(a)_{i+1}=\sum_{i<h<i'}h_{h,i+1}(a_h)$.  From this it is easy to
 see that $\lm(\mu(a))_i+D(\sg(a))_i=a_i$ as required.

 We now consider instead the map $\mu\lm$.  Define
 $\tau\:\prod_jA_{u(j)}\to\prod_jA_{u(j)}$ by 
 \[ \tau(b)_j = \sum_{k<j,\;u(k)=u(j)} b_k. \]
 We claim that $\mu\lm=1+D\tau$ (which implies that $\mu'\lm'=1$).
 The proof that $\mu(\lm(b))_j=b_j+D(\tau(b))_j$ again splits into two
 cases.  First suppose that $u(j+1)=u(j)$.  It is then immediate from
 the definitions that $\mu(\lm(b))_j=0$.  On the other hand, the sums
 defining $\tau(b)_j$ and $\tau(b)_{j+1}$ differ only by a single
 term, so $D(\tau(b))_j=-b_j$, as required.  Now suppose instead
 that $u(j+1)>u(j)$.  In this case we have $\tau(b)_{j+1}=0$, so
 \[ b_j+D(\tau(b))_j = b_j+\tau(b)_j =
     b_j + \sum_{k<j,\;u(k)=u(j)} b_k = \sum_{u(k)=u(j)} b_k.
 \]
 On the other hand, we have
 \[ \mu(\lm(b))_j =
     \sum_{u(j)\leq i<u(j+1)}\sum_{u(k)=i}f_{i,u(j)}(b_k).
 \]
 The inner sum has no terms unless $i=u(j)$, and in that context
 $f_{i,u(j)}$ is the identity, so the above reduces to 
 \[ \mu(\lm(b))_j = \sum_{u(k)=u(j)} b_k = b_j+D(\tau(b))_j \]
 as required.
\end{proof}

\section{Completion and derived completion}
\label{sec-completion}

\begin{definition}\label{defn-p-completion}
 Let $p$ be a prime.  For any abelian group $A$ we have a tower of
 surjections 
 \[ 0 = A/p^0 \xla{} A/p \xla{} A/p^2 \xla{} A/p^3 \xla{} \dotsb \]
 We write $A_p$ for the inverse limit of this tower, and call this the
 \emph{$p$-completion} of $A$.  In particular, we have a group $\Zp$,
 whose elements are called \emph{$p$-adic integers}.  For any $A$ we
 have a homomorphism $\eta\:A\to A_p$ given by 
 \[ \eta(a) = (a+p^0A,a+p^1A,a+p^2A,a+p^3A,\dotsc). \]
 We say that $A$ is \emph{$p$-complete} if $\eta$ is an isomorphism.
\end{definition}

\begin{example}\label{eg-finite-completion}
 Let $A$ be a finite abelian group, so $A$ splits as $B\oplus C$ say,
 where $|B|$ is a power of $p$ and $|C|$ is coprime to $p$.  For all
 $k$ we have $p^kC=C$, and for large $k$ we have $p^kB=0$.  It follows
 that $A_p=B$.  In particular, we see from this that finite abelian
 $p$-groups are $p$-complete.
\end{example}
\begin{example}\label{eg-divisible-completion}
 Suppose that $A$ is divisible.  Then for all $k$ we have $p^kA=A$, so
 $A/p^k=0$; it follows that $A_p=0$.  In particular, we have $\Q_p=0$
 and $(\Q/\Z)_p=0$.
\end{example}
\begin{remark}\label{rem-padic-rationals}
 The symbol $\Q_p$ is often used for $\Q\ot\Zp$, which is not zero;
 it is known as the field of $p$-adic rationals.  However, this is
 different from the group that we have called $\Q_p$, which is
 trivial. 
\end{remark}

\begin{proposition}\label{prop-completion-idempotent}
 For all $A$ and $k\geq 0$ the projection $\pi_k\:A\to A/p^kA$ induces
 an isomorphism $A_p/p^kA_p=A/p^kA$.
\end{proposition}
\begin{proof}
 For notational simplicity, we will treat only the case $k=1$.  The
 general case is essentially the same, after we have used
 Proposition~\ref{prop-limit-cofinal} to identify $A_p$ with the
 inverse limit of the sequence
 \[ A/p^k \xla{} A/p^{2k} \xla{} A/p^{3k} \xla{} \dotsb. \]
 First, the projection $\pi_1\:\prod_i A/p^i\to A/p$ restricts to give
 a homomorphism $\phi\:A_p=\invlim_iA/p^i\to A/p$.
 Proposition~\ref{prop-surjective-tower} tells us that this is
 surjective.  As $A/p$ has exponent $p$, the subgroup $pA_p$ is
 contained in the kernel.  We need to prove that the kernel is
 precisely $pA_p$.  Suppose that $a\in\ker(\phi)$.  Choose $a_i\in A$
 representing the component of $a$ in $A/p^iA$.  As $A/p^0A=0$ and
 $\phi(a)=0$ we can take $a_0=a_1=0$.  By the definition of the
 inverse limit we see that $a_i$ is the image of $a_{i+1}$ in
 $A/p^iA$, so $a_{i+1}=a_i+p^ib_i$ for some $b_i\in A$ (and we may
 take $b_0=0$).  Now put $c_n=\sum_{i=1}^np^{i-1}b_i\in A$.  It is
 visible that $c_{n+1}=c_n+p^nb_{n+1}=c_n\pmod{p^nA_n}$, so the cosets
 $c_n+p^nA$ define an element $c\in A_p$.  We also see by induction
 that $pc_k=a_{k+1}=a_k\pmod{p^kA_k}$, so $pc=a$.  Thus $a\in pA_p$ as
 claimed.  
\end{proof}
\begin{remark}\label{rem-completion-idempotent}
 We see from the proposition that $A_p=(A_p)_p$, so $A_p$ is
 $p$-complete, as one would expect.  There is a subtlety analogous to
 Remark~\ref{rem-idempotent-monad} here; we leave it to the reader to
 check that the two natural maps $A_p\to(A_p)_p$ are the same, and
 they are both isomorphisms. 
\end{remark}

We next examine the structure of $\Zp$ in more detail.  First, as the
groups $\Z/p^k$ have canonical ring structures, and the maps
$\Z/p^{k+1}\to\Z/p^k$ are ring maps, we see that $\Zp$ is a subring
of $\prod_k\Z/p^k$.   We will write $\pi_k$ for the projection
$\Zp\to\Z/p^k$, which is a surjective ring homomorphism.  Note also
that for $n\in\Z$ we have $\eta(n)=0$ iff $\pi_k\eta(n)=0$ for all $k$
iff $n$ is divisible by $p^k$ for all $k$ iff $n=0$.  Thus, $\eta$
gives an injective ring map $\Z\to\Zp$.  We will usually suppress
notation for this and regard $\Z$ as a subring of $\Zp$.

\begin{definition}\label{defn-padic-metric}
 For $a\in\Zp$ we define $v(a)=\min\{k\st\pi_k(a)\neq 0\}$ (or
 $v(a)=\infty$ if $\pi_k(a)=0$ for all $k$, which means that $a=0$).
 We also define $d(a,b)=p^{-v(a-b)}$, with the convention
 $p^{-\infty}=0$.  
\end{definition}

\begin{proposition}\label{prop-padic-metric}
 The function $d$ defines a metric on $\Zp$ (called the
 \emph{$p$-adic metric}), with respect to which it is complete and
 compact.  Moreover, the subspace $\Z$ is dense.
\end{proposition}
\begin{proof}
 It is clear that $d(a,b)=d(b,a)$, and that this is nonnegative and
 vanishes if and only if $a=b$.  This just leaves the triangle
 inequality $d(a,c)\leq d(a,b)+d(b,c)$.  This is clear if $a=b$ or
 $b=c$, so suppose that $a\neq b\neq c$.  Put
 $m=\min(v(a-b),v(b-c))$.  For $k<m$ we have
 $\pi_k(a)=\pi_k(b)=\pi_k(c)$.  It follows that $v(c-a)\geq m$, so
 \[ d(a,c)\leq p^{-m}=\max(d(a,b),d(b,c))\leq d(a,b)+d(b,c) \]
 as required.  

 Now consider a Cauchy sequence $(a_0,a_1,a_2,\dotsc)$ in $\Zp$.
 Given $k\in\N$ we can choose $m$ such that $d(a_i,a_j)<p^{-k}$ for
 all $i,j\geq m$.  From the definition of $d$, this means that the
 element $\pi_k(a_i)\in\Z/p^k$ is independent of $i$ for $i\geq m$.
 Let $b_k$ denote this element.  If $i$ is large enough we will have
 $b_k=\pi_k(a_i)$ and also $b_{k+1}=\pi_{k+1}(a_i)$; using this we see
 that the projection $\Z/p^{k+1}\to\Z/p^k$ sends $b_{k+1}$ to $b_k$.
 Thus, sequence $(b_0,b_1,b_2,\dotsc)$ is an element $b\in\Zp$, and
 by construction $a_i\to b$ as $i\to\infty$.  This shows that $\Zp$
 is compact.

 Next, suppose we are given $k\in\N$, and we put
 $T_k=\{0,1,\dotsc,p^k-1\}$ and $a\in\Zp$.  The map $T_k\to\Z/p^k$ is
 a bijection, so for any $a\in\Zp$ there is a unique $m\in T_k$ with
 $\pi_k(m)=\pi_k(a)$, so $d(m,a)<p^{-k}$.  Using this we see that $\Z$
 is dense in $\Zp$.  Next, recall that an \emph{$\ep$-net} in a
 metric space $X$ is a finite set $F\sse X$ such that every point is
 within $\ep$ of a point in $F$, that $X$ is \emph{totally bounded} if
 it has an $\ep$-net for every $\ep>0$, and that a complete metric
 space is compact if and only if it is totally bounded.  The set $T_k$
 is a $2^{-k}$-net, so $\Zp$ is totally bounded, so it is compact.
\end{proof}

\begin{proposition}\label{prop-digits}
 Put $D=\{0,1,2,\dotsc,p-1\}$ (the set of $p$-adic digits).  Then
 there is a bijection  
 \[ \sg \: \prod_{i=0}^\infty D \to \Zp \]
 given by $\sg(u)=\sum_iu_ip^i$ (a convergent sum with respect to the
 $p$-adic metric).  In particular, $\Zp$ is uncountable.  We call
 $\sg^{-1}(a)$ the \emph{base $p$ expansion} of $a$. 
\end{proposition}
\begin{proof}
 It is elementary that the corresponding map
 $\prod_{i=0}^{k-1}D\to\Z/p^k$ is a bijection, and the claim follows
 by passing to inverse limits.
\end{proof}

\begin{proposition}\label{prop-Zp-omni}
 The ring $\Zp$ is torsion free and is an integral domain.  It is
 also a local ring, with $p.\Zp$ being the unique maximal ideal.  The
 group of units is 
 \[ \Zp^\tm = \{a\in\Zp\st\pi_0(a)\neq 0\} = \Zp\sm p\Zp, \]
 and every nonzero element is a unit times $p^k$ for some $k$.
\end{proposition}
\begin{proof}
 First, consider an element $a\in\Zp$ with $\pi_0(a)=1$.  We see from
 Proposition~\ref{prop-completion-idempotent} that $a=1-px$ for some
 $x\in\Zp$.  It follows easily that the series $\sum_i(px)^i$ is
 Cauchy, so it converges to some $b\in\Zp$, and we find that $ab=1$.
 More generally, suppose merely that $\pi_0(a)\neq 0$ in $\Z/p$.  As
 $\Z/p$ is a field, we can find $b\in\Z$ such that $\pi_0(b)$ is
 inverse to $\pi_0(a)$.  We then find that $ab$ is invertible by the
 previous case, and it follows that $a$ is invertible.  Conversely, as
 $\pi_0$ is a ring map it certainly sends units to nonzero elements.
 We therefore see that 
 \[ \Zp^\tm = \{a\in\Zp\st\pi_0(a)\neq 0\} = \Zp\sm p\Zp \]
 as claimed.  Note also that $\Zp/p\Zp$ is the field $\Z/p$, so
 $p\Zp$ is a maximal ideal.  If $\mathfrak{m}$ is any maximal ideal
 we must have $\mathfrak{m}\cap\Zp^\tm=\emptyset$, so
 $\mathfrak{m}\leq p\Zp$, so $\mathfrak{m}=p\Zp$ by maximality.
 Thus, $\Zp$ is a local ring.  

 Next, using base $p$ expansions we see easily that multiplication by
 $p$ is injective, and that every nonzero element $a\in\Zp$ can be
 written as $a=p^kb$ for some $k\geq 0$ and $b$ with $\pi_0(b)\neq 0$,
 so $b\in\Zp^\tm$.  As multiplication by $p^k$ is injective and $b$
 is invertible we see that multiplication by $a$ is injective.  This
 means that $\Zp$ is an integral domain.  By considering
 $a\in\Z\subset\Zp$ we also see that $\Zp$ is torsion free. 
\end{proof}

We can now understand the completion of free modules.
\begin{definition}\label{defn-AZ}
 Let $I$ be a set, and let $f$ be a function from $I$ to $\Zp$.  We
 say that $f$ is \emph{asymptotically zero} if for all $k$, the set
 $\{i\st v(f(i))<k\}$ is finite.  We write $AZ(I)$ for the set of
 asymptotically zero maps, which is a group under addition.
\end{definition}
\begin{proposition}
 The completion $\Z[I]_p$ is naturally isomorphic to $AZ(I)$.
\end{proposition}
\begin{proof}
 First, we put 
 \[ (\Z/p^k)[I] =
     \{f\:I\to\Z/p^k\st \{i\st f(i)\neq 0\} \text{ is finite }\}.
 \]
 We write $\pi_k$ for the projection $\Z\to\Z/p^k$, or the projection
 $\Zp\to\Z/p^k$.  We then write $\pi'_k(f)=\pi_k\circ f$; this defines
 a map $AZ(I)\to(\Z/p^k)[I]$, which we can restrict to
 $\Z[I]\leq AZ(I)$.  Note that $\pi'_k(f)=0$ iff $f(i)\in p^k\Z_p$ for
 all $i$, in which case we can define $g=f/p^k\:I\to\Z_p$ and we find
 that $g$ is again asymptotically zero, so $f\in p^kAZ(I)$.  This
 shows that $\pi'_k$ induces an isomorphism
 $AZ(I)/p^kAZ(I)\to(\Z/p^k)[I]$.  By a similar argument, it also
 induces an isomorphism $\Z[I]/p^k\Z[I]\to(\Z/p^k)[I]$.  Thus, we
 have $\Z[I]_p=\invlim_k(\Z/p^k)[I]$.  The maps
 $\pi'_k\:AZ(I)\to(\Z/p^k)[I]$ therefore assemble to give a
 homomorphism $\pi'\:AZ(I)\to\Z[I]_p$.  In the opposite direction,
 suppose we have $g\in\Z[I]_p$.  For each $k\geq 0$ we therefore have
 $\pi_k(g)\in\Z[I]/p^k\Z[I]=(\Z/p^k)[I]$, so there is a unique map 
 \[ g_k \: I \to \{0,1,\dotsc,p^k-1\} \]
 with $\pi_k(g)(i) = g_k(i)\pmod{p^k}$ for all $k$.  Note that the set
 $\{i\st g_k(i)\neq 0\}$ is finite for all $k$.  If we fix $i$, we
 find that the sequence $\{g_k(i)\}_{k\geq 0}$ is Cauchy, converging
 to some element $g_\infty(i)\in\Zp$ say, and we have
 $g_\infty(i)=g_k(i)\pmod{p^k}$ for all $k$.  This means that
 $g_\infty\in AZ(I)$ and $\pi'(g_\infty)=g$, so $\pi'$ is surjective.
 We also have $\pi'(h)=0$ iff $h(i)$ is divisible by $p^k$ for all $i$
 and $k$, which implies that $h=0$.  Thus, the map $\pi'$ is an
 isomorphism. 
\end{proof}

\begin{proposition}\label{prop-completion-monoidal}
 For any abelian groups $A$ and $B$ there is a natural map
 $\mu\:A_p\ot B_p\to(A\ot B)_p$, which induces an isomorphism
 $(A_p\ot B_p)_p\to(A\ot B)_p$.   
\end{proposition}
\begin{proof}
 Note that $p^k.1_{A\ot B}=(p^k.1_A)\ot 1_B=1_A\ot(p^k.1_B)$.  Using
 the right exactness of tensor products we see that  
 \[ (A\ot B)/p^k = (A/p^k)\ot B = A\ot(B/p^k) = (A/p^k)\ot(B/p^k). \]
 Combining this with Proposition~\ref{prop-completion-idempotent}
 gives 
 \[ (A_p\ot B_p)/p^k = 
     (A_p/p^k)\ot(B_p/p^k) = 
      (A/p^k)\ot (B/p^k) = (A\ot B)/p^k.
 \] 
 Passing to inverse limits gives an isomorphism
 $(A_p\ot B_p)_p=(A\ot B)_p$.  We can compose this with the map
 $\eta\:A_p\ot B_p\to(A_p\ot B_p)_p$ to get a map
 $\mu\:A_p\ot B_p\to(A\ot B)_p$, which is what we most often need for
 applications. 
\end{proof}

\begin{proposition}\label{prop-completion-tensor}
 There is a natural map $\Zp\ot A\to A_p$, which is an isomorphism
 when $A$ is finitely generated.
\end{proposition}
\begin{proof}
 The map is just the composite 
 \[ \Zp\ot A\xra{1\ot\eta} \Zp\ot A_p \xra{\mu} (\Z\ot A)_p = A_p \]
 (or it can be defined more directly by the method used for $\mu$).
 By the classification of finitely generated abelian groups, it will
 suffice to prove that we have an isomorphism when $A=\Z$ or
 $A=\Z/p^k$ or $A=\Z/q^k$ for some prime $q\neq p$.  The case $A=\Z$
 is clear.  When $A=\Z/p^k$ we have $A_p=A$ as in
 Example~\ref{eg-finite-completion}.  We also $\Zp\ot A=\Zp/p^k\Zp$
 by the right exactness of tensoring, and this is the same as
 $\Z/p^k=A$ by Proposition~\ref{prop-completion-idempotent}, so
 $\Zp\ot A=A_p$ as claimed.  Finally, if $q$ is different from $p$
 then $q^k$ is invertible in $\Zp$ so $\Zp\ot\Z/q^k=\Zp/q^k\Zp=0$,
 and similarly $(\Z/q^k)_p=0$ as in
 Example~\ref{eg-finite-completion} again. 
\end{proof}
\begin{corollary}\label{cor-completion-exact}
 If $A\to B\to C$ is an exact sequence of finitely generated abelian
 groups, then the resulting sequence $A_p\to B_p\to C_p$ is also
 exact. 
\end{corollary}
\begin{proof}
 As $\Zp$ is torsion free, Proposition~\ref{prop-tensor-exact} tells
 us that the sequence $\Zp\ot A\to\Zp\ot B\to\Zp\ot C$ is exact.  
\end{proof}

We can now assemble our results to prove something closely analogous
to Proposition~\ref{prop-local-omni}:
\begin{proposition}\label{prop-complete-omni}\nl
 \begin{itemize}
  \item[(a)] The group $A_p$ is always $p$-complete.
  \item[(b)] The map $\eta\:A\to A_p$ is an isomorphism if and
   only if $A$ is $p$-complete.
  \item[(c)] If $A$ is $p$-complete then it can be regarded as a module
   over $\Zp$.
  \item[(d)] Suppose that $f\:A\to B$ is a homomorphism, and that $B$
   is $p$-complete.  Then there is a unique homomorphism
   $f'\:A_p\to B$ such that $f'\circ\eta=f\:A\to B$.
 \end{itemize}
\end{proposition}
\begin{proof}\nl
 \begin{itemize}
  \item[(a)] As mentioned previously, this follows from
   Proposition~\ref{prop-completion-idempotent} by taking inverse
   limits.  
  \item[(b)] This is true by definition, and is only mentioned to
   complete the correspondence with
   Proposition~\ref{prop-local-omni}. 
  \item[(c)] Proposition~\ref{prop-completion-monoidal} gives a map 
   \[ \Zp\ot A = \Zp\ot A_p \to (\Z\ot A)_p = A_p = A. \]
   For $r\in\Zp$ and $a\in A$ we can thus define $ra=\mu(r\ot a)$.
   Equivalently, this is characterised by the fact that
   $\pi_k(ra)=\pi_k(r)\pi_k(a)$ in $A/p^kA$ (where we have used the
   obvious structure of $A/p^kA$ as a module over $\Z/p^k$).  From
   this description it is clear that our multiplication rule is
   associative, unital and distributive, so it makes $A$ into a module
   over $\Zp$. 
  \item[(d)] The map $f\:A\to B$ induces an isomorphism
   $f_p\:A_p\to B_p$, and we have an isomorphism $\eta\:B\to B_p$.  We
   can and must take $f'=\eta^{-1}\circ f_p$.
 \end{itemize}
\end{proof}

While our definition of completion is quite natural and
straightforward, its exactness properties for infinitely generated
groups are very delicate, and they do not relate well to topological
constructions.  We will therefore introduce a different definition
that often agrees with completion, but has better formal properties.

\begin{definition}\label{defn-Loi}
 For any abelian group $A$, we let $A\psb{x}$ denote the group of
 formal power series $v(x)=\sum_{i=0}^\infty a_ix^i$ with $a_i\in A$
 for all $i$.  This is a module over $\Z\psb{x}$ by the obvious rule
 \[ \left(\sum_in_ix^i\right)\left(\sum_ja_jx^j\right) =
     \sum_k\left(\sum_{i=0}^kn_ia_{k-i}\right)x^k.
 \]
 We define
 \begin{align*}
  L_0A &= A\psb{x}/((x-p).A\psb{x}) \\
  L_1A &= \{v(x)\in A\psb{x}\st (x-p)v(x)=0\}.
 \end{align*}
 We can identify $A$ with the set of constant series in $A\psb{x}$,
 and then restrict the quotient map $A\psb{x}\to L_0A$ to $A$ to get a
 natural map $\eta\:A\to L_0A$.  We say that $A$ is
 \emph{Ext-$p$-complete} if $\eta\:A\to L_0A$ is an isomorphism and
 $L_1A=0$.  We also call $L_0A$ the \emph{derived completion} of $A$.
\end{definition}

\begin{remark}\label{rem-Loi-derived}
 Readers familiar with the general theory of derived functors should
 consult Corollary~\ref{cor-Loi-derived} to see why the term is
 appropriate here.
\end{remark}
\begin{remark}\label{rem-Li-auto}
 It will follow from Proposition~\ref{prop-functors-complete} that the
 condition $L_1A=0$ is actually automatic when $\eta\:A\to L_0A$ is an
 isomorphism.  However, it is easier to develop the theory if we have
 both conditions in the initial definition.
\end{remark}

\begin{remark}\label{rem-Lo-monoidal}
 There is an evident product map 
 \[ \mu\:A\psb{x}\ot B\psb{x}\to (A\ot B)\psb{x} \]
 given by
 \[ \mu\left(\left(\sum_ia_ix^i\right)\ot\left(\sum_jb_jx^j\right)\right)
     = \sum_k \left(\sum_{k=i+j}a_i\ot b_j\right)x^k.
 \]
 This induces a map $\mu\:(L_0A)\ot(L_0B)\to L_0(A\ot B)$, which fits
 in a commutative diagram
 \begin{center}
  \begin{tikzcd}
   & A\ot B \arrow[dl,"\eta\ot\eta"'] \arrow[dr,"\eta"] \\
   (L_0A)\ot(L_0B) \arrow[rr,"\mu"'] & & L_0(A\ot B).
  \end{tikzcd}
 \end{center}
\end{remark}

\begin{remark}\label{rem-Loi-six}
 If we have a short exact sequence $A\xra{}B\xra{}C$, we can apply the
 Snake Lemma to the diagram
 \begin{center}
  \begin{tikzcd}
   A\psb{x} \arrow[rightarrowtail,r] \arrow[d,"x-p"'] &
   B\psb{x} \arrow[twoheadrightarrow,r] \arrow[d,"x-p"'] &
   C\psb{x}         \arrow[d,"x-p"] \\
   A\psb{x} \arrow[rightarrowtail,r] &
   B\psb{x} \arrow[twoheadrightarrow,r] &
   C\psb{x}
  \end{tikzcd}
 \end{center}
 to obtain an exact sequence 
 \[ L_1A \mra L_1B \to L_1C \to L_0A \to L_0B \era L_0 C. \]
\end{remark}

\begin{proposition}\label{prop-ext-complete}
 \begin{itemize}
  \item[(a)] If we have a short exact sequence $A\to B\to C$ in which
   two of the three terms are Ext-$p$-complete, then so is the third.
  \item[(b)] Finite sums and retracts of Ext-$p$-complete groups are
   Ext-$p$-complete. 
  \item[(c)] The kernel, cokernel and image of any homomorphism
   between Ext-$p$-complete groups are Ext-$p$-complete.
  \item[(d)] The product of any (possibly infinite) family of
   Ext-$p$-complete groups is Ext-$p$-complete.
  \item[(e)] If $p^k.1_A=0$ for some $k$ then $A$ is
   Ext-$p$-complete. 
  \item[(f)] If $A$ is $p$-complete then it is Ext-$p$-complete.
 \end{itemize}
\end{proposition}
\begin{proof}\nl
 \begin{itemize}
  \item[(a)] Chase the diagram 
   \begin{center}
    \begin{tikzcd}
     &&&
     A \arrow[rightarrowtail,r] \arrow[d] &
     B \arrow[twoheadrightarrow,r] \arrow[d] &
     C \arrow[d] \\
     L_1A \arrow[rightarrowtail,r] &
     L_1B \arrow[r] &
     L_1C \arrow[r] &
     L_0A \arrow[r] &
     L_0B \arrow[twoheadrightarrow,r] &
     L_0C
    \end{tikzcd}
   \end{center}
  \item[(b)] Clear.
  \item[(c)] Consider a homomorphism $f\:A\to B$
   between Ext-$p$-complete groups, and the resulting short exact
   sequences $\img(f)\to B\to\cok(f)$ and $\ker(f)\to A\to\img(f)$.
   These give diagrams
   \begin{center}
    \begin{tikzcd}
     &&&
     \img(f) \arrow[rightarrowtail,r] \arrow[d] &
     B \arrow[twoheadrightarrow,r] \arrow[d,"\simeq"] &
     \cok(f) \arrow[d] \\
     L_1\img(f) \arrow[rightarrowtail,r] &
     0 \arrow[r]  &
     L_1\cok(f) \arrow[r] &
     L_0\img(f) \arrow[r] &
     L_0B \arrow[twoheadrightarrow,r] &
     L_0\cok(f)
    \end{tikzcd}
   \end{center}
   and
   \begin{center}
    \begin{tikzcd}
     &&&
     \ker(f) \arrow[rightarrowtail,r] \arrow[d] &
     A \arrow[twoheadrightarrow,r] \arrow[d,"\simeq"] &
     \img(f) \arrow[d] \\
     L_1\ker(f) \arrow[rightarrowtail,r] &
     0 \arrow[r] &
     L_1\img(f) \arrow[r] &
     L_0\ker(f) \arrow[r] &
     L_0A \arrow[twoheadrightarrow,r] &
     L_0\img(f).
    \end{tikzcd}
   \end{center}
   From the first diagram we see that $L_1\img(f)=0$ and that the map
   $\img(f)\to L_0\img(f)$ is injective, and from the second we see that
   the map $\img(f)\to L_0\img(f)$ is surjective; thus $\img(f)$ is
   Ext-$p$-complete.  Given this, it is a special case of~(a) that
   $\ker(f)$ and $\cok(f)$ are also Ext-$p$-complete as claimed.
  \item[(d)] This is clear, because
   $(\prod_iA_i)\psb{x}=\prod_iA_i\psb{x}$.
  \item[(e)] If $k=1$ the definitions give
   $L_0A=A\psb{x}/(x.A\psb{x})=A$ and
   $L_1A=\{v(x)\in A\psb{x}\st xv(x)=0\}=0$ as required.  The general
   case follows by induction using~(a) and the short exact sequence
   $pA\to A\to A/pA$.
  \item[(f)] As the tower $\{A/p^k\}$ consists of surjections, the
   $\invlim^1$ term is zero.  As $A$ is $p$-complete, we therefore
   have a short exact sequence
   \[ A \mra \prod_kA/p^k \era \prod_kA/p^k. \]
   The second and third terms are Ext-$p$-complete by parts~(e)
   and~(d), so $A$ is Ext-$p$-complete by part~(a).
 \end{itemize}
\end{proof}

\begin{proposition}\label{prop-Li}
 Let $A$ be any abelian group.  Then $L_1A=\invlim_kA[p^k]$ and there
 is a natural short exact sequence
 \begin{center}
  \begin{tikzcd}
   \invlim_k^1 A[p^k] \arrow[rightarrowtail,r,"\xi"] &
   L_0A \arrow[twoheadrightarrow,r,"\zt"] &
   A_p.
  \end{tikzcd}
 \end{center}
 The limit symbols here refer to the tower
 \[ 0=A[p^0] \xla{p} A[p] \xla{p} A[p^2] \xla{p} A[p^3] \xla{p} \dotsb \]
\end{proposition}
\begin{proof}
 First, an element of $L_1A$ is a series $v(x)=\sum_ia_ix^i$ with
 $(x-p)v(x)=0$, which means that $pa_0=0$ and $pa_{i+1}=a_i$ for all
 $i\geq 0$.  It follows inductively that $p^{i+1}a_i=0$ for all $i$,
 so the sequence $(0,a_0,a_1,\dotsc)$ is an element of
 $\invlim_iA[p^i]$.  All steps here can be reversed so
 $L_1A=\invlim_iA[p^i]$.  Next, define $\zt'_i\:A\psb{x}\to A/p^iA$ by 
 \[ \zt'_i(\sum_ja_jx^j) = \sum_{j<i}a_jp^j + p^iA. \]
 It is clear that $\zt'_i(v(x))=\zt'_{i+1}(v(x))\pmod{p^iA}$, so the
 maps $\zt'_i$ fit together to give a homomorphism
 $\zt'\:A\psb{x}\to A_p$, which can be described heuristically as
 $\zt'(v(x))=v(p)$.  It is also easy to check that $\zt'((x-p)w(x))=0$
 for all $w(x)$, so there is an induced map $\zt\:L_0A\to A_p$.  Given
 an arbitrary element $b\in A_p$ we can choose $b_i\in A$ representing
 the coset $\pi_i(b)\in A/p^iA$.  These will then satisfy
 $b_{i+1}=b_i\pmod{p^iA}$, so we can choose $a_i\in A$ with
 $b_{i+1}=p^ia_i+b_i$.  The series $v(x)=\sum_ia_ix^i$ then has
 $\zt(v(x))=b$, so we see that $\zt$ is surjective.

 Next, for any $c\in\prod_iA[p^i]$ put
 $\xi''(c)=\sum_ic_ix^i\in A\psb{x}$, and let $\xi'(c)$ denote the image
 of $\xi''(c)$ in $L_0A$.  It is clear by construction that $\zt\xi'=0$.
 Suppose that $c$ lies in the image of the map
 $D\:\prod_iA[p^i]\to\prod_iA[p^i]$, so there is a sequence
 $(d_i)_{i\geq 0}$ with $p^id_i=0$ and $c_i=d_i-pd_{i+1}$ for all
 $i$.  Note that $d_0=p^0d_0=0$ and put $w(x)=\sum_id_{i+1}x^i$; we
 find that $\xi''(c)=(x-p)w(x)$ and so $\xi'(c)=0$.  We thus have an
 induced map $\xi\:\cok(D)=\invlim^1_iA[p^i]\to L_0A$ with $\zt\xi=0$.
 Consider an arbitrary element $v(x)=\sum_ia_ix^i\in A\psb{x}$ with
 $\zt'(v(x))=0$.  This  means that for all $i\geq 0$ we can choose
 $b_i\in A$ with $\sum_{j<i}a_jp^j=p^ib_i$.  It follows that $b_0=0$
 and $p^ib_i+p^ia_i=p^{i+1}b_{i+1}$, so the element
 $c_i=b_i+a_i-pb_{i+1}$ has $p^ic_i=0$, so $c\in\prod_iA[p^i]$.  If we
 put $w(x)=\sum_ib_{i+1}x^i$ we find that 
 \[ v(x) = \xi''(c) - (x-p) w(x), \]
 so in $L_0A$ we have $v=\xi(c)$.  This proves that the map
 $\xi\:\invlim^1_iA[p^i]\to\ker(\zt)$ is surjective.  

 Finally, suppose we have $c\in\prod_iA[p^i]$ with $\xi'(c)=0$.  This
 means that there exists $u(x)=\sum_ib_ix^i\in A\psb{x}$ with
 $\xi''(c)=(x-p)u(x)$, so $c_0=-pb_0$ and $c_{i+1}=b_i-pb_{i+1}$ for
 all $i\geq 0$.  Form this it follows easily that $p^{i+1}b_i=0$ for
 all $i\geq 0$, so we have an element
 $b'=(0,b_0,b_1,\dotsc)\in\prod_iA[p^i]$.  We now see that $c=D(b')$,
 so $c$ represents the zero element of $\invlim^1_iA[p^i]$, thus, the
 map $\xi\:\invlim^1_iA[p^i]\to\ker(\zt)$ is injective.
\end{proof}

\begin{remark}
 It is clear from the definitions that the diagram
 \begin{center}
  \begin{tikzcd}
   & A \arrow[dl,"\eta"'] \arrow[dr,"\eta"] \\
   L_0A \arrow[rr,"\zt"'] & & A_p
  \end{tikzcd}
 \end{center}
 commutes.
\end{remark}

\begin{definition}\label{defn-bounded-torsion}
 We say that an abelian group $A$ has \emph{bounded $p$-torsion} if
 there exists $k\geq 0$ such that $p^k.\tors_p(A)=0$.
\end{definition}

\begin{corollary}\label{cor-Loi-free}
 If $A$ has bounded $p$-torsion (in particular, if $A$ is a free
 abelian group) then $L_1A=0$ and $L_0A=A_p$.
\end{corollary}
\begin{proof}
 The tower $\{A[p^i]\}$ is nilpotent, so
 $\invlim_iA[p^i]=\invlim_i^1A[p^i]=0$ by
 Proposition~\ref{prop-nilpotent-tower}. 
\end{proof}

\begin{corollary}\label{cor-Loi-derived}
 Suppose we have a short exact sequence $P\xra{f}Q\to A$ where $P$ and
 $Q$ are free abelian groups.  Then $L_0A$ and $L_1A$ are the cokernel
 and kernel of the induced map $P_p\to Q_p$.  
\end{corollary}
\begin{proof}
 This is immediate from Corollary~\ref{cor-Loi-free} and
 Remark~\ref{rem-Loi-six}. 
\end{proof}

\begin{proposition}\label{prop-functors-complete}
 For any abelian group $A$, the groups $L_0A$, $L_1A$, $A_p$ and
 $\invlim_k^1A[p^k]$ are all Ext-$p$-complete.
\end{proposition}
\begin{proof}
 The groups $A[p^k]$ and $A/p^k$ are Ext-$p$-complete by part~(e) of
 Proposition~\ref{prop-ext-complete}.  It follows by part~(d) that
 $\prod_kA[p^k]$ and $\prod_kA/p^k$ are Ext-$p$-complete, and then by
 part~(c) that the groups $\invlim_kA[p^k]=L_1A$, $\invlim^1_kA[p^k]$
 and $\invlim_kA/p^k=A_p$ are Ext-$p$-complete.  We can thus apply
 part~(a) to the short exact sequence 
 \begin{center}
  \begin{tikzcd}
   \invlim_k^1 A[p^k] \arrow[rightarrowtail,r,"\xi"] &
   L_0A \arrow[twoheadrightarrow,r,"\zt"] &
   A_p
  \end{tikzcd}
 \end{center}
 to deduce that $L_0A$ is Ext-$p$-complete.
\end{proof}

\begin{proposition}\label{prop-Lo-vanishing}
 For any abelian group $A$, we have $L_0A=0$ iff $A_p=0$ iff $A/p=0$.
\end{proposition}
\begin{proof}
 Proposition~\ref{prop-Li} shows that $A_p$ is a quotient of $L_0A$,
 and Proposition~\ref{prop-completion-idempotent} shows that $A/p$ is
 a quotient of $A_p$.  Conversely, if $A/p=0$ then $p.1_A$ is
 surjective, so $A/p^k=0$ for all $k$ and the maps in the tower
 $\{A[p^k]\}$ are all surjective.  It follows that
 $A_p=\invlim_kA/p^k=0$ and (using
 Proposition~\ref{prop-mittag-leffler}) that $\invlim^1_kA[p^k]=0$, so
 the short exact sequence in Proposition~\ref{prop-Li} shows that
 $L_0A=0$. 
\end{proof}

We next explain a more traditional construction of the functors $L_0$
and $L_1$.  This involves a group known as $\Zpi$.

\begin{definition}\label{defn-Zpi}
 Define $f_k\:\Z/p^k\to\Z/p^{k+1}$ by $f_k(a+p^k\Z)=pa+p^{k+1}\Z$, so
 we have a sequence
 \begin{center}
  \begin{tikzcd}
   0=
   \Z/p^0 \arrow[rightarrowtail,r,"f_0"] & 
   \Z/p   \arrow[rightarrowtail,r,"f_1"] &
   \Z/p^2 \arrow[rightarrowtail,r,"f_2"] &
   \Z/p^3 \arrow[rightarrowtail,r,"f_3"] &
   \Z/p^4 \arrow[rightarrowtail,r,"f_4"] &
   \dotsb 
  \end{tikzcd}
 \end{center}
 We define $\Zpi$ to be the colimit of this sequence.
\end{definition}

\begin{proposition}
 There are canonical isomorphisms
 \[ \Zpi = \Z[1/p]/\Z = \tors_p(\QZ) =  (\QZ)_{(p)} = \Q/\Zpl. \]
\end{proposition}
\begin{proof}
 First, consider the diagram 
 \begin{center}
  \begin{tikzcd}
   \Z \arrow[equal,r] \arrow[rightarrowtail,d,"1"'] & 
   \Z \arrow[equal,r] \arrow[rightarrowtail,d,"p"'] & 
   \Z \arrow[equal,r] \arrow[rightarrowtail,d,"p^2"'] & 
   \Z \arrow[equal,r] \arrow[rightarrowtail,d,"p^3"'] & 
   \Z \arrow[equal,r] \arrow[rightarrowtail,d,"p^4"'] &
   \dotsb \\
   \Z \arrow[rightarrowtail,r,"p"] \arrow[twoheadrightarrow,d] & 
   \Z \arrow[rightarrowtail,r,"p"] \arrow[twoheadrightarrow,d] & 
   \Z \arrow[rightarrowtail,r,"p"] \arrow[twoheadrightarrow,d] & 
   \Z \arrow[rightarrowtail,r,"p"] \arrow[twoheadrightarrow,d] & 
   \Z \arrow[rightarrowtail,r,"p"] \arrow[twoheadrightarrow,d] & 
   \dotsb \\
   \Z/p^0 \arrow[rightarrowtail,r,"f_0"'] &
   \Z/p^1 \arrow[rightarrowtail,r,"f_1"'] &
   \Z/p^2 \arrow[rightarrowtail,r,"f_2"'] &
   \Z/p^3 \arrow[rightarrowtail,r,"f_3"'] &
   \Z/p^4 \arrow[rightarrowtail,r,"f_4"'] &
   \dotsb
  \end{tikzcd}
 \end{center}
 Using Propositions~\ref{prop-colimit-exact}
 and~\ref{prop-rational-colimit} we obtain a short exact sequence
 $\Z\to\Z[1/p]\to\Z/p^{\infty}$, so $\Zpi=\Z[1/p]/\Z$.  Next,
 for $a\in\Q$ we note that $a+\Z$ is a $p$-torsion element in $\QZ$
 iff $p^ka\in\Z$ for some $k$, iff $a\in\Z[1/p]$.  It follows that
 $\tors_p(\QZ)=\Z[1/p]/\Z$.  We also know from
 Proposition~\ref{prop-local-torsion} that
 $\tors_p(\QZ)=(\QZ)_{(p)}$.  It is clear that $\Q$ is $p$-local, so
 $\Q_{(p)}=\Q$.  We can thus apply
 Proposition~\ref{prop-localisation-exact} to the sequence
 $\Z\mra\Q\era\QZ$ to see that $(\QZ)_{(p)}=\Q/\Zpl$.
\end{proof}

\begin{proposition}
 There are natural isomorphisms $L_0A=\Ext(\Zpi,A)$ and
 $L_1A=\Hom(\Zpi,A)$.  
\end{proposition}
\begin{proof}
 In this proof we will identify $\Zpi$ with $\Z[\pinv]/\Z$.  Put
 $F=\bigoplus_{i=0}^\infty\Z$, and let $e_i$ be the $i$'th basis 
 vector in $F$.  Define maps 
 \[ F\xra{\phi}F\xra{\psi} \Zpi \]
 by 
 \[ \phi(e_i)=e_{i-1}-pe_i \hspace{6em} \psi(e_i)=p^{-i-1}+\Z \]
 (where $e_{-1}$ is interpreted as $0$), or equivalently
 \begin{align*}
  \phi(n_0,n_1,n_2,\dotsc) &= 
     (pn_0-n_1,pn_1-n_2,pn_2-n_3,\dotsc) \\
  \phi(m_0,m_1,m_2,\dotsc) &= \sum_im_ip^{-i-1}+\Z.
 \end{align*}
 By considering the first nonzero entry in $n=(n_0,n_1,\dotsc)$, we
 see that $\phi$ is injective.  Any element of $\Zpi$ can be written
 as $k/p^{i+1}+\Z$ for some $i\geq 0$ and $k\in\Z$, and this is the
 same as $\psi(ke_i)$, so $\psi$ is surjective.  It is clear from the
 definitions that $\psi\phi=0$, so $\img(\phi)\leq\ker(\psi)$.
 Conversely, suppose we have $n\in F$ with $\psi(n)=0$, so the number
 $q=\sum_in_ip^{-1-i}$ actually lies in $\Z$.  Put
 $m_i=\sum_{j<i}n_jp^{i-j-1}-p^iq$, and note that $m_i=0$ for
 $i\gg 0$, so $m\in F$.  We find that $\phi(m)=n$, so the sequence
 $F\xra{\phi}F\xra{\psi}\Zpi$ is short exact.  As $F$ is free, this
 gives us an exact sequence
 \begin{center}
  \begin{tikzcd}
   \Hom(\Zpi,A) \arrow[rightarrowtail,r,"\psi^*"] &
   \Hom(F,A) \arrow[r,"\phi^*"] &
   \Hom(F,A) \arrow[twoheadrightarrow,r] &
   \Ext(\Zpi,A).
  \end{tikzcd}
 \end{center}  
 Next, for any series $v(x)=\sum_ia_ix^i\in A\psb{x}$ we have a
 homomorphism $\al(v(x))\:F\to A$ given by $\al(v(x))(e_i)=a_i$ for
 all $i\geq 0$.  This construction gives an isomorphism
 $A\psb{x}\to\Hom(F,A)$.  If we make the conventions $a_{-1}=0$ and
 $e_{-1}=0$ we also have 
 \[ \al((x-p)v(x))(e_j) = 
      \al\left(\sum_i(a_{i-1}-pa_i)x^i\right)(e_j) =
      a_{j-1}-pa_j = \al(v(x))(e_{j-1}-pe_j) = 
      \al(v(x))(\phi(e_j))
 \]
 Thus, multiplication by $x-p$ on $A\psb{x}$ corresponds to $\phi^*$
 on $\Hom(F,A)$, and the claim follows from this.
\end{proof}

\end{document}